\input ./style/arxiv-general.cfg
\documentclass[aap,MSNbibl,seceqn,rotating,dvips]{arximspdf}
\makeatletter
   \@ifpackageloaded{graphicx}{}{\usepackage{graphicx}}
\makeatother

\doi{10.1214/14-AAP1064}
\volume{25}
\issue{5}
\pubyear{2015}
\firstpage{2909}
\lastpage{2958}
\docsubty{FLA}

\makeatletter
\newcommand{\rrvert}{\vert}
\newcommand{\rrVert}{\Vert}
\newcommand{\llvert}{\vert}
\newcommand{\llVert}{\Vert}
\newtheorem{proposition}{Proposition}[section]
\newtheorem{theorem}[proposition]{Theorem}
\newtheorem{lemma}[proposition]{Lemma}
\newproclaim{definition}[proposition]{Definition}
\newproclaim{remark}[proposition]{Remark}
\newtheorem{condition}[proposition]{Condition}
\newtheorem{lem}[proposition]{Lemma}
\newproclaim{rem}[proposition]{Remark}
\makeatother

\begin{document}
\begin{frontmatter}

\title{Escaping from an attractor: Importance sampling and rest points I}
\runtitle{Escaping from an attractor}

\begin{aug}
\author[A]{\fnms{Paul} \snm{Dupuis}\thanksref{T1}\ead[label=e1]{dupuis@dam.brown.edu}},
\author[B]{\fnms{Konstantinos} \snm{Spiliopoulos}\corref{}\thanksref{T2}\ead[label=e2]{kspiliop@math.bu.edu}}
\and
\author[C]{\fnms{Xiang} \snm{Zhou}\thanksref{T3}\ead[label=e3]{xiang.zhou@cityu.edu.hk}}
\runauthor{P. Dupuis, K. Spiliopoulos and X. Zhou}
\affiliation{Brown University, Boston University and City University of Hong Kong}
\address[A]{P. Dupuis\\
Lefschetz Center for Dynamical
Systems\\
Division of Applied Mathematics\\
Brown University\\
Providence, Rhode Island 02912\\
USA\\
\printead{e1}}
\address[B]{K. Spiliopoulos\\
Department of Mathematics \\
\quad and Statistics\\
Boston University\\
Boston, Massachusetts 02215\\
USA\\
\printead{e2}}
\address[C]{X. Zhou\\
Department of Mathematics\\
City University of Hong Kong\\
Kowloon Tong\\
Hong Kong, P6715\\
\printead{e3}}
\end{aug}
\thankstext{T1}{Supported in part by NSF Grants DMS-10-08331, DMS-13-17199, Department of Energy Grant
DE-SCOO02413 and the Air
Force Office of Scientific Research Grant FA9550-12-1-0399.}
\thankstext{T2}{Supported in part by a start-up fund from
Boston University, Department of Energy Grant \mbox{DE-}SCOO02413 and, during
revisions of this article, by NSF Grant DMS-13-12124.}
\thankstext{T3}{Supported in part by a start-up fund from City
University of Hong Kong and Department of Energy Grant DE-SCOO02413.}

%
\received{\smonth{3} \syear{2013}}
\revised{\smonth{6} \syear{2014}}

\begin{abstract}
We discuss importance sampling schemes for the estimation of finite time exit
probabilities of small noise diffusions that involve escape from an
equilibrium. A factor that complicates the analysis is that rest points are
included in the domain of interest. We build importance sampling schemes with
provably good performance both pre-asymptotically, that is, for fixed size of the
noise, and asymptotically, that is, as the size of the noise goes to zero, and
that do not degrade as the time horizon gets large. Simulation studies
demonstrate the theoretical results.
\end{abstract}

\end{frontmatter}

\section{Introduction}
\label{S:Introduction}

This paper considers the use of importance sampling for estimating hitting or
exit probabilities for stochastic processes. The process model is a
$d$-dimensional diffusion $X^{\varepsilon}\doteq\{X^{\varepsilon}(s),s\in [0,\infty)\}$ satisfying the stochastic differential equation
(SDE)
\begin{equation}\label{Eq:Diffusion1}
dX^{\varepsilon}(s)=b \bigl( X^{\varepsilon}(s) \bigr) \,ds+\sqrt{\varepsilon }
\sigma \bigl( X^{\varepsilon}(s) \bigr) \,dB(s), \qquad X^{\varepsilon
}(0)=x,
\end{equation}
where $\varepsilon>0$ and $B(s)$ is a standard $d$-dimensional Wiener process.
Of particular interest is the case of gradient flows, $b(x)=-DV(x)$, and
constant diffusion coefficient, though many aspects of the analysis are more
generally applicable. Let $\mathcal{D}\subset\mathbb{R}^{d}$ be an open set,
and denote by $\tau^{\varepsilon}$ the exit time of $X^{\varepsilon}(s)$ from
$\mathcal{D}$. We are concerned with the estimation of quantities such as the
probability that $X^{\varepsilon}$ leaves $\mathcal{D}$ before some time
$T\in(0,\infty)$, or that it exits through a particular subset $O\subset
\mathcal{D}$ before $T$, and related expected values. The principal novel
feature of this work is that the initial point is in the neighborhood of an
equilibrium point of the noiseless dynamics.

The estimation of such probabilities has several mathematical and
computational difficulties. It is related to the estimation of transition
probabilities between different metastable states within a given time horizon.
As is well known, standard Monte Carlo sampling techniques lead to
exponentially large relative errors as the noise coefficient $\varepsilon$
tends to zero. When rest points are in the domain of interest, the situation is
even more complicated than usual. This work will focus on this particularly
difficult issue.

The performance of unbiased estimators for rare event problems is usually
measured by the size of the second moment of the estimator based on a single
simulation. For a well-designed scheme the ratio of this second moment to the
quantity of interest will not grow too rapidly as $\varepsilon\downarrow0$.
One measure is the exponential rate of decay of the second moment. If this
rate of decay is exactly twice the decay rate for the probability of interest,
then the scheme is called asymptotically efficient (or weakly efficient). The
notion of \textit{strong} efficiency requires that the ratio of the second
moment to the square of the probability be bounded above uniformly for all
small $\varepsilon>0$. While such performance is certainly desirable, it is
not common when dealing with models such as (\ref{Eq:Diffusion1}) that involve
state dependent dynamics and complicated geometries. As we describe below, in
some sense both these measures are inadequate for the situation considered here.

A theory based on subsolutions to an associated Hamilton--Jacobi--Bellman (HJB)
equation has been developed for the design and performance analysis of
importance sampling; see, for example, \cite{dupwan3,dupwan5,dupwan6,dupspiwan}.
In this approach the change of measure (which, for reasons made evident
later, we will call the control) used in the importance sampling is defined
in terms of the gradient of a subsolution, and the~performance, as measured by
the decay rate of the second moment, is given by the value of the subsolution
at the initial location $x$. This theory is the starting point of our analysis
of (\ref{Eq:Diffusion1}), though as mentioned previously, the inclusion of rest
points will motivate some further developments.

For any particular class of process models and events, an essential step in
the application of this approach is the construction of appropriate
subsolutions. In this paper we will exploit the fact that the
Freidlin--Wentsell quasipotential \cite{frewen} can be used to construct
various subsolutions for certain \textit{time independent} problems related to
(\ref{Eq:Diffusion1}). In addition, for particular but important classes of
process models (e.g., gradient systems with constant diffusion matrix), the
quasipotential and hence these subsolutions take explicit and simple forms. As
we will discuss in detail, these subsolutions also give subsolutions for
the \textit{time dependent} problems, and when $T$ is large the value of the
subsolution at the starting point (which now includes time $t=0$ as well as
the location) will be close to the maximal value.

It follows that if the final time $T$ is large enough, then existing theory
implies that the estimator based on this subsolution should have a nearly
optimal decay rate for its second moment. While this is a valid statement,
there is an important qualitative difference between problems which include a
rest point in the domain of interest and those which do not. The distinction
is not on the decay rate, which behaves as expected in both situations, but
rather depends on the pre-exponential terms not captured by the decay rate.
When the domain does not contain a rest point, one has simultaneously good
rates of decay and control over the pre-exponential terms. However, when a
rest point is present schemes based only on this time independent subsolution
keep the desired decay rate but lose the good control over the pre-exponential
terms. The qualitative difference is related to the fact that in the former
case a subsolution designed on the basis of the $\varepsilon=0$ problem can be
shown to give useful bounds for the problem with $\varepsilon>0$, but in the
latter case this is no longer true. This qualitative distinction will be made
precise when we construct nonasymptotic bounds on the second moment
for the two cases. When $\varepsilon>0$ is small but not too small, the loss
of performance due to the large pre-exponential term can be significant,
rendering the associated importance sampling scheme little better than
ordinary Monte Carlo. As $\varepsilon\downarrow0$ the exponential decay rate
dominates, and importance sampling once again gives much greater performance
than ordinary Monte Carlo. However, the improvement is less than in the case
where rest points are not included, and an approach which avoids this loss of
performance would certainly be welcome.

In this paper we overcome this difficulty by constructing time dependent subsolutions that approximate
the zero-variance change of measure. The approach that we follow is to combine an explicit
solution to an approximating time dependent problem in a neighborhood of the
rest point with the time independent subsolution obtained via the
quasipotential away from the rest point. As we will show, such an approach
will maintain the high decay rate while at the same time properly controlling
the pre-exponential term. In the neighborhood of the rest point, one can
approximate the dynamics of the diffusion process by a Gauss--Markov process,
that is, a process with a constant diffusion matrix and drift that is affine in
the state. For these dynamics and appropriate terminal conditions for the
localizing problem, the solution to the related PDE can be constructed in
terms of the famous linear/quadratic regulator problem from optimal control
theory. As a consequence an explicit and nearly optimal scheme for a surrogate
problem can be identified in the neighborhood of the rest point, which is then
merged with the explicit scheme based on the quasipotential in that part of
the domain where it is particularly effective.

In this paper we analyze the difficulties caused by the presence of rest
points in a general setting. We describe and theoretically justify a
resolution of these difficulties in the case of dimension one, and present
computational data for this case. The construction of the localizing problem
is more elaborate in dimension greater than one, and will be presented in a
companion paper along with the results of numerical experiments in higher dimensions.

The contents of this paper are as follows. In Section~\ref{S:LDPandIS} we
review the relevant large deviation theory and importance sampling. In Section~\ref{S:EffectivenessOfSubsolutions} we discuss the effectiveness of time
independent subsolutions. In particular, we show that if rest points are not
part of the domain of interest, then subsolutions lead to both good decay rates
and nonasymptotic bounds for the second moment of the corresponding unbiased
estimator. However, when rest points are included in the domain of interest,
the situation is more complicated, and even if the decay rate is good, the
prelimit bounds may not be as good as desired. In Section~\ref{S:LinearProblem}, we present a change of measure for the problem with a
rest point for a quadratic potential function with provably good
pre-asymptotic and asymptotic performance and which does not degrade as $T$
gets larger. In Section~\ref{S:NonLinearProblem} we extend the discussion to
the nonlinear problem with rest points. Simulation data, demonstrating the
discussions in Sections~\ref{S:LinearProblem} and \ref{S:NonLinearProblem} are
also presented in the corresponding sections. The \hyperref[app]{Appendix} has some auxiliary
lemmas that are used in the main body of the manuscript.

\section{Related large deviation and importance sampling results}
\label{S:LDPandIS}

In this section we recall well-known large deviation results for probabilities
of exit times (Section~\ref{SS:LDP}), review importance sampling in the
context of small noise diffusions (Section~\ref{SS:IS}) and also recall
the notion of subsolutions to certain related HJB equations (Section~\ref{SS:RoleOfSubsolutions}).

In most of this paper, the following assumptions will be used: the assumptions
are stronger than necessary, but simplify the discussion considerably. For
example, the nondegeneracy of the diffusion matrix and regularity of the
boundary of $\mathcal{D}$ easily imply that a limit exists for
(\ref{Eq:EstimationTarget2}). They can be weakened, but the existence of the
limit then requires conditions that are best addressed in a problem dependent fashion.

\begin{condition}
\label{A:MainAssumption}
\textup{(i)} The drift\vspace*{-6pt} $b$ is bounded and Lipschitz continuous.
\begin{longlist}[(iii)]
\item[(ii)] {The coefficient $\sigma$ is bounded, Lipschitz continuous and uniformly
nondegenerate.}

\item[(iii)] $\mathcal{D}$ is an open and bounded subset of $\mathbb{R}^{d}$, and at
all points on its boundary~$\mathcal{D}$ satisfies an interior and exterior
cone condition; that is, there is $\delta>0$ such that if $x\in\partial
\mathcal{D}$, then there exist unit vectors $v_{1},v_{2}\in\mathbb{R}^{d}$ such that
\[
\bigl\{ y \dvtx  \llVert y-x\rrVert <\delta\mbox{ and }\bigl\llvert \langle
y-x,v_{1} \rangle \bigr\rrvert <\delta\llVert y-x\rrVert \bigr\}
\subset\mathcal{D}%
\]
and
\[
\bigl\{ y\dvtx \llVert y-x\rrVert <\delta\mbox{ and }\bigl\llvert \langle
y-x,v_{2} \rangle \bigr\rrvert <\delta\llVert y-x\rrVert \bigr\} \cap
\mathcal{D}=\varnothing.
\]
We also assume that if $x\in\partial\mathcal{D}$ and if $\phi$ solves
$\dot{\phi}=b(\phi),\phi(0)=x$, then $\phi(t)\in\mathcal{D}$ for all
$t\in(0,\infty)$.
\end{longlist}
\end{condition}

\subsection{Large deviation results}
\label{SS:LDP}

Fix $T\in(0,\infty)$, and consider an initial point $(t,x)\in [
0,T)\times\mathcal{D}$. Consider a bounded and class $\mathcal{C}^{2}$
function $h\dvtx \mathbb{R}^{d}\rightarrow\mathbb{R}$. Let $\mathbb{E}_{t,x}$ denote
expected value given $X^{\varepsilon}(t)=x$, and define
\begin{equation}\label{Eq:EstimationTargetSmooth1}%
\theta^{\varepsilon}(t,x)\doteq\mathbb{E}_{t,x} \bigl[
e^{-({1}/{\varepsilon})h (  X^{\varepsilon}(\tau^{\varepsilon}) )
}1_{ \{  \tau^{\varepsilon}\leq T \}  } \bigr].
\end{equation}
Since $\theta^{\varepsilon}$ scales exponentially, it is also useful to define
\begin{equation}\label{Eq:EstimationTarget2}%
G^{\varepsilon}(t,x)\doteq-\varepsilon\log\theta^{\varepsilon}(t,x).
\end{equation}
Although for now we focus on the case where $h$ is bounded and continuous, we are also interested in cases where $h$ is discontinuous and takes the value
$\infty$. An example is when for some set $O\subset\partial\mathcal{D}$,
$h(x)=\infty$ if $x\notin O$ and $h(x)=0$ if $x\in O$. In this case
$\theta^{\varepsilon}(t,x)$ equals the probability of exiting $\mathcal{D}$
through $O$ before time $T$. For these cases and under mild regularity
conditions on $O$, statements analogous to Theorem~\ref{T:LDPTheorem} below hold.

Let $\mathcal{AC} ( [t,T]\dvtx \mathbb{R}^{d}) $ be the set of
absolutely continuous functions on $[t,T]$ with values in $\mathbb{R}^{d}$. We
denote the local rate function by
\[
L(x,v)\doteq\tfrac{1}{2} \bigl\langle v-b(x),a^{-1}(x) \bigl[
v-b(x) \bigr] \bigr\rangle, 
\]
where $a(x)=\sigma(x)\sigma^{T}(x)$, and the corresponding rate or action
functional for $\phi\in\mathcal{AC} (  [t,T]\dvtx \mathbb{R}^{d} )$
by
\[
I_{tT}(\phi)\doteq\int_{t}^{T}L\bigl(
\phi(s),\dot{\phi}(s)\bigr)\,ds. 
\]
For all other $\phi\in\mathcal{C} (  [t,T]\dvtx \mathbb{R}^{d} )$, set
$I_{tT}(\phi)=\infty$. The following large deviations result is well known; see, for example,
\cite{fleson1,frewen}.

\begin{theorem}
\label{T:LDPTheorem}
Assume Condition \ref{A:MainAssumption}. Then for each
$(t,x)\in [0,T)\times\mathcal{D}$
\[
\lim_{\varepsilon\downarrow0}G^{\varepsilon}(t,x)=G(t,x)\doteq\inf
_{\phi
\in\Lambda(t,x)} \bigl[ I_{tT}(\phi)+h\bigl(\phi(T)\bigr)
\bigr], 
\]
where
\[
\Lambda(t,x)= \bigl\{ \phi\in\mathcal{C}\bigl([t,T]\dvtx \mathbb{R}^{d}
\bigr)\dvtx \phi (t)=x,\phi(s)\in\mathcal{D}\mbox{ for }s\in [ t,T],\phi(T)\in
\partial\mathcal{D} \bigr\}.
\]
\end{theorem}

\subsection{Preliminaries on importance sampling}
\label{SS:IS}

We briefly review the use of importance sampling for estimating
$\theta^{\varepsilon}(t,x)$ for a given function $h$. Let $\Gamma^{\varepsilon}(t,x)$ be any unbiased estimator of $\theta^{\varepsilon}(t,x)$
that is defined on some probability space with probability measure
$\bar{\mathbb{P}}$. Thus $\Gamma^{\varepsilon}(t,x)$ is a random variable such
that
\[
\bar{\mathbb{E}}\Gamma^{\varepsilon}(t,x)=\theta^{\varepsilon}(t,x),
\]
where $\bar{\mathbb{E}}$ is the expectation operator associated with
$\bar{\mathbb{P}}$. In this paper we will consider only unbiased estimators.

In Monte Carlo simulation, one generates a number of independent copies of
$\Gamma^{\varepsilon}(t,x)$, and the estimate is the sample mean. The specific
number of samples required depends on the desired accuracy, which is measured
by the variance of the sample mean. However, since the samples are independent,
it suffices to consider the variance of a single sample. Because of
unbiasedness, minimizing the variance is equivalent to minimizing the second
moment. By Jensen's inequality,
\[
\bar{\mathbb{E}}\bigl(\Gamma^{\varepsilon}(t,x)\bigr)^{2}\geq\bigl(
\bar{\mathbb{E}}%
\Gamma^{\varepsilon}(t,x)\bigr)^{2}=
\theta^{\varepsilon}(t,x)^{2}.
\]
It then follows from Theorem \ref{T:LDPTheorem} that
\[
\limsup_{\varepsilon\rightarrow0}-\varepsilon\log\bar{\mathbb{E}}%
\bigl(
\Gamma^{\varepsilon}(t,x)\bigr)^{2}\leq2G(t,x),
\]
and thus $2G(t,x)$ is the best possible rate of decay of the second moment.
If
\[
\liminf_{\varepsilon\rightarrow0}-\varepsilon\log\bar{\mathbb{E}}%
\bigl(
\Gamma^{\varepsilon}(t,x)\bigr)^{2}\geq2G(t,x),
\]
then $\Gamma^{\varepsilon}(t,x)$ achieves this best decay rate and is said to
be \textit{asymptotically optimal}. While asymptotic optimality or near-asymptotic optimality is desirable,
as noted in the \hyperref[S:Introduction]{Introduction} one may also
desire good behavior of the pre-exponential term. To keep the terminology
clear we will avoid the conventional usage of terms such as asymptotic
optimality, and refer instead to properties of the ``decay
rate'' and the ``pre-exponential
term.''

The unbiased estimators $\Gamma^{\varepsilon}(t,x)$ that we consider are all
based on measure transformation. Consider $u^{\varepsilon}(s)$, a sufficiently
integrable and adapted function, such that
\[
\frac{d\bar{\mathbb{P}}^{\varepsilon}}{d\mathbb{P}}=\exp \biggl\{ -\frac
{1}{2\varepsilon}\int_{t}^{T}
\bigl\llVert u^{\varepsilon}(s)\bigr\rrVert^{2}\,ds+\frac{1}{\sqrt{\varepsilon}}\int
_{t}^{T} \bigl\langle u^{\varepsilon
}(s),dB(s)
\bigr\rangle \biggr\}
\]
defines a family of probability measures $\bar{\mathbb{P}}^{\varepsilon}$.
Then by Girsanov's theorem, for each $\varepsilon>0$,
\[
\bar{B}(s)=B(s)-\frac{1}{\sqrt{\varepsilon}}\int_{t}^{s}u^{\varepsilon}%
(
\rho)\,d\rho,\qquad t\leq s\leq T
\]
is a Brownian motion on $[t,T]$ under the probability measure $\bar
{\mathbb{P}}^{\varepsilon}$, and $X^{\varepsilon}$ satisfies $X^{\varepsilon
}(t)=x$ and
\[
dX^{\varepsilon}(s)=b \bigl( X^{\varepsilon}(s) \bigr) \,ds+\sigma \bigl(
X^{\varepsilon}(s) \bigr) \bigl[ \sqrt{\varepsilon}\,d\bar{B}%
(s)+u^{\varepsilon}(s)\,ds
\bigr].
\]
For our purposes, $u^{\varepsilon}(s)$ is either given as a process that is
progressively measurable with respect to a suitable filtration that measures
the Wiener process (sometimes called an \textit{open loop} control), or else
it is of \textit{feedback} form, in which case there is a suitably measurable
function $\bar{u}^{\varepsilon}\dvtx [0,T]\times\mathbb{R}^{d}\rightarrow
\mathbb{R}^{d}$ such that $u^{\varepsilon}(s)=\bar{u}^{\varepsilon
}(s,X^{\varepsilon}(s))$. Of course when implementing importance sampling we
consider only the latter form. Letting
\[
\Gamma^{\varepsilon}(t,x)=\exp \biggl\{ -\frac{1}{\varepsilon}h \bigl(
X^{\varepsilon}\bigl(\tau^{\varepsilon}\bigr) \bigr) \biggr\} 1_{ \{
\tau^{\varepsilon}\leq T \}  }\,
\frac{d\mathbb{P}}{d\bar{\mathbb{P}%
}^{\varepsilon}}\bigl(X^{\varepsilon}\bigr),
\]
it follows easily that under $\bar{\mathbb{P}}^{\varepsilon}$, $\Gamma
^{\varepsilon}(t,x)$ is an unbiased estimator for $\theta^{\varepsilon}(t,x)$.
The performance of this estimator is characterized by its second moment,
\begin{equation}
\label{eqn:defofQ}
Q^{\varepsilon}\bigl(t,x;\bar{u}^{\varepsilon}\bigr)\doteq\bar{
\mathbb{E}}^{\varepsilon
} \biggl[ \exp \biggl\{ -\frac{2}{\varepsilon}h \bigl(
X^{\varepsilon}%
\bigl(\tau^{\varepsilon}\bigr) \bigr) \biggr\}
1_{ \{  \tau^{\varepsilon}\leq
T \}  } \biggl( \frac{d\mathbb{P}}{d\bar{\mathbb{P}}^{\varepsilon}%
}\bigl(X^{\varepsilon}\bigr) \biggr)
^{2} \biggr].
\end{equation}
The goal of this paper is to investigate the effect of rest points on
$Q^{\varepsilon}(t,x;\bar{u}^{\varepsilon})$ and how one can choose controls
that guarantee both good decay rates and pre-exponential bounds for
$Q^{\varepsilon}(t,x;\bar{u}^{\varepsilon})$.

We conclude this section with a review of subsolutions to related HJB
equations. Such subsolutions are essential for constructing and analyzing good
important sampling schemes.

\subsection{Subsolutions to a related PDE}
\label{SS:RoleOfSubsolutions}

Let
\[
\mathbb{H}(x,p)= \bigl\langle b(x),p \bigr\rangle -\tfrac{1}{2}\bigl\llVert
\sigma^{T}(x)p\bigr\rrVert ^{2}.
\]
The construction of good importance sampling schemes for a quantity such as
(\ref{Eq:EstimationTargetSmooth1}) is closely related to the HJB equation
\begin{eqnarray}\label{Eq:HJBequation1}
U_{t}(t,x)+\mathbb{H}\bigl(x,DU(t,x)\bigr)&=& 0\qquad\mbox{for }(t,x)
\in [0,T)\times \mathcal{D},
\\
U(t,x) &=& h(x)\qquad\mbox{for }t\leq T,x\in\partial\mathcal{D},
\nonumber
\\[-8pt]
\label{Eq:HJBterm_cond}
\\[-8pt]
\nonumber
U(T,x) &=& \infty \qquad \mbox{for }x\in\mathcal{D},
\end{eqnarray}
and more precisely to its subsolutions. It can be shown that $G$ defined in
Theorem~\ref{T:LDPTheorem} is the unique continuous viscosity solution of
(\ref{Eq:HJBequation1}) and (\ref{Eq:HJBterm_cond}); see \cite{fleson1}.

\begin{definition}
\label{Def:ClassicalSubsolution}
A function $\bar{U}(t,x)\dvtx [0,T]\times
\mathbb{R}^{d}\rightarrow\mathbb{R}$ is a classical subsolution to the HJB
equation (\ref{Eq:HJBequation1}) and (\ref{Eq:HJBterm_cond}) if:
\begin{longlist}[(iii)]
\item[(i)] $\bar{U}$ is continuously differentiable;

\item[(ii)] $\bar{U}_{t}(t,x)+\mathbb{H}(x,D\bar{U}(t,x))\geq0$ for\vspace*{1pt} every
$(t,x)\in [0,T)\times\mathcal{D}$;

\item[(iii)] $\bar{U}(t,x)\leq h(x)$ for $t\leq T,x\in\partial\mathcal{D}$ and
$\bar{U}(T,x)\leq\infty$ for $x\in\mathcal{D}$.
\end{longlist}
\end{definition}

The connection between subsolutions and the performance of importance sampling
schemes has been established in several papers, such as
\cite{dupwan5,dupspiwan}. These papers either consider classical subsolutions
or, more generally, piecewise classical subsolutions. To simplify the
discussion, we consider here just classical subsolutions. In the present
setting, we have the following theorem regarding asymptotic optimality
(Theorem 4.1 in \cite{dupspiwan}).

\begin{theorem}
\label{T:UniformlyLogEfficientRegime1}
Let $\{X^{\varepsilon},\varepsilon>0\}$
be the unique strong solution to (\ref{Eq:Diffusion1}). Consider a bounded and
continuous function $h\dvtx \mathbb{R}^{d}\rightarrow\mathbb{R}$, and assume Condition
\ref{A:MainAssumption}.
Let $\bar{U}(t,x)$ be a subsolution according to
Definition \ref{Def:ClassicalSubsolution}, and define the control
$u^{\varepsilon}(s)=-\sigma^{T}(X^{\varepsilon}(s))D\bar{U}(s,X^{\varepsilon
}(s))$. Then
\[
\liminf_{\varepsilon\rightarrow0}-\varepsilon\log Q^{\varepsilon
}
\bigl(t,x;u^{\varepsilon}\bigr)\geq G(t,x)+\bar{U}(t,x). \label{Eq:GoalRegime1Subsolution}%
\]
\end{theorem}

Since $\bar{U}$ is a subsolution, it is automatic that $G(s,y)\geq\bar{U}(s,y)$
for all $(s,y)\in [0,T]\times\mathcal{D}$. If $G(t,x)=\bar{U}(t,x)$, then
the scheme has the largest possible decay rate.

\section{Qualitative properties of schemes based on subsolutions}
\label{S:EffectivenessOfSubsolutions}
In this section we justify some of the
claims made in the \hyperref[S:Introduction]{Introduction} regarding the differences in performance
between importance sampling schemes when rest points are included in the
domain of interest and when they are not. We consider just the problem of
estimating the probability of escape from a set before time $T$, and even then
consider a particular setup. However, the example will illustrate the
difference between the two cases, and also suggest how one might improve the
performance when rest points are involved.

\begin{rem}
Much of the prior application of subsolutions to importance sampling
\cite{dupledwan3,dupledwan4,dupsezwan,dupwan6} has involved the estimation of
escape probabilities for classes of stochastic networks, in which case the
origin is often the unique stable point for the law of large numbers dynamics
[the analogue of (\ref{Eq:Diffusion1}) with $\varepsilon=0$]. The event
most often studied in this context is that of escape from a set (i.e., buffer
overflow) before reaching the origin, after starting near, but not at the
origin. The analogous event for the diffusion model (\ref{Eq:Diffusion1}) is
one of the problems that are the focus of the present work. However, the
difficulties that will be described momentarily for the diffusion model do not
arise when dealing with the analogous estimation problem for stochastic
networks, and indeed in that setting the proximity of the rest point has
little impact on either the rate of decay or the pre-exponential term. This is
related to the fact that the law of large numbers trajectories for stochastic
networks reach the origin in finite time, as opposed to the infinite time it
takes for the solution to (\ref{Eq:Diffusion1}) with $\varepsilon=0$ to
reach a stable equilibrium point when not starting at such a point. In turn,
this property is responsible for the fact that minimizing trajectories in the
definition of the quasipotential are achieved on bounded time intervals for
stochastic network models, but take infinite time for processes such as
(\ref{Eq:Diffusion1}).
\end{rem}

For the remainder of this section we concentrate on the special case of
$b(x)=-DV(x)$ and $\sigma(x)=I$, and on a particular estimation problem. We
first argue that if the domain of interest does not include a rest point, then
given a time-independent subsolution and associated control, not only is a good
decay rate obtained, but good bounds on the pre-exponential terms hold as
well. We then show why this is not possible when a rest point is included.

Assume that $x=0$ is the global minimum for $V(x)$, so that $DV(0)=0$, and
that $DV(x)\neq0$ for all $x\neq0$. Without loss we assume that $V(0)=0$. Let
$0<\ell<L$, and define $\mathcal{D}\doteq\{x\in\mathbb{R}^{d}\dvtx \ell<V(x)<L\}$ and
$A_{c}\doteq\{x\in\mathbb{R}^{d}\dvtx V(x)=c\}$. Then the problem is to estimate
\[
\theta^{\varepsilon}(t,y)\doteq\mathbb{P}_{t,y} \bigl\{
X^{\varepsilon}\mbox{ hits }A_{L}\mbox{ before hitting
}A_{\ell}\mbox{ and before time }T \bigr\},
\]
where the initial point $y$ is such that $\ell<V(y)<L$. This corresponds to
(\ref{Eq:EstimationTargetSmooth1}), but here $h$ is not bounded and smooth,
and instead $h(x)=0$ if $x\in A_{L}$ and $h(x)=\infty$ if $x\in A_{\ell}$. For
this problem one can also identify the rate of decay $G(t,y)$ via
(\ref{Eq:EstimationTarget2}). A one-dimensional example is illustrated in Figure~\ref{fig:no_rest}.


\begin{figure}[t]

\includegraphics{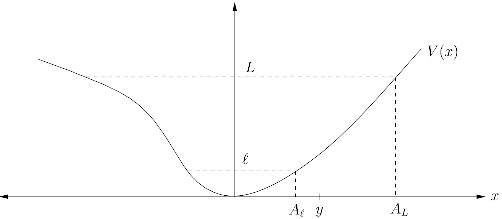}

\caption{Escape problem with no rest point.}
\label{fig:no_rest}
\end{figure}

The quasipotential with respect to the equilibrium point $0$ is defined by
\[
S(0,x)\doteq\inf\bigl\{I_{0T}(\phi)\dvtx \phi\in\mathcal{C}\bigl([0,T]\dvtx
\mathbb{R}^{d}\bigr),\phi(0)=0,\phi(T)=x,T\in(0,\infty)\bigr\}. 
\]
It follows from the variational characterization of $S$ that $x\rightarrow
S(0,x)$ is always a weak sense solution to $\mathbb{H}(x,-DS(0,x) )  =0$, and therefore by adding an appropriate constant $C$ to
satisfy any needed boundary and terminal conditions, $-S(0,x)+C$ will always
define a weak sense subsolution. In the present case $S(0,x)$ takes the
explicit form (Theorem 4.3.1 in \cite{frewen})
\[
S(0,x)=2 \bigl( V(x)-V(0) \bigr) =2V(x),
\]
and it is easy to check that $U(x)=-2(V(x)-L)$ is a subsolution according to
Definition \ref{Def:ClassicalSubsolution}. Indeed, $U(x)=0$ for $x\in A_{L}$,
while the boundary condition $U(x)\leq\infty$ for $x\in A_{\ell}$ and terminal
condition $U(x)\leq\infty$ for $x\in\mathcal{D}$ hold vacuously; see Figure~\ref{fig:sub_no_rest}. The control (i.e., change of measure) suggested by this
subsolution for the importance sampling scheme is $\bar{u}(x)=2DV(x)$.%
\begin{figure}

\includegraphics{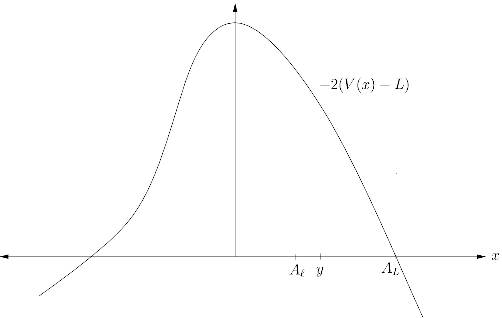}

\caption{Subsolution for both cases.}
\label{fig:sub_no_rest}
\end{figure}

Recall that $Q^{\varepsilon}(t,x;\bar{u})$ is defined in (\ref{eqn:defofQ}) as the second moment for the scheme based on $\bar{u}$.
In equation (\ref{Eq:game_rep}) below we give a representation for  $Q^{\varepsilon}(t,x;\bar{u})$.
The representation follows
essentially from the arguments of Section~2.3 of \cite{dupspiwan}, which is in turn
based on the representation for exponential integrals with respect to Brownian
motion given in \cite{boudup}. It is given in terms of the
value of a stochastic differential game, where the player corresponding to the
importance sampling scheme has already selected their control (i.e., $\bar{u}$).
The characterization of performance for importance sampling in terms of
games was first introduced in \cite{dupwan3}. The only difference between the
use here and in \cite{dupspiwan} is that there the function $h$ is bounded,
which is not true here since we consider an escape probability. However, the
bounds stated below can be obtained by first replacing $\infty1_{\{
\tau>T \} }$ by $M1_{ \{  \tau>T \}  }$ and then letting
$M\uparrow\infty$.

Let $\mathfrak{F}_{t}$ be a filtration satisfying the usual conditions of
completion and right continuity and which measures the Wiener process. Let
$\mathcal{A}$ denote the set of all $\mathfrak{F}_{t}$-progressively
measurable $d$-dimensional processes $v= \{  v(s),0\leq s\leq T \}$
that satisfy
\[
\mathbb{E}\int_{0}^{T}\bigl\llVert v(t)\bigr
\rrVert ^{2}\,dt<\infty.
\]
Let $\varepsilon>0$ be fixed, and let $\hat{X}^{\varepsilon}$ be the unique
strong solution to
\[
d\hat{X}^{\varepsilon}(s)=-DV\bigl(\hat{X}^{\varepsilon}(s)\bigr)\,ds+ \bigl[
\sqrt{\varepsilon}\,dB(s)-\bigl[\bar{u}\bigl(\hat{X}^{\varepsilon}(s)\bigr)-v(s)
\bigr]\,ds \bigr] \label{Eq:control_diffusion}%
\]
with initial condition $\hat{X}^{\varepsilon}(0)=y$. Let $\hat{\tau
}^{\varepsilon}$ denote the first time $\hat{X}^{\varepsilon}$ exits
$\mathcal{D}$. Then
\begin{eqnarray}
&& \hspace*{4pt}-\varepsilon\log Q^{\varepsilon}(0,y;\bar{u})
\nonumber
\\[-8pt]
\label{Eq:game_rep}
\\[-8pt]
\nonumber
&&\hspace*{4pt} \qquad=\inf_{v\in\mathcal{A}}%
\mathbb{E} \biggl[ \frac{1}{2}\int_{0}^{\hat{\tau}^{\varepsilon}}
\bigl\llVert v(s)\bigr\rrVert ^{2}\,ds-\int_{0}^{\hat{\tau}^{\varepsilon}}
\bigl\llVert \bar {u}\bigl(\hat{X}^{\varepsilon}(s)\bigr)\bigr\rrVert
^{2}\,ds+\infty1_{ \{  \hat{\tau
}^{\varepsilon}>T \}  } \biggr].
\end{eqnarray}

It is important to note that (\ref{Eq:game_rep}) provides a nonasymptotic
representation for the performance measure $Q^{\varepsilon}(0,y;\bar{u})$.
However, to obtain a more concrete statement regarding the performance of the
importance sampling scheme, we will want bounds on $Q^{\varepsilon}%
(0,y;\bar{u})$ that are more explicit than the right-hand side of
(\ref{Eq:game_rep}). We do this by observing that when viewed as a function of
an arbitrary starting point $(t,x)$, $-\varepsilon\log Q^{\varepsilon
}(t,x;\bar{u})$ also satisfies a nonlinear PDE of the same general form as
(\ref{Eq:HJBequation1}) (plus terminal and boundary conditions), and thus
lower bounds can be obtained by constructing subsolutions for this PDE.
However, a key difference is that in contrast to (\ref{Eq:HJBequation1}), the
PDE for (\ref{Eq:game_rep}) involves a second derivative term. One cannot
avoid this issue, in that second derivative information and $\varepsilon$
dependence are needed if one is to obtain nonasymptotic bounds, even when the
change of measure is based on a first order equation.

We next give the statement of the lower bound [see (\ref{Eq:LBQ})] as it applies to the special
case of this section. A more general statement and the proof will be given in
Lemma \ref{L:GeneralBound}. The proof is an easy consequence of
It\^{o}'s formula and the min/max representation
\[
\mathbb{H}(x,p)=\inf_{v}\sup_{u}
\biggl[ \bigl\langle p,-DV(x)-u+v \bigr\rangle -\dfrac{1}{2}\llVert u\rrVert
^{2}+\dfrac{1}{4}\llVert v\rrVert ^{2} \biggr].
\]
Define
\[
\mathcal{G}^{\varepsilon}[W](t,x)=W_{t}(t,x)+\mathbb{H}
\bigl(x,DW(t,x)\bigr)+\frac{\varepsilon}{2}D^{2}W(t,x),
\]
and let $\bar{W}$ be a subsolution to $\mathcal{G}^{\varepsilon}[W]=0$
together with the boundary conditions $W(t,x)=0$ for $t<T,x\in A_{L}$,
$W(t,x)=\infty$ for $t<T,x\in A_{\ell}$, and terminal condition $W(T,x)=\infty
$ for $x\in\mathcal{D}$. Suppose\vspace*{1pt} $\bar{u}$ is the control based on a given
smooth function $\bar{U}$, that is, $\bar{u}(t,x)=-D\bar{U}(t,x)$. Then
\begin{eqnarray}
&& -\varepsilon\log Q^{\varepsilon}(0,y;\bar{u})\nonumber
\\
&& \qquad=\inf_{v\in\mathcal{A}\dvtx \hat{\tau}^{\varepsilon}\leq T \ \mathrm{w.p.1}%
}\mathbb{E} \biggl[ \frac{1}{2}\int
_{0}^{\hat{\tau}^{\varepsilon}}\bigl\llVert v(s)\bigr\rrVert
^{2}\,ds-\int_{0}^{\hat{\tau}^{\varepsilon}}\bigl\llVert \bar
{u}\bigl(s,\hat{X}^{\varepsilon}\bigr)\bigr\rrVert ^{2}\,ds \biggr]
\nonumber
\\[-8pt]
\label{Eq:LBQ}
\\[-8pt]
\nonumber
&&\qquad \geq 2\bar{W}(0,y)+\inf_{v\in\mathcal{A}\dvtx \hat{\tau}^{\varepsilon}\leq T \ \mathrm{w.p.1}}\mathbb{E} \biggl[ \int
_{0}^{\hat{\tau}^{\varepsilon}}2\mathcal{G}%
^{\varepsilon}[
\bar{W}]\bigl(s,\hat{X}^{\varepsilon}\bigr)\,ds\\
&& \hspace*{121pt}\qquad\qquad{}-\int_{0}^{\hat{\tau
}^{\varepsilon}}
\bigl\llVert D\bar{W}\bigl(s,\hat{X}^{\varepsilon}\bigr)-D\bar {U}\bigl(s,
\hat{X}^{\varepsilon}\bigr)\bigr\rrVert ^{2}\,ds \biggr].
\nonumber
\end{eqnarray}

Next we show how (\ref{Eq:LBQ}) can be used to obtain bounds that are uniform
in $T$. For $\eta\in(0,1)$ define
\[
U^{\eta}(x)\doteq(1-\eta)U(x),
\]
where $U(x)=-2(V(x)-L)$ is the subsolution based on the quasipotential for $V$ as above, and
assume that $U$ is twice continuously differentiable. Then as with $U$, the
appropriate boundary and terminal inequalities hold for $U^{\eta}$. We next
evaluate the right-hand side of (\ref{Eq:LBQ}) when the subsolution is $\bar{W}=U^{\eta}$ and the control is based on $\bar
{U}=U$. A straightforward calculation gives
\[
\mathbb{H}\bigl(x,DU^{\eta}(x)\bigr)=2\bigl(\eta-\eta^{2}
\bigr)\bigl\llVert DV(x)\bigr\rrVert ^{2},
\]
and therefore
\[
\mathcal{G}^{\varepsilon}\bigl[U^{\eta}\bigr](x)-\tfrac{1}{2}\bigl
\llVert DU^{\eta
}(x)-DU(x)\bigr\rrVert ^{2}=2\bigl(\eta-2
\eta^{2}\bigr)\bigl\llVert DV(x)\bigr\rrVert ^{2}-
\varepsilon(1-\eta)D^{2}V(x).
\]
For $\varepsilon>0$ but smaller than a constant that depends on $\inf_{x\in\mathcal{D}}\llVert  DV(x)\rrVert  ^{2}$ and $\sup_{x\in
\mathcal{D}}\llVert  D^{2}V(x)\rrVert  ^{2}$, there is $\eta
=\eta(\varepsilon)$ with $\eta(\varepsilon)\rightarrow0$ as $\varepsilon
\rightarrow0$ such that the last display is nonnegative. We then obtain from
(\ref{Eq:LBQ}) the nonasymptotic upper bound
\[
Q^{\varepsilon}(0,y;\bar{u})\leq e^{-({2}/{\varepsilon})U^{\eta}%
(y)}=e^{-({2}/{\varepsilon})(1-\eta)U(y)}.
\label{Eq:non-asy_bound2}%
\]

Note that this bound is independent of $T$, and also that this argument is not
possible when $0\in\mathcal{D}$. Indeed, since $\mathbb{H}(0,p)=-\llVert
p\rrVert  ^{2}/2\leq0$ for all $p$, $\mathcal{G}^{\varepsilon}[U^{\eta
}](x)\geq0$ is not possible for any choice of $\eta<1$ when $0\in\mathcal{D}$.

The quality of the bound obtained by this method depends on the degree to
which the subsolution obtained for the PDE $\mathcal{G}^{\varepsilon}[W]=0$
(plus boundary and terminal conditions) accurately approximates the solution
to this equation. In this example, we have used a crude method to produce
such a subsolution, which is to simply reduce a given subsolution to the
$\varepsilon=0$ equation by a constant factor of $(1-\eta)$. An examination of
the calculations suggest that the bound is not at all tight, which turns out
to be true. In fact, in this situation we can construct a better subsolution
and hence a tighter bound. For example, when $V(x)=x^2/2$, then so long as the origin is not included
$-x^{2}+L+2\varepsilon\log(x/\sqrt{L})$ can be used to obtain tighter bounds,
though this function is not convenient for the time dependent problem.

Note that the two functions $U$ and $U_{\eta}$ play very different roles here.
One is used to design an importance sampling scheme (here $U$), and one used
for its analysis (here $U_{\eta}$). Indeed, $U_{\eta}$ with $\eta>0$ is used
only for the analysis of the scheme that corresponds to $U$, and in particular
to derive a bound that is independent of~$T$.  However, the design of the
scheme and thus the simulation algorithm use the control $\bar{u}(t,x)=-D\bar{U}(t,x)$.

Next we consider the behavior of $Q^{\varepsilon}(0,y;\bar{u})$ when
$0\in\mathcal{D}$. In this case, we claim that $Q^{\varepsilon}(0,y;\bar{u})$
grows without bound in $T$ for all $\varepsilon>0$, and therefore the
performance of the control based on the quasipotential degrades as $T$ becomes
large. To show this is true, we use the game representation (\ref{Eq:game_rep}) to establish a
lower bound on $Q^{\varepsilon}(0,y;\bar{u})$. We again examine a particular
situation, which is to estimate the probability of escape from $\mathcal{D}%
=\{x\in\mathbb{R}^{d}\dvtx V(x)<L\}$ before time $T$, after starting at $y$ at time
$0$; see Figure~\ref{fig:with_rest}. The subsolution is\vspace*{1pt} still that of Figure~\ref{fig:sub_no_rest}.

\begin{figure}

\includegraphics{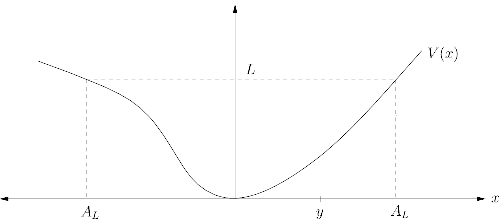}

\caption{Escape problem with a rest point.}
\label{fig:with_rest}
\end{figure}

With the understanding that $\hat{\tau}^{\varepsilon}$ now represents the time
of escape of $\hat{X}^{\varepsilon}$ from $\{x\in\mathbb{R}^{d}\dvtx V(x)<L\}$,
representation (\ref{Eq:game_rep}) is still valid. Suppose that $T$ is large,
and note that, while $\bar{u}(x)=-DU(x)=2DV(x)$ destabilizes the origin when
used to construct the measure used for importance sampling, in the
representation (\ref{Eq:game_rep}), it actually \emph{increases} the stability
of the origin, in the sense that $-DV(x)-\bar{u}(x)=-3DV(x)$. As a
consequence, it is easy to construct a control $v$ which shows poor
performance as $T\rightarrow\infty$. The construction is suggested in Figure~\ref{fig:bad_large_T}. With $T$ large we divide $[0,T]$ into an initial part
$[0,T-K)$ and a final part $[T-K$,$T]$, with $K$ fixed. During the first part
we apply $v(t)=0$. Because the resulting dynamics of $\hat{X}^{\varepsilon}$
are stable about the origin, with very high probability the process settles
around $0$ for the entire interval $[0,T-K)$. In the game representation there
is then a running cost of $\Vert\bar{u}(\hat{X}^{\varepsilon}(s))\Vert^{2}$,
which one can check is of order $\varepsilon>0$. In the second portion we
apply a control which leads to escape prior to time $T$, with a cost that may
depend on $K$ but is independent of $T$. An example of such a control, at
least away from the origin, is as illustrated in Figure~\ref{fig:bad_large_T}.
The precise details of the construction in this second part are not important.
All that is needed is that such a control exists, which can easily be
demonstrated.
\begin{figure}

\includegraphics{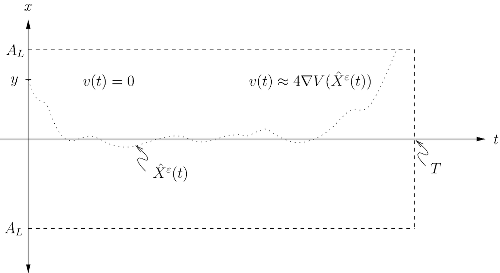}

\caption{Construction of $v$.}
\label{fig:bad_large_T}
\vspace*{-6pt}
\end{figure}

When the two parts are combined, we have a control that provides an upper
bound of the form
\[
-\varepsilon\log Q^{\varepsilon}(0,y;\bar{u})\leq-\varepsilon
C_{1}%
[T-K]+C_{2},
\]
where $C_{1}$ and $C_{2}$ are positive constants. This shows that
\begin{equation}\label{Eq:bad_lower_bound}
Q^{\varepsilon}(0,y;\bar{u})\geq e^{C_{1}[T-K]}e^{-({1}/{\varepsilon})C_{2}},
\end{equation}
and thus for fixed $\varepsilon>0$ and large $T$ the term we have called the
pre-exponential term dominates and the scheme is far from optimal. We
find that in this situation there are two exponential scalings, one in the
noise strength and one in the length of the time interval, and the issue of
which dominates depends on their relative sizes.

These effects are reflected in computational data. The measure used for comparing
different schemes is the relative error per sample. This is defined for an estimate based on $N$ samples by
\[
\mbox{Relative error per sample} \doteq \sqrt{N} \frac{\mathrm{Standard} \ \mathrm{deviation}  \ \mathrm{of} \  \mathrm{the} \  \mathrm{estimator}}%
{\mathrm{Expected} \  \mathrm{value} \  \mathrm{of} \  \mathrm{the} \ \mathrm{estimator}}.
\]
The quantities reported in the tables are based on simulated data and thus give the corresponding estimated relative error per sample. In Tables~\ref{Table1a} and
\ref{Table1b} we present both estimated values and relative errors for the
problem of escape from an interval of the form $[-A,A]$, with $A=1$. The
process is a one-dimensional Gauss--Markov model with drift $-cx$ and diffusion
coefficient $\sqrt{\varepsilon}\sigma$ [see (\ref{Eq:GM})], with $c=\sigma=1$.
In the tables, values of $T$ appear at the top, and values of $\varepsilon$
along the left-hand side. Each computed value is based on $N=10^{7}$ samples. A dash
indicates that no samples escaped. To ease the presentation the relative
errors are rounded to the nearest integer. Owing to the fact that the
subsolution based on the quasipotential is far from optimal in any sense when
$T$ is small, the relative errors are large for small $T$ and decrease until
approximately $T=2$. For larger $T$ the errors grow rapidly with~$T$. Note
that the estimated relative errors in Table~\ref{Table1b} are not necessarily
accurate for large $T$, since they are subject to the same errors that can
affect the probability estimates, but do indicate the qualitative worsening of
estimation accuracy.
\begin{table}[t]
\tabcolsep=0pt
\tablewidth=\textwidth
\caption{Using the subsolution based on quasipotential throughout. Estimated
values for different pairs~$(\varepsilon,T)$, two-sided problem}
\label{Table1a}
\begin{tabular*}{\tablewidth}{@{\extracolsep{\fill}}lcccccccc@{}}
\hline
& \multicolumn{8}{c@{}}{$\bolds{T}$}\\[-4pt]
& \multicolumn{8}{l@{}}{\hrulefill}\\
$\bolds{\varepsilon}$ & $\mathbf{0.25}$ & $\mathbf{0.5}$ & $\mathbf{1}$ & $\mathbf{1.5}$ &
$\mathbf{2.5}$ & $\mathbf{10}$ & $\mathbf{14}$ & $\mathbf{18}$\\
\hline
$0.20$ & 9.8e$-$07 & 2.0e$-$04 & 3.3e$-$03 & 9.0e$-$03  & 9.1e$-$03 & 1.2e$-$01
& 1.7e$-$01 & 2.0e$-$01\\
$0.16$ & 3.6e$-$08 & 2.5e$-$05 & 7.3e$-$04 & 2.2e$-$03 & 6.6e$-$03 & 4.0e$-$02
& 5.7e$-$02  &  6.4e$-$02 \\
$0.13$ & 9.8e$-$10  & 2.3e$-$06  & 1.3e$-$04  & 5.1e$-$04 & 1.6e$-$03  & 1.1e$-$02
&  1.5e$-$02  &  2.7e$-$02 \\
$0.11$ &  4.7e$-$11  &  2.4e$-$07  &  2.4e$-$05  &  1.1e$-$04  &  3.9e$-$04  &  2.8e$-$03
&  4.0e$-$03  &  6.0e$-$03 \\
$0.09$ & $-$ &  8.8e$-$09  &  2.1e$-$06  & 1.2e$-$05  & 5.2e$-$05  &  4.0e$-$04  &
 5.8e$-$04  &  8.3e$-$04 \\
 $0.07$ & $-$ &  5.5e$-$11  &  5.0e$-$08  &  4.3e$-$07  &  2.2e$-$06  & 1.9e$-$05  &
2.8e$-$05  & 3.7e$-$05\\
0.05 & $-$ &  5.6e$-$15  &  5.9e$-$11  &  9.7e$-$10  & 6.9e$-$09  & 7.1e$-$08 &
 1.1e$-$07  &  1.3e$-$07\\
 \hline
\end{tabular*}
\vspace*{-12pt}
\end{table}

\begin{table}[b]
\vspace*{-6pt}
\tablewidth=250pt
\caption{Using the subsolution based on quasipotential throughout. Relative
errors per sample for different pairs $(\varepsilon,T)$, two-sided problem}
\label{Table1b}
\begin{tabular*}{250pt}{@{\extracolsep{\fill}}lcccccccc@{}}\hline
& \multicolumn{8}{c@{}}{$\bolds{T}$}\\[-4pt]
& \multicolumn{8}{l@{}}{\hrulefill}\\
$\bolds{\varepsilon}$ & $\mathbf{0.25}$ & $\mathbf{0.5}$ & $\mathbf{1}$ & $\mathbf{1.5}$ &
$\mathbf{2.5}$ & $\mathbf{10}$ & $\mathbf{14}$ & $\mathbf{18}$\\
\hline
$0.20$ & \phantom{00}$91$ & \phantom{00}$7$ & $2$ & $1$ & $1$ & $10$ & $51$ & $179$\\
$0.16$ & \phantom{0}$253$ & \phantom{0}$10$ & $2$ & $1$ & $1$ & $10$ & $48$ & $139$\\
$0.13$ & \phantom{0}$748$ & \phantom{0}$16$ & $3$ & $1$ & $1$ & \phantom{0}$9$ & $48$ & $378$\\
$0.11$ & $1594$ & \phantom{0}$26$ & $3$ & $1$ & $1$ & $10$ & $42$ & $272$\\
$0.09$ & $-$ & \phantom{0}$49$ & $4$ & $2$ & $1$ & \phantom{0}$9$ & $43$ & $357$\\
$0.07$ & $-$ & $127$ & $5$ & $2$ & $1$ & \phantom{0}$8$ & $47$ & $251$\\
$0.05$ & $-$ & $714$ & $8$ & $2$ & $1$ & \phantom{0}$8$ & $42$ & $145$\\
\hline
\end{tabular*}
\end{table}

Tables~\ref{Table1c} and \ref{Table1d} present the approximated
values and relative errors for the problem with the domain $(-\infty,A]$ and
escape is possible therefore only at $A$. The results are of the same qualitative
form as before, and carried out only to $T=10$.

\begin{table}[t]
\tablewidth=\textwidth
\caption{Using the subsolution based on quasipotential throughout. Estimated
values for different pairs~$(\varepsilon,T)$, one-sided problem}
\label{Table1c}
\begin{tabular*}{\tablewidth}{@{\extracolsep{\fill}}lccccccc@{}}
\hline
& \multicolumn{7}{c@{}}{$\bolds{T}$}\\[-4pt]
& \multicolumn{7}{c@{}}{\hrulefill}\\
$\bolds{\varepsilon}$  & $\mathbf{0.25}$ & $\mathbf{0.5}$ & $\mathbf{1}$ & $\mathbf{1.5}$ &
$\mathbf{2.5}$ & $\mathbf{7}$ & $\mathbf{10}$\\
\hline
$0.20$ & 4.6e$-$07  & 1.0e$-$04  & 1.7e$-$03 & 4.5e$-$03 & 1.1e$-$02 & 4.2e$-$02
& 6.2e$-$02\\
$0.16$ & 2.1e$-$08  & 1.3e$-$05  & 3.7e$-$04  & 1.2e$-$03  & 3.3e$-$03  & 1.3e$-$02
& 2.0e$-$02 \\
$0.13$ & 2.7e$-$10  & 1.1e$-$06  & 6.5e$-$05 & 2.5e$-$04  & 7.9e$-$04  & 3.5e$-$03
& 5.3e$-$03\\
$0.11$ & 1.4e$-$11  & 1.2e$-$07  & 1.2e$-$05 & 5.7e$-$05 & 2.0e$-$04 & 9.2e$-$04
& 1.4e$-$03\\
$0.09$ & $-$ & 4.3e$-$09 & 1.1e$-$06 & 6.5e$-$06  & 2.6e$-$06 & 1.3e$-$04 &
$2.0\mathrm{e}{-}04$\\
$0.07$ & $-$ & 2.4e$-$11 & 2.5e$-$08 & 2.2e$-$07 & 1.1e$-$06  & 6.1e$-$06 &
9.3e$-$06\\
$0.05$ & $-$ & 1.7e$-$15 & 3.0e$-$12  & 4.9e$-$10  & 3.5e$-$09  & 2.2e$-$08 &
3.5e$-$08\\
\hline
\end{tabular*}
\end{table}

\begin{table}[b]
\tablewidth=250pt
\caption{Using the subsolution based on quasipotential throughout. Relative
errors per sample for different pairs $(\varepsilon,T)$, one-sided problem}
\label{Table1d}
\begin{tabular*}{250pt}{@{\extracolsep{\fill}}lccccccc@{}}
\hline
 & \multicolumn{7}{c@{}}{$\bolds{T}$}\\[-4pt]
& \multicolumn{7}{c@{}}{\hrulefill}\\
$\bolds{\varepsilon}$ & $\mathbf{0.25}$ & $\mathbf{0.5}$ & $\mathbf{1}$ & $\mathbf{1.5}$ &
$\mathbf{2.5}$ & $\mathbf{7}$ & $\mathbf{10}$\\
\hline
$0.20$ & \phantom{0}$132$ & \phantom{00}$10$ & \phantom{0}$3$ & $2$ & $2$ & $5$ & $15$\\
$0.16$ & \phantom{0}$331$ & \phantom{00}$15$ & \phantom{0}$3$ & $2$ & $2$ & $4$ & $14$\\
$0.13$ & $1418$ & \phantom{00}$23$ & \phantom{0}$4$ & $2$ & $2$ & $4$ & $14$\\
$0.11$ & $3162$ & \phantom{00}$36$ & \phantom{0}$4$ & $2$ & $2$ & $4$ & $14$\\
$0.09$ & $-$ & \phantom{00}$70$ & \phantom{0}$6$ & $3$ & $2$ & $4$ & $13$\\
$0.07$ & $-$ & \phantom{0}$194$ & \phantom{0}$7$ & $3$ & $2$ & $4$ & $12$\\
$0.05$ & $-$ & $1300$ & $12$ & $4$ & $2$ & $4$ & $12$\\
\hline
\end{tabular*}
\end{table}

It is useful to compare the two situations and identify why uniform control
of pre-exponential terms was not possible when $0\in\mathcal{D}$. In both
cases the control was based on the quasipotential, which is a valid
subsolution to the $\varepsilon=0$ problem. When using It\^{o}'s formula to bound
the second moment of the estimator, we must of course deal with the second
derivative term, which is multiplied by $\varepsilon>0$. It can happen that
this term has a sign that degrades (increases) the second moment, and indeed
this is always true in a neighborhood of the origin (this is essentially due
to the convexity of $V$ near the origin). For the case where $0\notin
\mathcal{D}$ and for sufficiently small $\varepsilon>0$, this could be
balanced by using that when $x\neq0$ $\bar{u}(x)$ and therefore $D\bar{U}(x)$
are nonzero. However, this is not possible when the rest point is included in
the domain of interest. Indeed, the running cost that is accumulated in the
construction leading to (\ref{Eq:bad_lower_bound}) corresponds to this term,
and as that argument shows it cannot be removed. The construction also
suggests how the large variance comes about, which is that some trajectories
generated under the change of measure defined by $\bar{u}$ remain in a
neighborhood of the origin for a long time, in spite of the fact that with
such dynamics the origin is an unstable equilibrium point. The likelihood
ratios along these trajectories can vary greatly and, even though they are
themselves relatively unlikely, they are likely enough to increase the
variance of the estimator to the point where it will become worse than
standard Monte Carlo. As such, they are reminiscent of the ``rogue'' trajectories which lead to poor performance of
nondynamic forms of importance sampling as discussed in
\mbox{\cite{dupwan3,glakou,glawan}}.

It will turn out that to overcome the difficulties introduced by the rest
point, one must do a much better job of approximating the optimal change of
measure than is possible based just on a time and $\varepsilon$-independent
subsolution. However, it also turns out that the additional accuracy is needed
only near the rest point, where in fact explicit time and $\varepsilon
$-dependent solutions can be found. These are then combined with the simple
time-independent subsolution based on the quasipotential to produce schemes
that are nearly optimal and which protect against both sources of significant
variance. An overview of the construction of such schemes is the topic of the
next section.

\section{Combining subsolutions with a refined local analysis: The linear
problem}
\label{S:LinearProblem}

In this section we combine a local analysis that produces a time and
\mbox{$\varepsilon$-dependent} scheme near the rest point with a scheme based on the
quasipotential elsewhere. There are of course few process models and problems
for which the related HJB equation can be solved explicitly. However, a class
of processes where this is possible are the Gauss--Markov models, that is, SDEs
with drift that is linear in the state and constant diffusion matrix. For
these processes and for terminal conditions of the appropriate form, both the
limit ($\varepsilon=0$) PDE and the prelimit ($\varepsilon>0$) PDE have an
explicit solution that can be expressed in terms of the value function of a
linear-quadratic regulator (LQR) control problem.
\begin{figure}[b]
\vspace*{-3pt}

\includegraphics{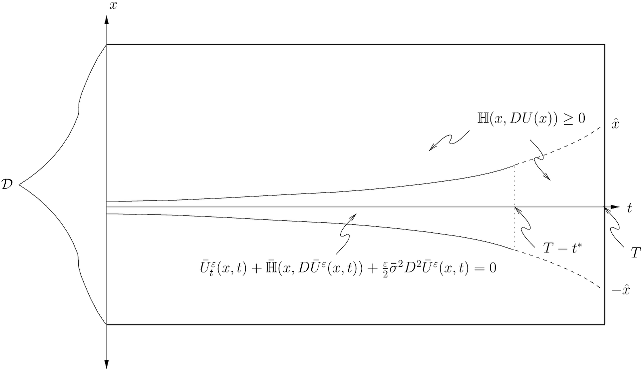}

\vspace*{-3pt}
\caption{Partition of the state space in a combined scheme.}
\label{fig:localization2}
\end{figure}

Our ultimate approach to the construction of importance sampling schemes is
suggested by Figure~\ref{fig:localization2}. The problem of interest is of the
form (\ref{Eq:EstimationTargetSmooth1}) or an analogous problem involving
escape from $\mathcal{D}$ prior to $T$. The particular problem will fix the
boundary and terminal conditions on $[0,T]\times\partial\mathcal{D}$ and
$\{T\}\times\mathcal{D}$. In the figure, we are interested in estimating the
exit probability
\[
\theta^{\varepsilon}=\mathbb{P}_{0,0} \bigl\{ X^{\varepsilon}\mbox{ hits }{-}A_{1}\mbox{ or }A_{2}\mbox{ before time }T \bigr\},
\]
where $0$ is the rest point and $\mathcal{D}=(-A_{1},A_{2})$ with $A_{1},A_{2}>0$.

In most of the domain, which is the section outside the curves that terminate
at $\pm\hat{x}$ and after time $T-t^{\ast}$, the control is based on a
subsolution $U$ constructed in terms of the quasipotential. Within the curves
the solution to (\ref{Eq:Diffusion1}) is well approximated by a Gauss--Markov
process. Hence within this region we will use a control that would be
appropriate for a problem if the process were instead the approximating
Gauss--Markov model. The function $\mathbb{\bar{H}}$ that defines the PDE for
this region is therefore the one corresponding to the Gauss--Markov process.
Besides the dynamics in this region (which are determined by the Gauss--Markov
approximation), we must choose a terminal condition. This will be given by the
minimum of two quadratic functions, one centered at $\hat{x}$ and one centered
at $-\hat{x}$. The parameter $\hat{x}$ is a nominal value that for purposes of
the present discussion can be taken to be $1$. The parameter $t^{\ast}$ plays
two roles. One is to determine the size of the region on which the true
dynamics are approximated by Gauss--Markov dynamics. The second role is related
to the fact that if the control based on the quadratics centered at ${\pm}\hat{x}$ is used all the way to $T$, then singularities in the gradient (and hence the control) will develop as $t \uparrow T$. For this
reason, we switch back to the control based on the quasipotential after
$T-t^{\ast}$, and $t^{\ast}$ must be chosen so that the subsolution property
is preserved across the handoff time $T-t^{\ast}$.

While the discussion above suggests the correct decomposition of the domain,
the actual construction is more complex, since the control should transition
nicely when moving between the regions, and the construction of a scheme for
which rigorous bounds can be proved will require additional mollification and
approximations. However, the building blocks are always subsolutions to the
indicated PDEs. The form of the surrogate problem for the Gauss--Markov model
needs to be explained, as well as various boundary and terminal conditions. To
simplify the discussion, we consider a sequence of successively more general
problems. In this paper we complete the analysis for the one-dimensional
problem. The multi-dimensional problem will be addressed in a companion paper.

\subsection{A one-dimensional Gauss--Markov model}

Our first goal is to construct a subsolution for the $\varepsilon=0$ problem
for the processes that will be used in the localization. This can be related
to the problem of estimating the probability that the solution to
\[
dX^{\varepsilon}(s)=-cX^{\varepsilon}(s)\,ds+\sqrt{\varepsilon}\bar{\sigma }\,dB(s),
\qquad X^{\varepsilon}(t)=x\in(-A,A) \label{Eq:GM}
\]
escapes from the interval $[-A,A]$ before time $T$, a problem that was also
used for the computational examples of the last subsection. The parameters $c$
and $\bar{\sigma}$ are positive constants. In this subsection we will take
$\hat{x}=A$, and because of this can postpone the issue regarding
singularities in the control at $t=T$. We thus also take~$t^{\ast}$ to be
zero, and will return to the role of $t^{\ast}$ and its selection for a
general problem in the next subsection. The corresponding PDE for the escape
probability is
\[
U_{t}^{\varepsilon}+\mathbb{H}\bigl(x,DU^{\varepsilon}\bigr)+
\frac{\varepsilon}{2}%
\bar{\sigma}^{2}D^{2}U^{\varepsilon}=0,
\qquad \mathbb{H}(x,p)=-cxp-\dfrac{1}{2}\bar{\sigma}^{2}p^{2},
\label{eqn:eps_pos_PDE}%
\]
plus the terminal and boundary conditions
\[
U^{\varepsilon}(t,x)=\cases{0, &\quad $x=\pm A,t\in [0,T]$,
\vspace*{3pt}\cr
\infty,  &\quad$x\in(-A,A),t=T$.}
\]
While simple in appearance, this equation does not have an explicit solution.

The equation obtained in the limit $\varepsilon\rightarrow0$ is more
tractable, and the unique viscosity solution can be described as follows.
$U^{0}(t,x)$ corresponds to the variational problem
\[
\inf \biggl\{ \int_{t}^{T}\frac{1}{2\bar{\sigma}^{2}}
\bigl\llvert \dot{\phi }(s)+c\phi(s)\bigr\rrvert ^{2}\,ds\dvtx \phi(t)=x,\bigl
\llvert \phi(s)\bigr\rrvert \geq A\mbox{ some }s\in [ t,T] \biggr\}.
\]
Depending on how and when the minimizing trajectory leaves $[-A,A]$, the
solution takes a particular explicit form. (In all cases the minimizer can be
found by solving the appropriate Euler--Lagrange equation.) If for the initial
condition $(t,x)$ the minimizing trajectory leaves before time $T$, then
\[
U^{0}(t,x)=F_{1}(x)\doteq\frac{c}{\bar{\sigma}^{2}} \bigl[
A^{2}-x^{2} \bigr].
\]
This is the case when the minimal cost is the negative of the quasipotential,
translated by a constant to satisfy the boundary condition at the exit
location. Such initial conditions satisfy $|x|\geq Ae^{c(t-T)}$. When $0<x<$
$Ae^{c(t-T)}$ the minimizer leaves through $A$ at exactly time $T$, and the
minimizing value is
\[
U^{0}(t,x)=F_{2}(t,x)\doteq\frac{c}{\bar{\sigma}^{2}}
\frac{(A-xe^{c(t-T)}
)^{2}}{[1-e^{2c(t-T)}]}.
\]
One can also interpret $F_{2}$ as the minimal cost for a linear quadratic
regulator with a singular terminal cost applied at time $T$, that is, a cost that
equals $0$ at $A$ and $\infty$ otherwise.

By symmetry it is clear that when $x<0$ the minimizing trajectory will exit
at~$-A$. Define
\begin{equation}
\label{Eq:BasicFcn1}
U_{+}^{0}(t,x)= \cases{F_{1}(x), & \quad\mbox{if $x\geq Ae^{c(t-T)}$,}
\vspace*{3pt}\cr
F_{2}(t,x), & \quad\mbox{if $0<x<Ae^{c(t-T)}$.}}
\end{equation}
Setting $U_{-}^{0}(t,x)=U_{+}^{0}(t,-x)$, we have
\[
U^{0}(t,x)= \cases{U_{+}^{0}(t,x), &\quad
\mbox{if $x\geq0$,}
\vspace*{3pt}\cr
U_{-}^{0}(t,x), & \quad\mbox{if $x
\leq0$.}} \label{Eq:BasicFcn2}%
\]

Note that when $x=Ae^{c(t-T)}$,
\begin{eqnarray*}
F_{2}\bigl(t,Ae^{c(t-T)}\bigr) &=& \frac{c}{\bar{\sigma}^{2}}
\frac{(A-xe^{c(t-T)})^{2}%
}{[1-e^{2c(t-T)}]}=\frac{c}{\bar{\sigma}^{2}}\frac{(A-Ae^{2c(t-T)})^{2}%
}{[1-e^{2c(t-T)}]}\\
&=& \frac{c}{\bar{\sigma}^{2}} \bigl[
A^{2}-x^{2} \bigr] =F_{1}\bigl(Ae^{c(t-T)}
\bigr),
\end{eqnarray*}
and therefore $U^{0}(t,x)$ is continuous for all $(t,x)\in [
0,T)\times [-A,A]$. In fact more is true, and one can check that
\[
DF_{2}\bigl(t,Ae^{c(t-T)}\bigr)=DF_{1}
\bigl(Ae^{c(t-T)}\bigr),
\]
and thus $U^{0}(t,x)$ has a continuous partial derivative in $x$ for
$x\in(0,A)$. The mapping $x\rightarrow F_{2}(t,x)$ is convex, and the graph of
this mapping lies above that of $F_{1}(x)$, which is concave. Thus the two
graphs intersect only at the point $(Ae^{c(t-T)},cA^{2}[1-e^{2c(t-T)}%
]/\bar{\sigma}^{2})$, where both functions and their first derivatives in $x$
agree; see Figure~\ref{fig:eps_zero_soln}.
\begin{figure}

\includegraphics{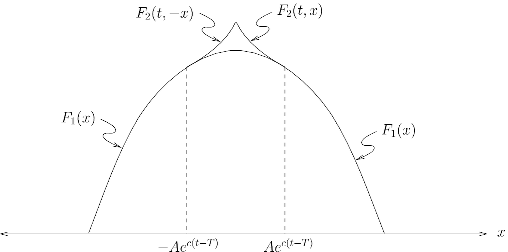}

\caption{The $\varepsilon=0$ solution for a fixed $t$.}
\label{fig:eps_zero_soln}
\end{figure}
Note also that $\hat{U}^{0}(t,x)\doteq U_{+}^{0}(t,x)$ for $x\in(-\infty,A]$
is the solution to the $\varepsilon=0$ problem with escape from $(-\infty,A]$
(a ``one-sided'' version of the problem of
escape from $[-A,A]$).

Simulation data for the schemes based on these two value functions are
presented below. The approximated values are omitted since they are
qualitatively similar to those in Tables~\ref{Table1a}~and~\ref{Table1c}, and
only relative errors are presented. Table~\ref{Table2a} gives data based on
$U^{0}$ for the problem of two-sided exit, and should be compared to Table~\ref{Table1b}. The use of the solution to HJB equation for $\varepsilon=0$
drastically improves the performance for small $T$, which is due to the fact
that the subsolution based on the quasipotential is a very poor approximation
to the solution for $\varepsilon>0$ for such~$T$. However, for large $T$ the
two schemes are comparably bad.
\begin{table}[t]
\tablewidth=250pt
\caption{Using the subsolution based on the explicit solution to the
$\varepsilon=0$ HJB equation. Relative errors per sample for different pairs
$(\varepsilon,T)$, two-sided problem}
\label{Table2a}
\begin{tabular*}{250pt}{@{\extracolsep{\fill}}lcccccccc@{}}
\hline
& \multicolumn{8}{c@{}}{$\bolds{T}$}\\[-4pt]
& \multicolumn{8}{l@{}}{\hrulefill}\\
$\bolds{\varepsilon}$ & $\mathbf{0.25}$ & $\mathbf{0.5}$ & $\mathbf{1}$ & $\mathbf{1.5}$ &
$\mathbf{2.5}$ & $\mathbf{10}$ & $\mathbf{14}$ & $\mathbf{18}$\\
\hline
$0.20$ & $2$ & $2$ & $1$ & $1$ & $1$ & $10$ & $57$ & $266$\\
$0.16$ & $3$ & $3$ & $1$ & $1$ & $1$ & $10$ & $55$ & $265$\\
$0.13$ & $4$ & $2$ & $1$ & $1$ & $1$ & \phantom{0}$9$ & $51$ & $394$\\
$0.11$ & $3$ & $2$ & $2$ & $1$ & $1$ & \phantom{0}$9$ & $44$ & $177$\\
$0.09$ & $3$ & $3$ & $2$ & $1$ & $1$ & \phantom{0}$9$ & $43$ & $314$\\
$0.07$ & $3$ & $3$ & $2$ & $2$ & $1$ & \phantom{0}$9$ & $67$ & $520$\\
$0.05$ & $4$ & $2$ & $2$ & $2$ & $1$ & \phantom{0}$8$ & $50$ & $278$\\
\hline
\end{tabular*}
\end{table}

The data for exit from one side are given in Table~\ref{Table2b}. In contrast
with the two-sided problem, the performance does not degrade so quickly with
large $T$. This suggests that the decline in performance is not so much due to
using the approximating $\varepsilon=0$ PDE, but rather the possible lack of
regularity of the solution to this equation.
\begin{table}[b]
\tablewidth=250pt
\caption{Using the subsolution based on the explicit solution to the
$\varepsilon=0$ HJB equation. Relative errors per sample for different pairs
$(\varepsilon,T)$, one-sided problem}
\label{Table2b}
\begin{tabular*}{250pt}{@{\extracolsep{\fill}}lccccccc@{}}
\hline
  & \multicolumn{7}{c@{}}{$\bolds{T}$}\\[-4pt]
  & \multicolumn{7}{l@{}}{\hrulefill}\\
$\bolds{\varepsilon}$  & $\mathbf{0.25}$ & $\mathbf{0.5}$ & $\mathbf{1}$ & $\mathbf{1.5}$ &
$\mathbf{2.5}$ & $\mathbf{7}$ & $\mathbf{10}$\\
\hline
$0.20$ & $1$ & $1$ & $1$ & $1$ & $1$ & $2$ & $3$\\
$0.16$ & $1$ & $1$ & $1$ & $1$ & $1$ & $2$ & $4$\\
$0.13$ & $1$ & $1$ & $1$ & $1$ & $1$ & $2$ & $3$\\
$0.11$ & $1$ & $1$ & $1$ & $1$ & $1$ & $2$ & $3$\\
$0.09$ & $1$ & $1$ & $1$ & $1$ & $1$ & $2$ & $3$\\
$0.07$ & $1$ & $1$ & $1$ & $1$ & $1$ & $2$ & $3$\\
$0.05$ & $1$ & $1$ & $1$ & $1$ & $1$ & $2$ & $3$\\
\hline
\end{tabular*}
\end{table}

We recall that the difficulty with the use of the subsolution based on
the quasipotential alone [here $F_{1}(x)$] was that in a neighborhood of $x=0$,
the second derivative was negative. As discussed in Section~\ref{S:EffectivenessOfSubsolutions}, when $T$ is large this concavity of the value function leads to poor
control of the variance of the associated scheme for both the one-sided and
two-sided problems.
When using the time-dependent $\varepsilon=0$ PDE as the basis for a scheme
for the one-sided problem, the introduction of $F_{2}(t,x)$ appears to have
largely mitigated the problem due to the second derivative term. Note that
this function determines the value of $\hat{U}^{0}(t,x)$ when $x=0$ and is
convex rather than concave in $x$. In contrast, for the two-sided problem
there is a concave singularity at the origin. There the second derivative in
$x$ is $-\infty$, and the subsolution property for the $\varepsilon>0$ problem
again fails. What is needed is a subsolution that works across the point $x=0$
for the $\varepsilon>0$ problem. Such a subsolution can be constructed using
the mollification introduced in the next subsection. Simulations based on such
a mollified version of the $\varepsilon=0$ subsolution (and with mollification
parameter $\delta=2\varepsilon$) are presented in Table~\ref{Table6b}, and
support the claim just made (the increase of relative errors when $T=23$ is again due to concavity creeping in, and will be dealt with in the
discussion and constructions below).
\begin{table}
\tablewidth=250pt
\caption{Using the subsolution based on the mollification of
$U_{+}^{0}(t,x)\wedge U_{-}^{0}(t,x)$. Relative errors per sample with
$\hat{x}=A=1$, two-sided problem}
\label{Table6b}
\begin{tabular*}{250pt}{@{\extracolsep{\fill}}lccccccccc@{}}
\hline
& \multicolumn{9}{c@{}}{$\bolds{T}$}\\[-4pt]
& \multicolumn{9}{l@{}}{\hrulefill}\\
$\bolds{\varepsilon}$ & $\mathbf{0.25}$ & $\mathbf{0.5}$ & $\mathbf{1}$ & $\mathbf{1.5}$ &
$\mathbf{2.5}$ & $\mathbf{10}$ & $\mathbf{14}$ & $\mathbf{18}$ & $\mathbf{23}$\\
\hline
$0.20$ & $1$ & $1$ & $1$ & $1$ & $1$ & $1$ & $3$ & $5$ & $61$\\
$0.16$ & $1$ & $1$ & $1$ & $1$ & $1$ & $1$ & $3$ & $5$ & $68$\\
$0.13$ & $1$ & $1$ & $1$ & $1$ & $1$ & $1$ & $2$ & $4$ & $64$\\
$0.11$ & $2$ & $1$ & $1$ & $1$ & $1$ & $1$ & $2$ & $4$ & $63$\\
$0.09$ & $2$ & $2$ & $2$ & $2$ & $1$ & $1$ & $2$ & $4$ & $58$\\
$0.07$ & $2$ & $2$ & $2$ & $2$ & $1$ & $1$ & $2$ & $3$ & $55$\\
$0.05$ & $2$ & $2$ & $2$ & $2$ & $1$ & $2$ & $2$ & $3$ & $52$\\
\hline
\end{tabular*}
\end{table}

In spite of the shortcomings we have described in this section, the solution
to the $\varepsilon=0$ problem serves as the starting point for a construction
that can be shown to perform well both in theory and in practice. There are
three important modifications that are needed:
\begin{itemize}
\item The first is that, in order to effectively deal with the $\varepsilon>0$
dynamics and in particular to avoid the degradation still present in Table~\ref{Table6b} when $T=23$, the region where the $F_{2}$ subsolution
determines the dynamics must be enlarged. The solution to the $\varepsilon=0$
problem constructed above leads to a region whose width vanishes exponentially in
$t-T$; see (\ref{Eq:BasicFcn1}).

\item The second modification is the use of a  mollification to eliminate
singularities such as the one at $x=0$ and help guarantee a global subsolution
property for the $\varepsilon>0$ PDE. The particular mollification we use is
very convenient, and was first used for importance sampling in \cite{dupwan5},
though with a somewhat different intended use.

\item The final modification was also alluded to previously, which is to
revert to the quasipotential based control in an interval of the form
$[T-t^{\ast},T]$. All three modifications will be introduced in the next
subsection in the context of the Gauss--Markov process.
\end{itemize}

\begin{remark}
It bears repeating that Tables~\ref{Table2a} and  \ref{Table2b} show one must be careful regarding singularities in the gradient of any subsolution used for importance sampling.
In particular, in the absence of rigorous estimates showing that the use of a particular weak sense subsolution is justified, one should appropriately
mollify and work with classical sense subsolutions.
\end{remark}

\subsection{Simulation scheme for the Gauss--Markov model}
\label{SS:1DimLinearModel_algorithm}

In this subsection we generalize the construction of the last subsection. As
discussed there, the generalizations are needed to address issues that play a
role in both the theoretical analysis of the scheme and its practical
performance. After introducing these generalizations, we will demonstrate
(numerically and theoretically) that the suggested change of measure does not
degrade in performance as $T$ gets larger and is close to optimality not only
as $\varepsilon\downarrow0$, but for fixed $\varepsilon>0$ as well.

We begin by introducing the first generalization, which is to replace the
terminal condition $\infty\cdot(x-\hat{x})^{2}+F_{1}(\hat{x})$, which was used
to define $F_{2}(t,x)$, by a terminal condition of the form $M(x-\hat{x}%
)^{2}/2+F_{1}(\hat{x})$. The parameter $\hat{x}$ replaces $A$ and is a nominal
value introduced to disconnect the localization from the boundary. The
solution to the LQR that corresponds to $M(x-\hat{x})^{2}/2+F_{1}(\hat{x})$,
which will be denoted $F_{2}^{M}(t,x)$, is automatically smaller than
$F_{2}(t,x)=F_{2}^{\infty}(t,x)$. The motivation for replacing $\infty$ by
$M/2$ is that we want the solution to the LQR problem [i.e., $F_{2}%
^{M}(t,x)$] to determine the control near $x=0$ for a set whose width is
uniformly (in $t$) bounded below away from zero. As discussed in the last
section, the second derivative term associated with $F_{1}$ is of the wrong
sign and degrades performance. The neighborhood where $F_{2}^{M}%
(t,x)<F_{1}(x)$ does not degenerate as $T-t\rightarrow\infty$, and its size is
decreasing in $M$. The introduction of $M$ complicates the construction by
also requiring\vspace*{1pt} mollification (in addition to the one that will be needed at
$x=0$), since $F_{2}^{M}$ can no longer be smoothly merged with $F_{1}$.

For $M\in(0,\infty)$ the solution to this LQR takes the form
\begin{eqnarray}
F_{2}^{M}(t,x)&=& a^{M}(t) \bigl( x-
\hat{x}e^{-c(t-T)} \bigr) ^{2}+F_{1}(\hat {x}),\nonumber\\
\eqntext{a^{M}(t) = \frac{ce^{2c(t-T)}}{(2c/M+\bar{\sigma}^{2})-\bar{\sigma
}^{2}e^{2c(t-T)}}>0.} \label{Eq:FM2_function}%
\end{eqnarray}
Recall that $F_{1}(x)=c[A^{2}-x^{2}]/\bar{\sigma}^{2}$ so that $F_{1}(A)=0$.
Define $F_{2,+}^{M}(t,x)=F_{2}^{M}(t,x),F_{2,-}^{M}(t,x)=F_{2}^{M}(t,-x)$. It
will be important to know which of $F_{2,+}^{M}$, $F_{2,-}^{M}$ and $F_{1}$
is smallest, and we note here several properties; see Figure~\ref{Fig:various_functions}. Let $K\doteq2c/M+\bar
{\sigma}^{2}$. The first is that there are two real solutions to $F_{2,+}%
^{M}(t,x)=F_{1}(x)$, and these take the form
\begin{equation}\label{Eq:roots}
\frac{\bar{\sigma}^{2}\hat{x}}{K} \biggl( e^{c(t-T)}\pm\sqrt{\frac{2cK}%
{M\bar{\sigma}^{4}}-
\frac{2c}{M\bar{\sigma}^{2}}e^{2c(t-T)}} \biggr).
\end{equation}
Between these roots $F_{2,+}^{M}(t,x)<F_{1}(x)$, and on the complement of the
interval the reverse inequality holds. The limit $t-T\rightarrow-\infty$ gives
the asymptotic endpoints of the interval where $F_{2,+}^{M}(t,x)<F_{1}(x)$,
which are
\[
\pm\hat{x}\sqrt{\frac{2c}{2c+M\bar{\sigma}^{2}}}.
\]

We point out here that a natural scaling for this problem, given that the size
of the neighborhood of zero where the quasipotential-based subsolution fails
for the $\varepsilon>0$ system scales like $\sqrt{\varepsilon}$, is to ask
that this width scale as $2\varepsilon^{\kappa},\kappa\in [0,1/2]$. If,
for example, the desired width is $2\varepsilon^{1/4}$, then when
$\varepsilon>0$ is small, we should take $M\approx2\hat{x}^{2}%
c/\bar{\sigma}^{2}\varepsilon^{1/2}$.

The next adaptation is required so that singularities in the control
associated with $F_{2,\pm}^{M}(t,x)$ as $t\uparrow T$ do not cause a problem,
as well as for purposes of localization. For a parameter $t^{\ast}\geq0$, the subsolution property will
require $U^{0}(T-t^{\ast},x)\leq F_{1}(x)$ for all $x\in [-A,A]$, where
(with an abuse of notation) $U^{0}(t,x)\doteq F_{2,+}^{M}(t,x)\wedge U_{+}%
^{0}(t,x)\wedge F_{2,-}^{M}(t,x)\wedge U_{-}^{0}(t,x)$. This is true if and
only if the smaller solution to $F_{2,+}^{M}(T-t^{\ast},x)=F_{1}(x)$ is less
than or equal to zero. This root was found to be
\[
\frac{\bar{\sigma}^{2}\hat{x}}{K} \biggl( e^{-ct^{\ast}}-\sqrt{\frac{2cK}%
{M\bar{\sigma}^{4}}-
\frac{2c}{M\bar{\sigma}^{2}}e^{-2ct^{\ast}}} \biggr),
\]
and the restriction that this be nonpositive can be simplified to
\[
t^{\ast}\geq-\frac{1}{2c}\log\frac{2c}{M\bar{\sigma}^{2}}.
\]
The inequality $U^{0}(T-t^{\ast},x)\leq F_{1}(x)$ will ensure that the
subsolution property is preserved if we switch from using $U^{0}$ for $t\leq
T-t^{\ast}$ to using $F_{1}$ for $t>T-t^{\ast}$. If $T\leq t^{\ast}$, this
means that we always use $F_{1}$, but our interest here is in large $T$.

We assume that $M\geq4c/\bar{\sigma}^{2}$ so that $t^{\ast}>0$. To guarantee
that the smaller root is strictly negative and to conveniently satisfy a bound
used later, we assume
\begin{equation}\label{Eq:bound_on_tstar}
t^{\ast}\geq-\frac{2}{c}\log\frac{2c}{M\bar{\sigma}^{2}}.
\end{equation}

Besides enforcing the subsolution property across the handoff at time
$T-t^{\ast}$, the selection of $t^{\ast}$ plays a key role in determining the
region used for the localization for the general nonlinear problem. Owing to
the exponential decay, one can confine the localization to a small region with
a modest value of $t^{\ast}$. Suppose we consider confining to a region that
scales as $2\varepsilon^{1/4}$ for all $t\in [0,T-t^{\ast}]$ and
arbitrarily large $T$. As discussed previously,\vspace*{1pt} for small $t$ and large
$T-t^{\ast}$ this suggests that $M$ be approximately $2\hat{x}^{2}%
c/\bar{\sigma}^{2}\varepsilon^{1/2}$, which means that $t^{\ast}\geq-\frac
{2}{c}\log[\varepsilon^{1/2}/\hat{x}^{2}]$. Recalling that the loss in
performance of the importance sampling scheme when the quasipotential-based
subsolution is used scales like $\varepsilon c$, this gives the loss over the interval $[T-t^{\ast},T]$ as scaling like $-\varepsilon
\log[\varepsilon/\hat{x}^{2}]$. We will return to such considerations in the
next section.
\begin{figure}

\includegraphics{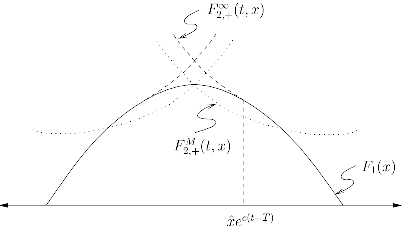}

\caption{Relations between $F_{1},F_{2,\pm}^{\infty}$ and $F_{2,\pm}^{M}$.}
\label{Fig:various_functions}
\end{figure}

Finally there is the issue of mollification. Owing to the replacement of
$F_{2,\pm}^{\infty}(t,x)$ by $F_{2,\pm}^{M}(t,x)$, there are several
sources of discontinuity in the gradient. The subsolution prior to
mollification would generally take the form $F_{2,+}^{M}(t,x)\wedge U_{+}%
^{0}(t,x)\wedge F_{2,-}^{M}(t,x)\wedge U_{-}^{0}(t,x)$. Since we know
that $F_{2,+}^{M}(t,x)\wedge F_{2,-}^{M}(t,x)\leq F_{1}(x)$ near
$x=0$ (see Figure~\ref{Fig:various_functions}), we can instead use the simpler expression
\[
F_{2,+}^{M}(t,x)\wedge F_{2,-}^{M}(t,x)
\wedge F_{1}(x).
\]

We next state a result that will be used to derive performance bounds for
schemes based on the mollification. The proof is deferred to the \hyperref[app]{Appendix}. We
consider the general one-dimensional process model
\[
dX^{\varepsilon}=b\bigl(X^{\varepsilon}\bigr)\,dt+\sqrt{\varepsilon}\sigma
\bigl(X^{\varepsilon
}\bigr)\,dB,
\]
where $b$ and $\sigma$ are Lipschitz continuous. Letting $\alpha
(x)=\sigma(x)^{2}$, the relevant \mbox{$\varepsilon$-dependent} PDE is
\begin{eqnarray*}
\mathcal{G}^{\varepsilon} [ U ] (t,x)&=& U_{t}(t,x)+DU(t,x)b(x)-
\frac{1}{2}\bigl\llvert \sigma(x)DU(t,x)\bigr\rrvert ^{2}+
\frac{\varepsilon}{2}%
\alpha(x)D^{2}U(t,x)\\
&=& 0.
\label{Eq:eps}%
\end{eqnarray*}

\begin{lem}
\label{Lem:mollification}
Suppose that the functions $\tilde{U}_{i}%
\dvtx [0,T]\times\mathbb{R}$, $i=1,\ldots,n$ are twice continuously differentiable
in $x$, once continuously differentiable in $t$ and satisfy
\[
\mathcal{G}^{\varepsilon}[\tilde{U}_{i}](t,x)\geq
\gamma_{i}(x,t,\varepsilon)
\]
for given lower bounds $\gamma_{i}(x,t,\varepsilon)$. For $\delta>0$ define
\[
U^{\delta}(t,x)=-\delta\log \Biggl( \sum_{i=1}^{n}e^{-({1}/{\delta})\tilde
{U}_{i}(t,x)}
\Biggr), \label{Eq:2-moll}%
\]
and define the weights
\[
\rho_{i} ( t,x;\delta ) =\frac{e^{-({1}/{\delta})\tilde{U}%
_{i}(t,x)}}{\sum_{i=1}^{n}e^{-({1}/{\delta})\tilde{U}_{i}(t,x)}}. \label{Eq:2-moll-weights}%
\]
Then
\[
\min \bigl\{ \tilde{U}_{i}(t,x),i=1,\ldots,n \bigr\} \geq
U^{\delta}
(t,x)\geq\min \bigl\{ \tilde{U}_{i}(t,x),i=1,
\ldots,n \bigr\} -\delta\log n, \label{Eq:2-moll-ineq}%
\]
and for $0<\varepsilon\leq\delta$
\begin{eqnarray*}
 \mathcal{G}^{\varepsilon}\bigl[U^{\delta}\bigr](t,x)
& \geq &
\frac{1}{2} \biggl( 1-\frac{\varepsilon}{\delta} \biggr) \Biggl[ \sum
_{i=1}^{n}\rho_{i} ( t,x;\delta ) \bigl
\llvert \sigma(x)D\tilde{U}_{i}(t,x)\bigr\rrvert ^{2}\\
&&\hspace*{9pt}\qquad\qquad{}-\Biggl
\llvert \sum_{i=1}^{n}\rho_{1} (
t,x;\delta ) \sigma(x)D\tilde{U}_{i}(t,x)\Biggr\rrvert ^{2}
\Biggr]
\\
&&{}+\sum_{i=1}^{n}\rho_{i} (
t,x;\delta ) \gamma _{i}(x,t,\varepsilon)
\\
&\geq&  \sum_{i=1}^{n}\rho_{i}
( t,x;\delta ) \gamma_{i}%
(x,t,\varepsilon).
\end{eqnarray*}
\end{lem}

Based on this result we consider the mollification of $U^{0}(t,x)\doteq
F_{2,+}^{M}(t,x)\wedge F_{2,-}^{M}(t,x)\wedge F_{1}(x)$,
\[
U^{\delta}(t,x)=-\delta\log \bigl( e^{-({1}/{\delta})F_{2,+}^{M}%
(t,x)}+e^{-({1}/{\delta})F_{2,-}^{M}(t,x)}+e^{-({1}/{\delta})F_{1}%
(x)}
\bigr).
\]
For reasons that will be made clear in the analysis (see Lemma
\ref{L:Region_y_hatx}), we generally use $\delta=2\varepsilon$.

As discussed previously, we return to the control based on the quasipotential
for the last $t^{\ast}$ units of time, and so the subsolution on the whole domain $[-A,A] \times [0,T]$ takes the form%
\begin{equation}
\label{Eq:ControlDesignLinear}
\bar{U}^{\delta}(t,x)=\cases{ F_{1}(x), &\quad
$t>T-t^{\ast}$,\vspace*{3pt}
\cr
U^{\delta}(t,x), &\quad $t\leq
T-t^{\ast}$.}
\end{equation}
Note that the mollification reduces values, in that $U^{\delta}(t,x)\leq
U^{0}(t,x)$, and so the requirement $U^{\delta}(T-t^{\ast},x)\leq F_{1}(x)$
for $x\in [-A,A]$ holds, since $U^{0}(T-t^{\ast},x)\leq F_{1}(x)$.

\subsection{Analysis of the scheme}
\label{S:AnalysisLinearProblem}

We next present a rigorous and nonasymptotic bound for the second moment of
the importance sampling scheme constructed in the last subsection. To derive a
bound that is valid for $\varepsilon>0$ and uniform in $T$, we use the same
representation as in (\ref{Eq:LBQ}). In particular, we choose $\bar
{U}(t,x)=\bar{U}^{\delta}(t,x)$ defined via (\ref{Eq:ControlDesignLinear}) for
the design and $\bar{W}(t,x)=\bar{U}^{\delta,\eta}(t,x)$ for the analysis, where
\[
\bar{U}^{\delta,\eta}(t,x)=\cases{ F_{1}(x), &\quad$t>T-t^{\ast}$,
\vspace*{3pt}
\cr
U^{\delta,\eta}(t,x), &\quad $t\leq T-t^{\ast}$,}
\label{Eq:ControlAnalysisLinear}
\]
with
\[
U^{\delta,\eta}(t,x)=(1-\eta)U^{\delta}(t,x).
\]
As discussed in Section~\ref{S:EffectivenessOfSubsolutions} this choice is driven by the need for a subsolution
for the $\varepsilon>0$ dynamics with an explicit form, and this limitation of
the technique leads to conservative bounds on the true performance. To
simplify notation, for smooth functions $W,U$ we define
%
\begin{equation}\label{Eq:Ge_2}
\mathcal{G}^{\varepsilon}[W,U](t,x)=\mathcal{G}^{\varepsilon}[W](t,x)-
\tfrac{1}{2}\bigl\llvert \bar{\sigma} \bigl( DW(t,x)-DU(t,x) \bigr) \bigr
\rrvert ^{2},
\end{equation}
where $\mathcal{G}^{\varepsilon}[W]$ is defined right after (\ref{Eq:game_rep}). Using that
$\mathcal{G}^{0}[F_{1}]=\mathcal{G}^{0}[F_{2}^{M}]=0$,
\begin{eqnarray*}
\mathcal{G}^{\varepsilon}[F_{1}](x) &=& \frac{\varepsilon}{2}\bar{
\sigma}^{2}%
D^{2}F_{1}(x)=-\varepsilon c
\quad\mbox{and}\\
\mathcal{G}^{\varepsilon}\bigl[F_{2}^{M}
\bigr](t,x) &=& \frac{\varepsilon}{2}\bar{\sigma}^{2}D^{2}F_{2}^{M}%
(t,x)=
\varepsilon\bar{\sigma}^{2}a^{M}(t). \label{Eq:values_for_pieces}%
\end{eqnarray*}
Since the problem is symmetric, it suffices to consider only $x\in [
0,A]$. A key role is played by the weights associated with the exponential
mollification of the subsolution, which take the forms
\[
\rho_{2}^{M,\pm}(t,x;\delta)=\frac{e^{-({1}/{\delta})F_{2,\pm}^{M}(t,x)}%
}{e^{-({1}/{\delta})F_{2,+}^{M}(t,x)}+e^{-({1}/{\delta})F_{2,-}^{M}%
(t,x)}+e^{-({1}/{\delta})F_{1}(x)}}%
\]
and
\[
\rho_{1}(t,x;\delta)=\frac{e^{-({1}/{\delta})F_{1}(x)}}{e^{-({1}/{\delta})F_{2,+}^{M}(t,x)}+e^{-({1}/{\delta})F_{2,-}^{M}(t,x)}+e^{-({1}/{\delta})F_{1}(x)}}.
\]

To determine which of these dominate at any $(t,x)$, the relative sizes of the
functions $F_{2,\pm}^{M}(t,x)$ and $F_{1}(x)$ are required, with smaller
functions corresponding to more dominant weights. For this\vspace*{1pt} reason the
solutions to $F_{2,\pm}^{M}(t,x)=F_{1}(x)$ play an important role, and
especially the larger one identified in (\ref{Eq:roots}); see Figure~\ref{Fig:various_functions}. The following bounds on this root will be used to
partition the domain in the analysis. Let $z\doteq$ $\hat{x}(c/M\bar{\sigma
}^{2})^{1/2}/2$, and let $R(t)$ denote the larger root in (\ref{Eq:roots}).
Then we claim under (\ref{Eq:bound_on_tstar}) that
\[
2z\leq R(t)\leq8z \qquad \mbox{for all }t\in \bigl[0,T-t^{\ast}\bigr]
\]
and uniformly in $T$. We recall the definition $K\doteq(2c/M)+\bar{\sigma}%
^{2}\in(\bar{\sigma}^{2},\infty)$ and that $M\geq4c/\bar{\sigma}^{2}$, so that
$K\in(\bar{\sigma}^{2},2\bar{\sigma}^{2})$. Then the smallest possible value
of $R(t)$ satisfies
\[
\frac{\bar{\sigma}^{2}\hat{x}}{K}\sqrt{\frac{2cK}{M\bar{\sigma}^{4}}}%
>\frac{\bar{\sigma}^{2}\hat{x}}{\sqrt{2}\bar{\sigma}}\sqrt{
\frac{2c}%
{M\bar{\sigma}^{4}}}=\hat{x}\sqrt{\frac{c}{M\bar{\sigma}^{2}}}=2z,\
\]
while the largest satisfies
\begin{eqnarray*}
\frac{\bar{\sigma}^{2}\hat{x}}{K} \biggl( e^{-ct^{\ast}}+\sqrt{\frac{2cK}%
{M\bar{\sigma}^{4}}-
\frac{2c}{M\bar{\sigma}^{2}}e^{-2ct^{\ast}}} \biggr) & \leq & \frac{\bar{\sigma}^{2}\hat{x}}{K} \biggl(
\frac{2c}{M\bar{\sigma}^{2}%
}+\sqrt{\frac{2cK}{M\bar{\sigma}^{4}}} \biggr)\\
& \leq & \frac{\bar{\sigma}^{2}\hat{x}}{\bar{\sigma}^{2}}
\biggl( \frac{2c}%
{M\bar{\sigma}^{2}}+\sqrt{\frac{4c}{M\bar{\sigma}^{2}}} \biggr)
\\
& <& \hat{x} \biggl( 4\sqrt{\frac{c}{M\bar{\sigma}^{2}}} \biggr) =8z,
\end{eqnarray*}
where the first inequality uses $t^{\ast}\geq-\frac{2}{c}\log\frac{2c}%
{M\bar{\sigma}^{2}}$. We set $H\doteq10z>8z$.

In order to obtain bounds on the performance under the corresponding scheme,
we need to bound $\mathcal{G}^{\varepsilon}[U^{\delta,\eta},U^{\delta}](t,x)$
from below in various regions. By (\ref{Eq:Ge_2})
\begin{eqnarray}
\qquad\mathcal{G}^{\varepsilon}\bigl[U^{\delta,\eta},U^{\delta}\bigr](t,x) & =&
\mathcal{G}%
^{\varepsilon}\bigl[U^{\delta,\eta}\bigr](t,x)-
\tfrac{1}{2}\bigl\llvert \bar{\sigma} \bigl( DU^{\delta,\eta}(t,x)-DU^{\delta}(t,x)
\bigr) \bigr\rrvert ^{2}
\nonumber
\nonumber
\\[-8pt]
\label{Eq:CombinedExpression}
\\[-8pt]
\nonumber
& =& \mathcal{G}^{\varepsilon}\bigl[U^{\delta,\eta}\bigr](t,x)-\tfrac{1}{2}
\eta ^{2}\bigl\llvert \bar{\sigma}DU^{\delta}(t,x)\bigr\rrvert
^{2}.
\end{eqnarray}
We will use the notation
\[
\gamma_{1}=\mathcal{G}^{\varepsilon}[F_{1}](x)=-
\varepsilon c  \quad\mbox{and}\quad \gamma_{2}^{M}(t)=
\mathcal{G}^{\varepsilon}\bigl[F_{2}^{M}\bigr](t,x)=
\varepsilon \bar{\sigma}^{2}a^{M}(t). \label{Eq:def_gamM}%
\]
Straightforward calculations and some algebra give
\[
\mathcal{G}^{\varepsilon}\bigl[U^{\delta,\eta}\bigr](t,x)\geq(1-\eta)
\mathcal{G}%
^{\varepsilon}\bigl[U^{\delta}\bigr](t,x)+
\tfrac{1}{2} \bigl( \eta-\eta^{2} \bigr) \bigl\llvert \bar{
\sigma}DU^{\delta}(t,x)\bigr\rrvert ^{2}.
\]
For notational convenience, define
\begin{eqnarray*}
\beta_{0}(t,x) & =& \bar{\sigma}^{2} \bigl[
\rho_{2}^{M,+}\bigl|DF_{2,+}^{M}\bigr|^{2}+
\rho_{2}^{M,-}\bigl|DF_{2,-}
^{M}\bigr|^{2}+
\rho_{1}|DF_{1}|^{2}
\\
&&\hspace*{25pt}{}-\bigl\llvert \rho_{2}^{M,+}DF_{2,+}^{M}+
\rho_{2}^{M,-}DF_{2,-}^{M}+
\rho_{1}DF_{1}\bigr\rrvert ^{2} \bigr] (t,x).
\end{eqnarray*}
Note that by Jensen's inequality $\beta_{0}(t,x)\geq0$. We next apply Lemma
\ref{Lem:mollification} to $\mathcal{G}^{\varepsilon}[{U}^{\delta,\eta}](t,x)$
(while suppressing the dependence on $\delta$ in the notation for the $\rho
$'s) and use (\ref{Eq:CombinedExpression}) to get
\begin{eqnarray}
&& \mathcal{G}^{\varepsilon}\bigl[U^{\delta,\eta},U^{\delta}\bigr](t,x)\nonumber\\
 &&\label{Eq:bound_for_moll}\qquad \geq   (1-\eta)\frac{1}{2} \biggl( 1-\frac{\varepsilon}{\delta} \biggr) \beta
_{0}(t,x)+(1-\eta) \bigl[ \rho_{2}^{M,+}(t,x)+
\rho_{2}^{M,-}(t,x) \bigr] \gamma_{2}^{M}(t)\hspace*{-18pt}
\\
\nonumber
&&\qquad\quad{}+(1-\eta)\rho_{1}(t,x)\gamma_{1}+\frac{1}{2}
\bigl( \eta-2\eta ^{2} \bigr) \bigl\llvert \bar{\sigma}DU^{\delta}(t,x)
\bigr\rrvert ^{2}
\end{eqnarray}
for all $x\in [-A,A]$ and $t\in [0,T-t^{\ast}]$.

We will partition the domain according $z\doteq$ $\hat{x}(c/M\bar{\sigma}%
^{2})^{1/2}/2$ and $H\doteq10z$. We consider three cases depending on whether
$x\in [0,z]$, $x\in [ z,H]$ or $x\in [ H,A]$ if $x\geq0$. The
case $x<0$ is symmetric. Before proceeding with the analysis for each of the
cases, we give the definition of exponential negligibility, a concept used
frequently in the rest of the paper.

\begin{definition}
\label{D:ExponentialNegligibility}
A term is called \textit{exponentially negligible} if
it is bounded above in absolute value by a quantity of the form $\varepsilon
ce^{-{d}/{\varepsilon}}$, where $c<\infty$ and $d>0$.
\end{definition}

\begin{lemma}
\label{L:Region_0_y}
Assume that $(t,x)\in [0,T-t^{\ast}]\times
[0,z]$, $\delta\geq\varepsilon$ and $\eta\leq1/2$. Then, up to an
exponentially negligible term,
\[
\mathcal{G}^{\varepsilon}\bigl[U^{\delta,\eta},U^{\delta}\bigr](t,x)
\geq0.
\]
\end{lemma}

\begin{pf}
In this region $F_{1}(x)\geq F_{2,+}^{M}(t,x)$, and we claim that the
inequality is in fact strict. We have
\begin{equation}\label{Eq:F1-F2}
\quad\hspace*{3pt} F_{1}(x)-F_{2,+}^{M}(t,x)=\frac{c}{\bar{\sigma}^{2}}
\bigl[ \hat{x}^{2}
-x^{2} \bigr] -\frac{c}{K-\bar{\sigma}^{2}e^{2c(t-T)}}
\bigl[ xe^{c(t-T)}%
-\hat{x} \bigr] ^{2}.
\end{equation}
For each fixed $t$ this defines a concave function of $x$. At $x=0$ the value
is minimized when $t=T-t^{\ast}$. Using $e^{-2ct^{\ast}}\leq [
2c/M\bar{\sigma}^{2}]^{4}\leq c/M\bar{\sigma}^{2}$ (since $M\geq4c/\bar
{\sigma}^{2}$) and $K=\bar{\sigma}^{2}+2c/M$, we obtain the strictly positive
lower bound $\hat{x}^{2}c [  1/\bar{\sigma}^{2}-1/[\bar{\sigma}%
^{2}+c/M] ]  $. Since $F_{1}(2z)-F_{2,+}^{M}(t,2z)\geq0$, by\vspace*{1pt} concavity
there is $c_{1}>0$ such that (\ref{Eq:F1-F2}) is bounded below by $c_{1}$ for
all $(t,x)\in [0,T-t^{\ast}]\times [0,z]$. Thus the term in
(\ref{Eq:bound_for_moll}) involving the weight $\rho_{1}$ is exponentially
negligible. Since $\beta_{0}(t,x)\geq0$, $\gamma_{2}^{M}(t)>0$ and $\eta
\leq1/2$, all other terms are nonnegative, and the result follows.
\end{pf}

\begin{lemma}
\label{L:Region_hatx_A}
Assume that $(t,x)\in [0,T-t^{\ast}]\times [
H,A]$, $\delta\geq\varepsilon$ and $\eta\leq1/4$. Then letting $\varepsilon
_{0}\doteq cH^{2}/3\bar{\sigma}^{2}$, we have that for all $\varepsilon
\in(0,\varepsilon_{0})$ and with any $\eta\in [\varepsilon/(\varepsilon
+cH^{2}/\bar{\sigma}^{2}),1/4]$, up to an exponentially negligible term,
\begin{equation}\label{Eq:Bound2_b}
\mathcal{G}^{\varepsilon}\bigl[U^{\delta,\eta},U^{\delta}\bigr](t,x)
\geq0.
\end{equation}
\end{lemma}

\begin{pf}
In this region $F_{2,+}^{M}(t,x)\geq F_{1}(x)$, and it is straightforward that
the terms associated with $F_{2,-}^{M}(t,x)$ are exponentially negligible.
Note that $x\rightarrow F_{2,+}^{M}(t,x)-F_{1}(x)$ is\vspace*{1pt} convex, recall that for
each $t\in [0,T-t^{\ast}]$ the largest value where the two functions
agree is smaller than $8z\leq H$ and that
\begin{equation}\label{Eq:diff_of_deriv}
\hspace*{3pt}DF_{2,+}^{M}(t,x)-DF_{1}(x)=\frac{2c}{\bar{\sigma}^{2}}
\frac{K}{K-\bar{\sigma
}^{2}e^{2c(t-T)}} \biggl( x-\frac{\bar{\sigma}^{2}}{K}\hat{x}e^{c(t-T)} \biggr).
\end{equation}
Inserting the largest root for the given $t$ gives the value
\[
\frac{2c\hat{x}}{K-\bar{\sigma}^{2}e^{2c(t-T)}}\sqrt{\frac{2cK}{M\bar{\sigma
}^{4}}-\frac{2c}{M\bar{\sigma}^{2}}e^{-2ct^{\ast}}}.
\]
A lower bound on the first term is $2c\hat{x}/K$. Using $e^{-ct^{\ast}}
\leq [2c/M\bar{\sigma}^{2}]^{2}$, the definition of $K$ and
$4c/M\bar{\sigma}^{2}\leq1$ to bound the second term from below produces the
strictly positive lower bound
\[
DF_{2,+}^{M}(t,x)-DF_{1}(x)\geq
\frac{2c\hat{x}}{K}\sqrt{\frac{2c}{M\bar
{\sigma}^{2}} \biggl( \frac{c}{M}+\bar{
\sigma}^{2} \biggr) }
\]
for all $t\in [0,T-t^{\ast}]$ and $x\geq8z$. Since $H=10z>8z$, this shows
that there is \mbox{$c_{2}>0$} such that $F_{2,+}^{M}(t,x)-F_{1}(x)\geq c_{2}$ for
all $(t,x)\in [0,T-t^{\ast}]\times [ H,A]$. It follows that terms
involving $\rho_{2}^{M,\pm}(t,x)$ are exponentially negligible. Since
$\beta_{0}(t,x)\geq0$ and $\rho_{1}(t,x)=1$ up to an exponentially negligible
term,
\begin{eqnarray*}
\mathcal{G}^{\varepsilon}\bigl[U^{\delta,\eta},U^{\delta}\bigr](t,x) &
\geq  & (1-\eta)\rho_{1}(t,x)\gamma^{0}+\frac{1}{2}
\bigl( \eta-2\eta^{2} \bigr) \bar{\sigma}^{2}\bigl\llvert
\rho_{1}(t,x)DF_{1}(x)\bigr\rrvert ^{2}
\\
& \geq & -(1-\eta)\varepsilon c+2 \bigl( \eta-2\eta^{2} \bigr)
\frac{c^{2}%
}{\bar{\sigma}^{2}}x^{2}
\end{eqnarray*}
up to an exponentially negligible term. Choosing $\eta\leq1/4$ gives
\[
\mathcal{G}^{\varepsilon}\bigl[U^{\delta,\eta},U^{\delta}\bigr](t,x)
\geq-(1-\eta )\varepsilon c+\eta\frac{c^{2}}{\bar{\sigma}^{2}}x^{2},
\]
and for $\varepsilon$ small enough such that $\eta\in [\varepsilon
/(\varepsilon+cH^{2}/\bar{\sigma}^{2}),1/4]$, the last display is
nonnegative. For this interval to be nonempty imposes the constraint
$\varepsilon\leq\varepsilon_{0}\doteq cH^{2}/3\bar{\sigma}^{2}$. Hence in this
region and up to an exponentially negligible term,
\[
\mathcal{G}^{\varepsilon}\bigl[U^{\delta,\eta},U^{\delta}\bigr](t,x)
\geq0.
\]
\upqed\end{pf}

The final region is the most difficult, since $F_{1}(x)-F_{2,+}^{M}(t,x)$ can
be either positive or negative.

\begin{lemma}
\label{L:Region_y_hatx} Assume that $(t,x)\in [0,T-t^{\ast}]\times [
z,H]$, $\eta\leq1/4$, and set $\delta=2\varepsilon$. Then up to an
exponentially negligible term,
\[
\mathcal{G}^{\varepsilon}\bigl[U^{\delta,\eta},U^{\delta}\bigr](t,x)\geq
\frac{1}%
{2} \biggl[ \frac{c^{2}\eta}{2\bar{\sigma}^{2}} \bigl( z-\hat{x}e^{c(t-T)}%
 \bigr) ^{2}-2\varepsilon c \biggr] \wedge0. \label{Eq:Bound2_a}%
\]
\end{lemma}

\begin{pf}
While terms corresponding to $F_{2,-}^{M}(t,x)$ are exponentially negligible
in this region, since
\[
F_{1}(x)-F_{2,+}^{M}(t,x)
\]
changes sign, both $\rho_{2}^{M,+}$ and $\rho_{1}$ may be important. Since they
are negligible we omit terms corresponding to $F_{2,-}^{M}(t,x)$.

By (\ref{Eq:bound_for_moll}) we have up to an exponentially negligible term
\begin{eqnarray}
&& \quad\mathcal{G}^{\varepsilon}\bigl[U^{\delta,\eta},U^{\delta}\bigr](t,x)\nonumber\\
 && \quad\label{Eq:BoundForGreyArea1}\qquad \geq
 (1-\eta)\tfrac{1}{4}\beta_{0}(t,x)+(1-\eta)
\rho_{2}^{M,+}(t,x)\gamma_{2}%
^{M}(t)+(1-
\eta)\rho_{1}(t,x)\gamma_{1}\hspace*{-25pt}
\\
&&\qquad\qquad{}+\tfrac{1}{2} \bigl( \eta-2\eta^{2} \bigr) \bar{
\sigma}^{2}\bigl\llvert \rho_{2}^{M,+}(t,x)DF_{2}^{M}(t,x)+
\rho_{1}(t,x)DF_{1}(x)\bigr\rrvert ^{2}.\nonumber
\end{eqnarray}
As noted previously $\beta_{0}(t,x)\geq0$ for all $(t,x)$. However, we will
exploit the fact that $\beta_{0}(t,x)=0$ only for points $(t,x)$ such that
$DF_{1}(x)=DF_{2}^{M}(t,x)$. We distinguish two cases depending on whether
$\rho_{1}(t,x)>1/2$ or $\rho_{1}(t,x)\leq1/2$.

\begin{longlist}[Case II:]
\item[\textit{Case I}:] $\rho_{1}(t,x)>1/2$. We know that $\beta_{0}(t,x)\geq0$ and
$\gamma_{2}^{M}(t)\geq0$ and can ignore those terms. Using $\rho_{2}%
^{M,+}+\rho_{1}=1$, the terms that remain are
\begin{eqnarray*}
&& (1-\eta)\rho_{1}(t,x)\gamma_{1}
\\
&& \quad{}+\tfrac{1}{2} \bigl(
\eta-2\eta^{2} \bigr) \bar{\sigma}^{2} \bigl[ \rho_{1}(t,x)^{2}%
\bigl|DF_{1}(x)-DF_{2,+}^{M}(t,x)\bigr|^{2}
\\
&& \hspace*{45pt}\qquad\qquad{}+2\rho_{1}(t,x)DF_{2,+}^{M}(t,x)
\bigl(DF_{1}%
(x)-DF_{2,+}^{M}(t,x)\bigr)
\\
&&\hspace*{183pt}\qquad\qquad{}+\bigl|DF_{2,+}^{M}(t,x)\bigr|^{2} \bigr].
\end{eqnarray*}
We claim that for $(t,x)\in [0,T-t^{\ast}]\times [ z,H]$,
\[
DF_{2,+}^{M}(t,x) \bigl(DF_{1}(x)-DF_{2,+}^{M}(t,x)
\bigr)\geq0.
\]
First, we note that $e^{-ct^{\ast}}\leq(2c/M\bar{\sigma}^{2})^{1/2}%
=(4c/M\bar{\sigma}^{2})^{1/2}/4\leq1/4$, and thus $\hat{x}e^{ct^{\ast}}%
\geq4\hat{x}\geq H$. Therefore
\begin{eqnarray*}
DF_{2,+}^{M}(t,x)  &=& \frac{2ce^{2c(t-T)}}{K-\bar{\sigma}^{2}e^{2c(t-T)}%
} \bigl( x-
\hat{x}e^{-c(t-T)} \bigr) \\
&\leq & \frac{2ce^{2c(t-T)}}{K-\bar{\sigma}^{2}e^{2c(t-T)}} \bigl( H-\hat
{x}e^{ct^{\ast}} \bigr) \leq0.
\end{eqnarray*}

Second, by (\ref{Eq:diff_of_deriv}), the definition of $z$ and $e^{-ct^{\ast}%
}\leq(2c/M\bar{\sigma}^{2})^{2}$, we also have
\begin{eqnarray*}
DF_{1}(x)-DF_{2,+}^{M}(t,x) & =&
\frac{2c}{\bar{\sigma}^{2}}\frac{K}
{K-\bar{\sigma}^{2}e^{2c(t-T)}} \biggl( \frac{\bar{\sigma}^{2}}{K}\hat
{x}e^{c(t-T)}-x \biggr)
\\
& \leq & \frac{2c}{\bar{\sigma}^{2}}\frac{K}{K-\bar{\sigma}^{2}e^{2c(t-T)}%
} \biggl( \frac{\bar{\sigma}^{2}}{K}
\hat{x}e^{-ct^{\ast}}-z \biggr)
\\
& \leq & \frac{2c}{\bar{\sigma}^{2}}\frac{K}{K-\bar{\sigma}^{2}e^{2c(t-T)}%
} \biggl( \hat{x} \biggl(
\frac{2c}{M\bar{\sigma}^{2}} \biggr) ^{2}-\frac
{\hat{x}}{2} \biggl(
\frac{c}{M\bar{\sigma}^{2}} \biggr) ^{1/2} \biggr)
\\
& \leq & 0,
\end{eqnarray*}
where the last inequality uses $4c/M\bar{\sigma}^{2}\leq1$. We conclude that
$DF_{2}^{M}(t,x) \times\break (DF_{1}(x)-DF_{2}^{M}(t,x))\geq0$. Since $\rho_{1}%
(t,x)\in(1/2,1)$ and $\eta\leq1/4$, we obtain the bound
\begin{equation}\label{Eq:CaseI}
\hspace*{10pt}\quad\mathcal{G}^{\varepsilon}\bigl[U^{\delta,\eta},U^{\delta}\bigr](t,x)
\geq-(1-\eta )\varepsilon c+\tfrac{1}{16}\eta\bar{\sigma}^{2}\bigl|DF_{1}(x)-DF_{2,+}
^{M}(t,x)\bigr|^{2}.
\end{equation}
This gives a bound for Case~I.

\item[\textit{Case II}:]  $\rho_{1}(t,x)\leq1/2$. Here we will have to use $\beta
_{0}(t,x)$. Dropping other terms on the right that are not possibly negative,
we obtain from (\ref{Eq:BoundForGreyArea1}) that
\[
\mathcal{G}^{\varepsilon}\bigl[U^{\delta,\eta},U^{\delta}\bigr](t,x)
\geq(1-\eta)\tfrac{1}{4}\beta_{0}(t,x)+(1-\eta)
\rho_{1}(t,x)\gamma_{1}.
\]
Omitting exponentially negligible terms, we note\vspace*{-1pt} that
\begin{eqnarray*}
\beta_{0}(t,x) & =& \bar{\sigma}^{2} \bigl[ (1-
\rho_{1})\bigl\llvert DF_{2,+}^{M}\bigr\rrvert
^{2}+\rho_{1}\llvert DF_{1}\rrvert
^{2}\\[-2pt]
&&\hspace*{6pt}\quad{}-\bigl\llvert DF_{2,+}^{M}+\rho_{1}
\bigl(DF_{1}-DF_{2,+}^{M}\bigr)\bigr\rrvert
^{2} \bigr] (t,x)
\\[-2pt]
& =& \bar{\sigma}^{2}\rho_{1} \bigl[ \llvert
DF_{1}\rrvert ^{2}-\bigl\llvert DF_{2,+}^{M}
\bigr\rrvert ^{2}-2DF_{2,+}^{M}
\bigl(DF_{1}-DF_{2,+}^{M}\bigr)\\
&&\hspace*{111pt}\qquad{}-
\rho_{1}\bigl(DF_{1}-DF_{2,+}^{M}
\bigr)^{2} \bigr] (t,x)
\\[-2pt]
& =& \bar{\sigma}^{2}\rho_{1}\bigl(DF_{1}-DF_{2,+}^{M}
\bigr) \bigl[ DF_{1}+DF_{2,+}%
^{M}-DF_{2,+}^{M}
\\[-2pt]
&&\hspace*{88pt}\qquad{}
-\rho_{1}\bigl(DF_{1}-DF_{2,+}^{M}\bigr)
\bigr] (t,x)
\\[-2pt]
& =& \bar{\sigma}^{2}\rho_{1}(1-\rho_{1})
\bigl(DF_{1}-DF_{2,+}^{M}\bigr)^{2}(t,x)
\\[-2pt]
& \geq & \tfrac{1}{2}\bar{\sigma}^{2}\rho_{1}
\bigl(DF_{1}-DF_{2,+}^{M}\bigr)^{2}(t,x),
\end{eqnarray*}
where the last inequality uses $\rho_{1}(t,x)\leq1/2$, and so\vspace*{-2pt} obtain
\[
\mathcal{G}^{\varepsilon}\bigl[U^{\delta,\eta},U^{\delta}\bigr](t,x)
\geq(1-\eta)\rho _{1} \bigl[ \tfrac{1}{8}\bar{
\sigma}^{2}\bigl|DF_{1}(x)-DF_{2,+}^{M}(t,x)\bigr|^{2}
-
\varepsilon c \bigr].
\]
This gives a bound for Case II.

Straightforward estimation gives, for all $(t,x)\in [0,T-t^{\ast}]\times [ z,H]$, the lower bound\vspace*{-3pt}
\begin{eqnarray*}
DF_{2,+}^{M}(t,x)-DF_{1}(x)  &=&
\frac{2c}{\bar{\sigma}^{2}}\frac{K}{K-\bar{\sigma}^{2}e^{2c(t-T)}} \biggl( x-\frac{\bar{\sigma}^{2}}{K}\hat
{x}e^{c(t-T)} \biggr)\\
&\geq & \frac{2c}{\bar{\sigma}^{2}} \bigl( z-\hat{x}e^{c(t-T)}
\bigr).
\end{eqnarray*}
Using this bound and $\eta\leq1/4$, we get a lower bound from (\ref{Eq:CaseI})
in the form
\begin{eqnarray*}
\mathcal{G}^{\varepsilon}\bigl[U^{\delta,\eta},U^{\delta}\bigr](t,x) &
\geq&  \biggl[ \frac{1}{16}\eta\bar{\sigma}^{2}\bigl|DF_{1}(x)-DF_{2,+}^{M}(t,x)\bigr|^{2}-
\varepsilon c \biggr]
\\[-2pt]
& \geq &  \biggl[ \frac{c^{2}\eta}{4\bar{\sigma}^{2}} \bigl( z-\hat{x}%
e^{c(t-T)}
\bigr) ^{2}-\varepsilon c \biggr].
\end{eqnarray*}
Since this is less than the bound for Case II when both terms are negative,
the conclusion of the lemma follows.\quad\qed
\end{longlist}
\noqed\end{pf}

\begin{theorem}
\label{T:MainBoundLinear}
Assume $\delta=2\varepsilon,\eta\in(\varepsilon/(\varepsilon+cH^{2}/\bar{\sigma}^{2}),1/4)$ and that $z^{2}c\eta
\geq8\varepsilon\bar{\sigma}^{2}$. Let $\bar{u}$ be the control based on the
function $\bar{U}^{\delta}$ defined in (\ref{Eq:ControlDesignLinear}), that is,
$\bar{u}(t,x)=-\bar{\sigma}D\bar{U}^{\delta}(t,x)$. Then up to an
exponentially negligible term, we have
\[
-\varepsilon\log Q^{\varepsilon}(0,0;\bar{u})\geq2I_{1}(\varepsilon,
\eta,T,\hat{x},M)1_{ \{  T\geq t^{\ast} \}  }+2I_{2}(\varepsilon
,T)1_{ \{  T<t^{\ast} \}},
\]
where\vspace*{-2pt}
\begin{eqnarray*}
I_{1}(\varepsilon,\eta,T,\hat{x},M)  &=& (1-\eta)\bar{U}^{\delta}(0,0)+
\biggl( \log \biggl[ \frac{1}{\hat{x}} \bigl( z-\sqrt{4\varepsilon\bar{
\sigma}
^{2}/c\eta} \bigr) \biggr] \wedge0 \biggr) \varepsilon,
\\[-2pt]
I_{2}(\varepsilon,T) &=& 2L-cT\varepsilon
\end{eqnarray*}
and $L=\frac{1}{2}\frac{c}{\bar{\sigma}^{2}}A^{2}$.
\end{theorem}

\begin{rem}
Although the bound provided by Theorem \ref{T:MainBoundLinear} takes a
complicated form, it is important to note that it does not degrade as
$T\rightarrow\infty$, and this is also reflected in the simulation data. Also,
as noted in Sections~\ref{S:AnalysisLinearProblem} and \ref{SS:1DimLinearModel_algorithm} there are natural scalings under which $\eta\rightarrow0$
and $M\rightarrow\infty$ as $\varepsilon\rightarrow0$. Using the bound from
below given in Lemma \ref{Lem:mollification} and the explicit form of
$F_{2,\pm}^{M}(0,0)$, we obtain
\[
\bar{U}^{\delta}(0,0)\geq\frac{c}{({2c}/{M})+\bar{\sigma}^{2}-\bar{\sigma
}^{2}e^{-2cT}}\hat{x}^{2}+ \biggl(
2L-\frac{c}{\bar{\sigma}^{2}}\hat{x}
^{2} \biggr) -\delta\log3.
\]
If the natural scalings are used, then various terms vanish as $\varepsilon \rightarrow 0$, and we obtain the rate of decay
\[
2L+\frac{c\hat{x}^{2}}{\bar{\sigma}^{2}} \biggl[ \frac{e^{-2cT}}{1-e^{-2cT}
} \biggr]
\]
uniformly in $T$ as $\varepsilon\rightarrow0$.
\end{rem}
\begin{pf*}{Proof of Theorem \ref{T:MainBoundLinear}}
The starting point is
representation (\ref{Eq:game_rep}) (but rewritten for the more general process
model and with time dependent $u$), which is valid for every $\varepsilon>0$.
We can restrict to $v$ such that $\tau^{\varepsilon}\leq T$ w.p.1, obtaining
\begin{eqnarray}
&&\hspace*{8pt} -\varepsilon\log Q^{\varepsilon}(0,0;\bar{u})
\nonumber
\\[-8pt]
\label{Eq:Rep_with_constraint}
\\[-8pt]
\nonumber
&& \hspace*{8pt}\qquad=\inf_{v\in\mathcal{A}\dvtx \hat{\tau
}^{\varepsilon}\leq T \ \mathrm{w.p.1}}
\mathbb{E}_{0,0} \biggl[ \frac{1}{2}\int_{0}^{\hat{\tau}^{\varepsilon}}v(s)^{2}\,ds-
\int_{0}^{\hat{\tau}^{\varepsilon}%
}\bar{u}\bigl(s,\hat{X}^{\varepsilon}(s)
\bigr)^{2}\,ds \biggr].
\end{eqnarray}
We can also assume $T\geq t^{\ast}$ since the bound is straightforward
otherwise. We recall that under (\ref{Eq:bound_on_tstar}) the subsolution
property is preserved for $\bar{U}^{\delta}(t,x)$ at $T-t^{\ast}$, that is, that
\[
U^{\delta}\bigl(T-t^{\ast},x\bigr)\leq F_{1}(x)\qquad\mbox{for all }x\in [-A,A]. \label{Eq:SubsolutionProperty_tstar}%
\]

Next consider any control in the representation such that $\hat{\tau
}^{\varepsilon}\leq T$ w.p.1. We will apply It\^{o}'s formula separately over
the intervals $[0,(T-t^{\ast})\wedge\hat{\tau}^{\varepsilon})$ and
$[(T-t^{\ast})\wedge\hat{\tau}^{\varepsilon},\hat{\tau}^{\varepsilon})$ and
also use the boundary condition $\bar{U}^{\delta,\eta}(t,\pm A)\leq F_{1}(\pm
A)\leq0$ for $t\in [0,T]$. Since $U^{\delta,\eta}(T-t^{\ast},x)\leq
F_{1}(x)$, we obtain
\begin{eqnarray}
\quad-\bar{U}^{\delta,\eta}(0,0) & \geq & \mathbb{E}_{0,0} \bigl[
\bar{U}%
^{\delta,\eta}\bigl(\hat{\tau}^{\varepsilon},
\hat{X}^{\varepsilon}\bigl(\hat{\tau }^{\varepsilon}\bigr)\bigr)\nonumber\\
&&\label{Eq:two_partsA}
\hspace*{22pt}{}-\bar{U}^{\delta,\eta}\bigl(\bigl(T-t^{\ast}\bigr)\wedge\hat{\tau
}^{\varepsilon}-,\hat{X}^{\varepsilon}\bigl(\bigl(T-t^{\ast}\bigr)\wedge
\hat{\tau }^{\varepsilon}-\bigr)\bigr) \bigr] 1_{ \{  \hat{\tau}^{\varepsilon}\geq T-t^{\ast
} \} }
\\
\nonumber
&&{}+\mathbb{E}_{0,0} \bigl[ \bar{U}^{\delta,\eta}\bigl(
\bigl(T-t^{\ast}\bigr)\wedge \hat{\tau}^{\varepsilon},
\hat{X}^{\varepsilon}\bigl(\bigl(T-t^{\ast}\bigr)\wedge\hat{\tau
}^{\varepsilon}\bigr)\bigr)-\bar{U}^{\delta,\eta}(0,0) \bigr].
\end{eqnarray}

Using Lemma \ref{L:GeneralBound} and recalling the definition of
$\mathcal{G}^{\varepsilon}[W,U]$ in (\ref{Eq:Ge_2}), the contribution from
$[0,(T-t^{\ast})\wedge\hat{\tau}^{\varepsilon})$ gives
\begin{eqnarray*}
&&  \mathbb{E}_{0,0} \bigl[ \bar{U}^{\delta,\eta}\bigl(
\bigl(T-t^{\ast}\bigr)\wedge\hat{\tau }^{\varepsilon},
\hat{X}^{\varepsilon}\bigl(\bigl(T-t^{\ast}\bigr)\wedge\hat{\tau
}^{\varepsilon}\bigr)\bigr)-\bar{U}^{\delta,\eta}(0,0) \bigr]
\\
&&\qquad =\mathbb{E}_{0,0}\int_{0}^{(T-t^{\ast})\wedge\hat{\tau}^{\varepsilon}%
}
\biggl[ \mathcal{G}^{\varepsilon}\bigl[\bar{U}^{\delta,\eta},
\bar{U}^{\delta
}\bigr]\bigl(s,\hat{X}^{\varepsilon}(s)\bigr)\,ds\\
&&\hspace*{75pt}\qquad\quad{}-\frac{1}{4}v(s)^{2}\,ds+\frac{1}{2}\bar {u}\bigl(s,
\hat{X}^{\varepsilon}(s)\bigr)^{2} \biggr] \,ds.
\end{eqnarray*}
An analogous formula holds for $[(T-t^{\ast})\wedge\hat{\tau}^{\varepsilon
},\hat{\tau}^{\varepsilon})$, save that since $\bar{U}^{\delta,\eta}=\bar
{U}^{\delta}=F_{1}$, the term $\mathcal{G}^{\varepsilon}[\bar{U}^{\delta,\eta
},\bar{U}^{\delta}]$ simplifies to $\mathcal{G}^{\varepsilon}[F_{1}]$.
Rearranging and using (\ref{Eq:two_partsA}),
\begin{eqnarray*}
&& \mathbb{E}_{0,0}  \biggl[ \frac{1}{2}\int_{0}^{\hat{\tau}^{\varepsilon}%
}v(s)^{2}\,ds-
\int_{0}^{\hat{\tau}^{\varepsilon}}\bar{u}\bigl(s,\hat{X}^{\varepsilon
}(s)
\bigr)^{2}\,ds \biggr]\\
&& \qquad \geq  2\bar{U}^{\delta,\eta}(0,0)
+\mathbb{E}_{0,0}\int_{0}^{(T-t^{\ast})\wedge\hat{\tau}^{\varepsilon}%
}2
\mathcal{G}^{\varepsilon}\bigl[\bar{U}^{\delta,\eta},\bar{U}^{\delta}\bigr]
\bigl(s,\hat {X}^{\varepsilon}(s)\bigr)\,ds\\
&&\qquad\quad{}+\mathbb{E}_{0,0}\int
_{(T-t^{\ast})\wedge\hat{\tau
}^{\varepsilon}}^{\hat{\tau}^{\varepsilon}}1_{ \{  \hat{\tau}%
^{\varepsilon}\geq T-t^{\ast} \}  }2\mathcal{G}^{\varepsilon}%
[F_{1}]
\bigl(s,\hat{X}^{\varepsilon}(s)\bigr)\,ds.
\end{eqnarray*}
We now replace each term by a lower bound, using Lemmas \ref{L:Region_0_y},
\ref{L:Region_hatx_A} and \ref{L:Region_y_hatx} for $2\mathcal{G}%
^{\varepsilon}[\bar{U}^{\delta,\eta},\bar{U}^{\delta}]$. Since the bounds are
independent of the control process $v$, representation
(\ref{Eq:Rep_with_constraint}) implies
\begin{eqnarray*}
&& -\varepsilon\log Q^{\varepsilon}(0,0;\bar{u})\\
&& \qquad \geq2(1-\eta)\bar{U}^{\delta
}(0,0)+
\int_{J} \biggl[ \frac{c^{2}\eta}{2\bar{\sigma}^{2}} \bigl( z-\hat
{x}e^{c(s-T)} \bigr) ^{2}-2\varepsilon c \biggr]
\,ds-t^{\ast}2\varepsilon c,
\end{eqnarray*}
where $J$ are the times in $[0,T-t^{\ast}]$ where the integrand is negative.

We next use the constraint $z^{2}c\eta\geq8\varepsilon\bar{\sigma}^{2}$, which
guarantees that for $T-s$ sufficiently large, the integrand is in fact
positive. Let
\[
\frac{c^{2}\eta}{2\bar{\sigma}^{2}} \bigl( z-\hat{x}e^{-cb} \bigr) ^{2}-2
\varepsilon c\quad \mbox{or}\quad b=-\frac{1}{c}\log \biggl[ \frac{1}{\hat{x}%
}
\bigl( z-\sqrt{4\varepsilon\bar{\sigma}^{2}/c\eta} \bigr) \biggr].
\]

Then since the integrand is only negative for $s\geq T-b$, we obtain the lower
bound $- [  (b-t^{\ast})\vee0 ]  2\varepsilon c$ for the integral.
Adding the remaining $-t^{\ast}2\varepsilon c$ then gives the result as stated.
\end{pf*}

\subsection{Simulation results for the linear problem}
\label{S:SimulationResultsLinearProblem}
In this subsection we\break present simulation data for the linear problem and make
several comments on the application of the algorithm. For comparison purposes,
we consider the same two-sided problem corresponding to the data from Tables~\ref{Table1a}, \ref{Table1b}, \ref{Table2a} and \ref{Table6b}. Thus we
consider the small noise diffusion process with drift $b(x)=-V^{\prime}(x)$,
where $V(x)=\frac{1}{2}x^{2}$, diffusion coefficient $\sqrt{\varepsilon}$ and
starting from the stable equilibrium point $x=0$. The goal is to estimate the
probability of exiting the set $(-1,1)$ by a given time $T$.

As discussed in Section~\ref{SS:1DimLinearModel_algorithm}, the change of
measure for the importance sampling scheme is based on  subsolution
(\ref{Eq:ControlDesignLinear}). In order to apply it to a given pair
$(\varepsilon,T)$, one needs to choose the parameters $(\hat{x},M,t^{\ast
},\delta)$. Before presenting simulation data, we comment on these choices.

The analysis in Section~\ref{S:AnalysisLinearProblem} assumes $t^{\ast}%
\geq-\frac{2}{c}\log\frac{2c}{M\bar{\sigma}^{2}}$ and $\delta=2\varepsilon$,
and we will take $t^{\ast}=-\frac{2}{c}\log\frac{2c}{M\bar{\sigma}^{2}}$. As
noted before Lemmas \ref{L:Region_0_y}--\ref{L:Region_y_hatx}, it is natural to
allow quantities such as $z$ and $H$, which characterize the region where the
solution to the LQR replaces the subsolution based on the quasipotential, to
depend on $\varepsilon$. One would like the width of this region to scale like
$\varepsilon^{\kappa}$, with $\kappa\in(0,1/2]$, which in turn suggests that
$M$ scales like $2\hat{x}^{2}c/\bar{\sigma}^{2}\varepsilon^{2\kappa}$. However,
the exponential negligibility of certain terms that holds when parameters such
as $z,H$ and $M$ are independent of $\varepsilon$ need not hold when they
depend on $\varepsilon$. For example, the exponential negligibility of the
term $(1-\eta)\rho_{1}(t,x)\gamma_{1}$ appearing in Lemma \ref{L:Region_0_y}
should be examined.

Recall that the exponential rate of decay of terms like $(1-\eta)\rho
_{1}(t,x)\gamma_{1}$ is bounded by the smallest value of $F_{1}(x)-F_{2,+}%
^{M}(t,x)$. A lower bound of the form $\hat{x}^{2}c [  1/\bar{\sigma}%
^{2}-1/[\bar{\sigma}^{2}+c/M] ]$ was obtained in the proof of Lemma
\ref{L:Region_0_y}. Inserting the given scaling and approximating for small
$\varepsilon$ gives $c\varepsilon^{2\kappa}/2\bar{\sigma}^{2}$, and upon
dividing by $\delta=2\varepsilon$ gives the exponent $c\varepsilon^{2\kappa
-1}/4\bar{\sigma}^{2}$. Hence exponential negligibility requires $\kappa
\in(0,1/2)$, with smaller values of $\kappa$ giving a faster rate of decay.
Note, however, that the analysis assumes $M\geq4c/\bar{\sigma}^{2}$ and
$z^{2}c\eta\geq8\varepsilon\bar{\sigma}^{2}$. With regard to the condition
$M\geq4c/\bar{\sigma}^{2}$, inserting the given scaling for $M$ we get the
constraint $\hat{x}^{2}/2\varepsilon^{2\kappa}\geq1$. This is clearly
satisfied for small $\varepsilon>0$ if $\hat{x}$ is of order $1$. One may also
take here $\hat{x}$ to be of order $\varepsilon^{\lambda}$, and the constraint
will be satisfied for small $\varepsilon$ if $\lambda<\kappa$. We also remark
here that for the nonlinear problem, the condition $M\geq4c/\bar{\sigma}^{2}$
needs to be strengthened to $M>4c/\bar{\sigma}^{2}$. With regard to the
condition $z^{2}c\eta\geq8\varepsilon\bar{\sigma}^{2}$, inserting the given
scaling for $M$ and recalling the definition $z\doteq\hat{x}(c/M\bar{\sigma
}^{2})^{1/2}/2$, we obtain the constraint $\varepsilon^{2\kappa-1}\geq
64\bar{\sigma}^{2}/\eta c$. This constraint is satisfied for small enough
$\varepsilon$ when $\kappa\in(0,1/2)$, and moreover one can allow
$\eta\rightarrow0$ as $\varepsilon\rightarrow0$.

For the convenience of the reader and for purposes of an easy reference, we present in Table~\ref{TableParameters}
the suggested values for $(\delta,\hat{x},M,t^{*})$, given the value of
the strength of the noise $\varepsilon>0$.
\begin{table}[t]
\caption{Parameter values for the algorithm based on a given value of $\varepsilon>0$}%
\label{TableParameters}
\begin{tabular*}{\tablewidth}{@{\extracolsep{\fill}}lcccc@{}}
\hline
\textbf{Parameter} & $\bolds{\delta}$ & $\bolds{\hat{x}}$ & $\bolds{M}$ & $\bolds{t}^{\bolds{*}}$ \\
\hline
Values & $2\varepsilon$ & $O(1)$ or $\varepsilon^{\lambda}$ with $\lambda<\kappa$ & $\max\{\frac{2c}{\bar{\sigma}^{2}}\frac{\hat{x}^{2}}{\varepsilon^{2\kappa}}, \frac{5c}{\bar{\sigma}^{2}}\}$ with $\kappa\in(0,1/2)$
& $-\frac{2}{c}\log\frac{2c}{M\bar{\sigma}^{2}}$ \\
\hline
\end{tabular*}
\vspace*{-6pt}
\end{table}

Below we present simulation data for various choices of the parameters as
indicated in the corresponding tables. In Table~\ref{Table1a1}, estimated
probabilities are reported when $M=4$ and $\hat{x}=1$, whereas the related
relative error per sample estimates are reported in Table~\ref{Table1b1}
(since the relative errors are consistently smaller we now round to the
nearest $1/10$). In Tables~\ref{Table4b1} and \ref{Table5b1} relative errors
per sample estimates are reported for combinations of $(M,\hat{x})$ that
depend on $\varepsilon$. The related probability estimates are almost
identical to those in Table~\ref{Table1a1}. Note that the degradation in
performance as $T$ gets larger observed previously in Table~\ref{Table6b}, is no longer present.
This agrees with the theoretical performance bound appearing in
Theorem~\ref{T:MainBoundLinear}.

\begin{table}[b]
\tabcolsep=-0pt
\caption{Estimated values for different pairs $(\varepsilon,T)$. $M=4$,
$\hat{x}=1$}
\label{Table1a1}
\begin{tabular*}{\tablewidth}{@{\extracolsep{\fill}}lcccccccc@{}}
\hline
& \multicolumn{8}{c@{}}{$\bolds{T}$}\\[-4pt]
& \multicolumn{8}{c@{}}{\hrulefill} \\
$\bolds{\varepsilon}$ & $\mathbf{1.5}$ & $\mathbf{2.5}$ & $\mathbf{5}$ & $\mathbf{7}$ &
$\mathbf{10}$ & $\mathbf{14}$ & $\mathbf{18}$ & $\mathbf{23}$\\
\hline
$0.20$ & $9.1\mathrm{e}{-}03$ & $2.3\mathrm{e}{-}02$ & $5.7\mathrm{e}{-}02$ & $8.3\mathrm{e}{-}02$ & $1.2\mathrm{e}{-}01$ & $1.7\mathrm{e}{-}01$
& $2.1\mathrm{e}{-}01$ & $2.7\mathrm{e}{-}01$\\
$0.16$ & $2.2\mathrm{e}{-}03$ & $6.6\mathrm{e}{-}03$ & $1.8\mathrm{e}{-}02$ & $2.7\mathrm{e}{-}02$ & $4.0\mathrm{e}{-}02$ & $5.7\mathrm{e}{-}02$
& $7.4\mathrm{e}{-}02$ & $9.5\mathrm{e}{-}02$\\
$0.13$ & $5.1\mathrm{e}{-}04$ & $1.6\mathrm{e}{-}03$ & $4.6\mathrm{e}{-}03$ & $6.9\mathrm{e}{-}03$ & $1.1\mathrm{e}{-}02$ & $1.5\mathrm{e}{-}02$
& $2.0\mathrm{e}{-}02$ & $2.6\mathrm{e}{-}02$\\
$0.11$ & $1.1\mathrm{e}{-}04$ & $3.9\mathrm{e}{-}04$ & $1.2\mathrm{e}{-}03$ & $1.8\mathrm{e}{-}03$ & $2.8\mathrm{e}{-}03$ & $4.1\mathrm{e}{-}03$
& $5.4\mathrm{e}{-}03$ & $7.0\mathrm{e}{-}03$\\
$0.09$ & $1.3\mathrm{e}{-}05$ & $5.2\mathrm{e}{-}05$ & $1.7\mathrm{e}{-}04$ & $2.6\mathrm{e}{-}04$ & $4.1\mathrm{e}{-}04$ & $5.9\mathrm{e}{-}04$
& $7.8\mathrm{e}{-}04$ & $1.0\mathrm{e}{-}03$\\
$0.07$ & $4.3\mathrm{e}{-}07$ & $2.2\mathrm{e}{-}06$ & $7.6\mathrm{e}{-}06$ & $1.2\mathrm{e}{-}05$ & $1.9\mathrm{e}{-}05$ & $2.8\mathrm{e}{-}05$
& $3.7\mathrm{e}{-}05$ & $4.8\mathrm{e}{-}05$\\
$0.05$ & $9.7\mathrm{e}{-}10$ & $6.9\mathrm{e}{-}09$ & $2.8\mathrm{e}{-}08$ & $4.4\mathrm{e}{-}08$ & $7.0\mathrm{e}{-}08$ & $1.1\mathrm{e}{-}07$
& $1.4\mathrm{e}{-}07$ & $1.8\mathrm{e}{-}07$\\
\hline
\end{tabular*}
\end{table}

\begin{table}
\tablewidth=220pt
\caption{Relative errors per sample for different pairs
$(\varepsilon,T)$.
$M=4$ and $\hat{x}=1$}
\label{Table1b1}
\begin{tabular*}{220pt}{@{\extracolsep{\fill}}lcccccccc@{}}
\hline
& \multicolumn{8}{c@{}}{$\bolds{T}$}\\[-4pt]
& \multicolumn{8}{l@{}}{\hrulefill} \\
$\bolds{\varepsilon}$ & $\mathbf{1.5}$ & $\mathbf{2.5}$ & $\mathbf{5}$ & $\mathbf{7}$ & $\mathbf{10}$
& $\mathbf{14}$ & $\mathbf{18}$ & $\mathbf{23}$\\
\hline
$0.20$ & $1.7$ & $\phantom{0}1.4$ & $1.0$ & $0.9$ & $0.6$ & $0.6$ & $0.7$ & $0.9$\\
$0.16$ & $2.1$ & $\phantom{0}1.8$ & $1.2$ & $1.0$ & $0.8$ & $0.7$ & $0.7$ & $0.8$\\
$0.13$ & $2.4$ & $\phantom{0}2.2$ & $1.6$ & $1.4$ & $1.1$ & $0.8$ & $0.8$ & $0.8$\\
$0.11$ & $2.9$ & $\phantom{0}2.7$ & $2.0$ & $1.6$ & $1.3$ & $1.1$ & $1.0$ & $0.9$\\
$0.09$ & $3.6$ & $\phantom{0}3.6$ & $2.7$ & $2.3$ & $1.8$ & $1.5$ & $1.3$ & $1.2$\\
$0.07$ & $4.9$ & $\phantom{0}5.7$ & $4.2$ & $3.4$ & $2.9$ & $2.4$ & $2.1$ & $1.9$\\
$0.05$ & $8.9$ & $13.0$ & $9.9$ & $8.3$ & $6.8$ & $5.7$ & $5.0$ &
$4.4$\\
\hline
\end{tabular*}
\end{table}

\begin{table}[b]
\tablewidth=220pt
\caption{Relative\vspace*{1.5pt} errors per sample for different pairs $(\varepsilon,T)$.
$M=\frac{2}{\sqrt{\varepsilon}}$ and $\hat{x}=1$}
\vspace*{-3pt}
\label{Table4b1}
\begin{tabular*}{220pt}{@{\extracolsep{\fill}}lccccccc@{}}\hline
& \multicolumn{7}{c@{}}{$\bolds{T}$}\\[-4pt]
& \multicolumn{7}{c@{}}{\hrulefill}\\
$\bolds{\varepsilon}$  & $\mathbf{2.5}$ & $\mathbf{5}$ & $\mathbf{7}$ & $\mathbf{10}$ &
$\mathbf{14}$ & $\mathbf{18}$ & $\mathbf{23}$\\
\hline
$0.20$ & $1.3$ & $0.9$ & $0.7$ & $0.6$ & $0.6$ & $0.7$ & $1.0$\\
$0.16$ & $1.5$ & $1.1$ & $0.8$ & $0.7$ & $0.7$ & $0.7$ & $0.9$\\
$0.13$ & $1.7$ & $1.2$ & $1.0$ & $0.8$ & $0.7$ & $0.7$ & $0.8$\\
$0.11$ & $1.8$ & $1.4$ & $1.2$ & $0.9$ & $0.8$ & $0.7$ & $0.8$\\
$0.09$ & $2.0$ & $1.6$ & $1.3$ & $1.1$ & $0.9$ & $0.8$ & $0.8$\\
$0.07$ & $2.2$ & $1.9$ & $1.6$ & $1.3$ & $1.1$ & $1.0$ & $0.9$\\
$0.05$ & $2.4$ & $2.5$ & $2.1$ & $1.7$ & $1.5$ & $1.3$ & $1.1$\\
\hline
\end{tabular*}
\vspace*{-10pt}
\end{table}

\begin{table}
\vspace*{-6pt}
\tablewidth=220pt
\caption{Relative errors per sample for different pairs $(\varepsilon,T)$.
$M=\frac{2}{\varepsilon^{0.3}}$ and $\hat{x}=\varepsilon^{0.15}$}
\vspace*{-3pt}
\label{Table5b1}
\begin{tabular*}{220pt}{@{\extracolsep{\fill}}lccccccc@{}}
\hline
& \multicolumn{7}{c@{}}{$\bolds{T}$}   \\[-4pt]
& \multicolumn{7}{l@{}}{\hrulefill}\\
$\bolds{\varepsilon}$ & $\mathbf{2.5}$ & $\mathbf{5}$ & $\mathbf{7}$ & $\mathbf{10}$ &
$\mathbf{14}$ & $\mathbf{18}$ & $\mathbf{23}$\\
\hline
$0.20$ & $1.5$ & $0.9$ & $0.7$ & $0.7$ & $0.9$ & $1.1$ & $1.3$\\
$0.16$ & $1.7$ & $1.0$ & $0.8$ & $0.7$ & $0.8$ & $1.0$ & $1.2$\\
$0.13$ & $1.8$ & $1.1$ & $0.8$ & $0.8$ & $0.8$ & $1.0$ & $1.0$\\
$0.11$ & $1.9$ & $1.1$ & $0.9$ & $0.7$ & $0.9$ & $0.9$ & $1.1$\\
$0.09$ & $2.2$ & $1.2$ & $0.9$ & $0.8$ & $0.8$ & $0.9$ & $1.1$\\
$0.07$ & $2.4$ & $1.3$ & $1.0$ & $0.8$ & $0.8$ & $0.9$ & $1.1$\\
$0.05$ & $2.9$ & $1.5$ & $1.1$ & $0.9$ & $0.8$ & $0.9$ & $1.0$\\
\hline
\end{tabular*}
\end{table}

\section{The nonlinear one-dimensional problem}
\label{S:NonLinearProblem}

In this section, we extend the construction of Section~\ref{S:LinearProblem}
to the general nonlinear one-dimensional setting. We also generalize the
notation and allow the stable equilibrium to be an arbitrary point $x_{0}$.
Consider the process model (\ref{Eq:Diffusion1}) and assume that $b,\sigma
\in\mathcal{C}^{1}(\mathbb{R})$ and that $b(x_{0})=0$, $b^{\prime}(x_{0})<0$
and $\sigma^{2}(x)\geq\sigma_{1}^{2}>0$ for all $x\in\mathbb{R}$. Thus we can
write $b(x)=-V^{\prime}(x)$ with unique local minimum at $x=x_{0}$ and
$V^{\prime\prime}(x_{0})>0$. It is easy to see that the quasipotential with
respect to the equilibrium point $x_{0}$ takes the form
\[
S(x_{0},x)=\int_{x_{0}}^{x}-2
\frac{b(z)}{\sigma^{2}(z)}%
\,dz.\label{Eq:General1D:Quasipotential}%
\]

The problem of interest is to estimate the exit probability
\[
\theta^{\varepsilon}=\mathbb{P}_{x_{0}} \bigl\{ X^{\varepsilon}
\mbox{ hits }A_{1}\mbox{ or }A_{2}\mbox{ before time }T \bigr\},
\]
where $x_{0}$ is the initial (and rest) point such that $x_{0}\in(A_{1},A_{2})$.
Furthermore, we assume that $b(x)<0$ for all $x\in(x_{0},A_{2}]$ and
$b(x)>0$ for all $x\in[A_{1},x_{0})$. Set $L\doteq\frac{1}{2} [S(x_{0},A_{1})\vee S(x_{0},A_{1}) ]$.

The approach to the nonlinear problem is to merge the linearized dynamics
around the equilibrium point with the subsolution based on the quasipotential.
This subsolution is
\[
\bar{F}_{1}(x)=2 L-S(x_{0},x).
\]
Observe that the second order approximation to this function around the
equilibrium point $x_{0}$ is
\[
F_{1}(x)=2 L-\frac{c}{\bar{\sigma}^{2}}(x-x_{0})^{2},
\]
where $c=-b^{\prime}(x_{0})$ and $\bar{\sigma}=\sigma(x_{0})$. Let $\hat
{x}_{+},\hat{x}_{-}$ be such that $\hat{x}_{+}-x_{0}=x_{0}-\hat{x}_{-}$. The
appropriate translated version of $F_{2,+}^{M}$ is
\[
F_{2,+}^{M}(t,x)=a^{M}(t) \bigl(
(x-x_{0})-(\hat{x}_{+}-x_{0})e^{-c(t-T)}%
 \bigr) ^{2}+F_{1}(\hat{x}_{+})
\]
and $F_{2,-}^{M}(t,x_{0}-x)=F_{2,+}^{M}(t,x_{0}+x)$.

The subsolution for times less than $T-t^{\ast}$ will be the mollification of
$F_{2,+}^{M}(t,x)\wedge F_{2,-}^{M}(t,x)\wedge\bar{F}_{1}(x)$. Note that since
$\bar{F}_{1}(x)$ agrees with $F_{1}(x)$ up to second order, it is still the
case that $F_{2,+}^{M}(t,x)\wedge F_{2,-}^{M}(t,x)$ will be smallest near
$x=x_{0}$. Letting
\begin{equation}\label{Eq:second_moll}
U^{\delta}(t,x)=-\delta\log \bigl( e^{-({1}/{\delta})F_{2,+}^{M}%
(t,x)}+e^{-({1}/{\delta})F_{2,-}^{M}(t,x)}+e^{-({1}/{\delta})\bar{F}%
_{1}(x)}
\bigr)
\end{equation}
and%
\begin{equation}\label{Eq:NonLinearChangeOfMeasure}
\bar{U}^{\delta}(t,x)= \cases{
\bar{F}_{1}(x), &\quad $t>T-t^{\ast}$,
\vspace*{3pt}\cr
U^{\delta}(t,x), &\quad $t\leq T-t^{\ast}$,}
\end{equation}
the suggested importance sampling control that is used for the simulation is
\[
\bar{u}(t,x)=-\sigma(x)D\bar{U}^{\delta}(t,x).
\]
Notice that this construction reduces to the construction of the linear case
if the potential is indeed quadratic, since then $\bar{F}_{1}(x)=F_{1}(x)$.

In Section~\ref{SS:SimulationNonLinearProblem} we present simulation data
for the nonlinear problem, demonstrating the effectiveness of the suggested
change of measure. The analysis and the theoretical bound for the performance
of this scheme are completely analogous to the linear problem, modulo the
additional error coming from the linearization of the dynamics in the
neighborhood of the stable equilibrium point. In Section~\ref{SS:ProofNonLinearProblem}
we rigorously analyze the performance of this algorithm.

\subsection{Simulation results for nonlinear problem}
\label{SS:SimulationNonLinearProblem}

In this subsection, we\break present simulation data for the nonlinear problem. We
take the drift to be $b(x)=-V^{\prime}(x)$, where the potential function is
$V(x)=\frac{1}{2}(x^{2}-1)^{2}$. This potential function has two stable points
at $-1$ and at $+1$, and an unstable equilibrium at $0$. We assume that the
starting point is at the left equilibrium point $x_{0}=-1$, and the exit set
is the level set of the potential function $\mathcal{D}=\{x \dvtx  V(x)\leq L\}$,
with $L=0.45$. Thus exit occurs from either of the points $A_{1}=-1.40$ or
$A_{2}=-0.23$.

Notice that the local quadratic approximation around the equilibrium point is
$V_{q}(x)=\frac{1}{2}c(x+1)^{2}$ with $c=4$. Moreover, we have chosen, for
simplicity, the diffusion coefficient to be constant $\sigma(x)=1$. $N=10^{7}$
independent trajectories were used for the simulations.

We first investigate the performance of a change of measure based on the
quasipotential subsolution. Thus\vspace*{1pt} we change the measure via the control
$\bar{u}(x)=-D\bar{F}_{1}(x)$. Estimated values and the corresponding
estimated relative errors per sample for several values of $(\varepsilon,T)$
are in Tables~\ref{Table1aNL} and \ref{Table1bNL}, respectively.

\begin{table}[t]
\tabcolsep=0pt
\caption{Using the subsolution based on quasipotential throughout. Estimated
values for different pairs~$(\varepsilon,T)$}
\label{Table1aNL}
{\fontsize{8pt}{10pt}\selectfont{\begin{tabular*}{\tablewidth}{@{\extracolsep{\fill}}lccccccccc@{}}
\hline
& \multicolumn{9}{c@{}}{$\bolds{T}$}\\[-4pt]
& \multicolumn{9}{c@{}}{\hrulefill}\\
$\bolds{\varepsilon}$ & $\mathbf{0.5}$ & $\mathbf{1}$ & $\mathbf{1.5}$ & $\mathbf{2.5}$ &
$\mathbf{4}$ & $\mathbf{5}$ & $\mathbf{6}$ & $\mathbf{8}$ & $\mathbf{10}$\\
\hline
$0.14$ & $3.01\mathrm{e}{-}03 $ & $8.21\mathrm{e}{-}03$ & $1.36\mathrm{e}{-}02$ & $2.48\mathrm{e}{-}02$ & $5.82\mathrm{e}{-}02$ &
$6.22\mathrm{e}{-}02 $ & $4.11\mathrm{e}{-}02 $ & $4.29\mathrm{e}{-}02$ & $5.31\mathrm{e}{-}02$\\
$0.12$ & $1.03\mathrm{e}{-}03 $ & $2.91\mathrm{e}{-}03$ & $4.92\mathrm{e}{-}03$ & $8.95\mathrm{e}{-}03$ & $1.46\mathrm{e}{-}02$ &
$1.73\mathrm{e}{-}02 $ & $2.10\mathrm{e}{-}02 $ & $1.81\mathrm{e}{-}02$ & $1.72\mathrm{e}{-}02$\\
$0.09$ & $8.27\mathrm{e}{-}05 $ & $2.52\mathrm{e}{-}04$ & $4.35\mathrm{e}{-}04$ & $8.10\mathrm{e}{-}04$ & $1.33\mathrm{e}{-}03$ &
$1.53\mathrm{e}{-}03 $ & $1.65\mathrm{e}{-}03 $ & $1.58\mathrm{e}{-}03$ & $1.57\mathrm{e}{-}03$\\
$0.07$ & $4.60\mathrm{e}{-}06 $ & $1.49\mathrm{e}{-}05$ & $2.64\mathrm{e}{-}05$ & $4.97\mathrm{e}{-}05$ & $7.63\mathrm{e}{-}05$ &
$1.11\mathrm{e}{-}04 $ & $9.74\mathrm{e}{-}05 $ & $1.21\mathrm{e}{-}04$ & $9.19\mathrm{e}{-}05$\\
$0.05$ & $2.49\mathrm{e}{-}08 $ & $8.91\mathrm{e}{-}08$ & $1.62\mathrm{e}{-}07$ & $3.15\mathrm{e}{-}07$ & $5.03\mathrm{e}{-}07$ &
$6.19\mathrm{e}{-}07 $ & $6.78\mathrm{e}{-}07 $ & $5.85\mathrm{e}{-}07$ & $6.29\mathrm{e}{-}07$\\
$0.03$ & $1.25\mathrm{e}{-}13$ & $5.40\mathrm{e}{-}13$ & $1.03\mathrm{e}{-}12$ & $2.10\mathrm{e}{-}12$ & $3.80\mathrm{e}{-}12$ &
$6.12\mathrm{e}{-}12 $ & $4.87\mathrm{e}{-}12 $ & $1.04\mathrm{e}{-}11$ & $5.18\mathrm{e}{-}12$\\
\hline
\end{tabular*}}}
\end{table}

\begin{table}
\caption{Using subsolution based on quasipotential throughout. Relative errors
per sample for different pairs~$(\varepsilon,T)$}
\label{Table1bNL}
\begin{tabular*}{\tablewidth}{@{\extracolsep{\fill}}lccccccccc@{}}\hline
& \multicolumn{9}{c}{$\bolds{T}$}       \\[-4pt]
& \multicolumn{9}{c@{}}{\hrulefill}\\
$\bolds{\varepsilon}$ & $\mathbf{0.5}$ & $\mathbf{1}$ & $\mathbf{1.5}$ & $\mathbf{2.5}$ &
$\mathbf{4}$ & $\mathbf{5}$ & $\mathbf{6}$ & $\mathbf{8}$ & $\mathbf{10}$\\
\hline
$0.14$ & $2$ & $2$ & $4$ & $21$ & $799$ & $953 $ & $127 $ & \phantom{0}$169$ &
$488$\\
$0.12$ & $2$ & $2$ & $4$ & $18$ & $125$ & $315 $ & $649 $ & \phantom{0}$368$ &
$301$\\
$0.09$ & $2$ & $2$ & $4$ & $20$ & $143$ & $155 $ & $311 $ & \phantom{0}$173$ &
$288$\\
$0.07$ & $2$ & $2$ & $4$ & $17$ & \phantom{0}$68$ & $433 $ & $192 $ & \phantom{0}$540$ &
$272$\\
$0.05$ & $2$ & $2$ & $4$ & $16$ & \phantom{0}$77$ & $296 $ & $410 $ & \phantom{0}$148$ &
$287$\\
$0.03$ & $2$ & $2$ & $3$ & $14$ & $160$ & $638 $ & $347 $ & $1933$ &
$317$\\
\hline
\end{tabular*}
\end{table}

As we see from Table~\ref{Table1bNL}, even though the quasipotential
subsolution performs relatively well for small values of $T$, there is a clear
degradation of performance as $T$ gets larger. It is also interesting to note
that the degradation is uniform across all values of $\varepsilon$ for the
same value of $T$. This behavior parallels what was observed for the linear
problem. As was mentioned there, the large per sample relative errors for
$T\geq2.5$ should not be taken as being accurate, but just indicative of poor performance.
\begin{sidewaystable}
\tablewidth=\textwidth
\caption{Estimated probability values for different pairs $(\varepsilon,T)$
using the exponential mollification. Parameter choices $M=\frac{2c}%
{\bar{\sigma}^{2}}\frac{\hat{x}^{2}}{\varepsilon^{2\kappa}}$ with $\kappa=0.4$
and $\hat{x}=0.4$}
\label{Table2aNL}
\begin{tabular*}{\tablewidth}{@{\extracolsep{\fill}}lcccccccccc@{}}
\hline
& \multicolumn{10}{c@{}}{$\bolds{T}$}\\[-4pt]
& \multicolumn{10}{c@{}}{\hrulefill}\\
$\bolds{\varepsilon}$  & $\mathbf{0.5}$ & $\mathbf{1}$ & $\mathbf{1.5}$ & $\mathbf{2.5}$ &
$\mathbf{4}$ & $\mathbf{5}$ & $\mathbf{6}$ & $\mathbf{8}$ & $\mathbf{10}$ & $\mathbf{13}$\\
\hline
$0.14$ & $3.07\mathrm{e}{-}03 $ & $8.36\mathrm{e}{-}03$ & $1.39\mathrm{e}{-}02$ & $2.49\mathrm{e}{-}02$ & $4.11\mathrm{e}{-}02$ &
$5.12\mathrm{e}{-}02 $ & $6.24\mathrm{e}{-}02 $ & $8.33\mathrm{e}{-}02$ & $1.04\mathrm{e}{-}01$ & $1.34\mathrm{e}{-}01$\\
$0.12$ & $1.05\mathrm{e}{-}03 $ & $2.97\mathrm{e}{-}03$ & $4.99\mathrm{e}{-}03$ & $9.06\mathrm{e}{-}03$ & $1.52\mathrm{e}{-}02$ &
$1.91\mathrm{e}{-}02 $ & $2.32\mathrm{e}{-}02 $ & $3.12\mathrm{e}{-}02$ & $3.92\mathrm{e}{-}02$ & $5.09\mathrm{e}{-}03$\\
$0.09$ & $8.36\mathrm{e}{-}05 $ & $2.53\mathrm{e}{-}04$ & $4.39\mathrm{e}{-}04$ & $8.12\mathrm{e}{-}04$ & $1.38\mathrm{e}{-}03$ &
$1.76\mathrm{e}{-}03 $ & $2.14\mathrm{e}{-}03 $ & $2.89\mathrm{e}{-}03$ & $3.65\mathrm{e}{-}03$ & $4.78\mathrm{e}{-}03$\\
$0.08$ & $2.36\mathrm{e}{-}05 $ & $7.41\mathrm{e}{-}05$ & $1.29\mathrm{e}{-}04$ & $2.41\mathrm{e}{-}04$ & $4.11\mathrm{e}{-}04$ &
$5.24\mathrm{e}{-}04 $ & $6.38\mathrm{e}{-}04 $ & $8.63\mathrm{e}{-}04$ & $1.08\mathrm{e}{-}03$ & $1.43\mathrm{e}{-}03$\\
$0.07$ & $4.60\mathrm{e}{-}06 $ & $1.49\mathrm{e}{-}05$ & $2.65\mathrm{e}{-}05$ & $5.01\mathrm{e}{-}05$ & $8.55\mathrm{e}{-}05$ &
$1.09\mathrm{e}{-}04 $ & $1.33\mathrm{e}{-}04 $ & $1.81\mathrm{e}{-}04$ & $2.28\mathrm{e}{-}04$ & $2.99\mathrm{e}{-}04$\\
$0.05$ & $2.48\mathrm{e}{-}08 $ & $8.91\mathrm{e}{-}08$ & $1.63\mathrm{e}{-}07$ & $3.16\mathrm{e}{-}07$ & $5.44\mathrm{e}{-}07$ &
$6.99\mathrm{e}{-}07 $ & $8.51\mathrm{e}{-}07 $ & $1.16\mathrm{e}{-}06$ & $1.47\mathrm{e}{-}06$ & $1.92\mathrm{e}{-}06$\\
$0.04$ & $2.57\mathrm{e}{-}10 $ & $9.89\mathrm{e}{-}10$ & $1.85\mathrm{e}{-}09$ & $3.64\mathrm{e}{-}09$ & $6.35\mathrm{e}{-}09$ &
$8.15\mathrm{e}{-}09 $ & $1.01\mathrm{e}{-}08 $ & $1.36\mathrm{e}{-}08$ & $1.72\mathrm{e}{-}08$ & $2.26\mathrm{e}{-}08$\\
$0.03$ & $1.25\mathrm{e}{-}13 $ & $5.38\mathrm{e}{-}13$ & $1.03\mathrm{e}{-}12$ & $2.08\mathrm{e}{-}12$ & $3.68\mathrm{e}{-}12$ &
$4.74\mathrm{e}{-}12 $ & $5.80\mathrm{e}{-}12 $ & $7.94\mathrm{e}{-}12$ & $1.01\mathrm{e}{-}11$ & $1.32\mathrm{e}{-}11$\\
\hline
\end{tabular*}
\end{sidewaystable}

\begin{table}[t]
\caption{Relative errors per sample for different pairs $(\varepsilon,T)$
using the exponential mollification. Parameter choices $M=\frac{2c}%
{\bar{\sigma}^{2}}\frac{\hat{x}^{2}}{\varepsilon^{2\kappa}}$ with $\kappa=0.4$
and $\hat{x}=0.4$}
\label{Table2bNL}
\begin{tabular*}{\tablewidth}{@{\extracolsep{\fill}}lcccccccccc@{}}
\hline
& \multicolumn{10}{c@{}}{$\bolds{T}$}\\[-4pt]
& \multicolumn{10}{c@{}}{\hrulefill}\\
$\bolds{\varepsilon}$ & $\mathbf{0.5}$ & $\mathbf{1}$ & $\mathbf{1.5}$ & $\mathbf{2.5}$ &
$\mathbf{4}$ & $\mathbf{5}$ & $\mathbf{6}$ & $\mathbf{8}$ & $\mathbf{10}$ & $\mathbf{13}$\\
\hline
$0.14$ & $3.7$ & $2.3$ & $1.8$ & $1.3$ & $1.1$ & $1.0 $ & $1.0$ & $1.1$ &
$1.3$ & $1.7$\\
$0.12$ & $4.2$ & $2.6$ & $2.0$ & $1.5$ & $1.2$ & $1.1 $ & $1.1$ & $1.1$ &
$1.2$ & $1.5$\\
$0.09$ & $5.1$ & $3.3$ & $2.6$ & $1.9$ & $1.5$ & $1.3 $ & $1.2$ & $1.2$ &
$1.2$ & $1.2$\\
$0.08$ & $5.3$ & $3.7$ & $2.8$ & $2.1$ & $1.8$ & $1.4 $ & $1.3$ & $1.2$ &
$1.2$ & $1.2$\\
$0.07$ & $5.5$ & $4.1$ & $3.2$ & $2.3$ & $1.8$ & $1.6 $ & $1.5$ & $1.3$ &
$1.2$ & $1.2$\\
$0.05$ & $4.9$ & $5.4$ & $4.5$ & $3.2$ & $2.5$ & $2.2 $ & $2.0$ & $1.7$ &
$1.6$ & $1.3$\\
$0.04$ & $3.2$ & $6.5$ & $5.4$ & $4.1$ & $3.2$ & $2.8$ & $2.6$ & $2.2$ & $2.0$
& $1.8$\\
$0.03$ & $2.8$ & $8.0$ & $7.3$ & $5.6$ & $4.4$ & $3.9 $ & $3.5$ & $3.0$ &
$2.7$ & $2.4$\\
\hline
\end{tabular*}
\end{table}

\begin{table}[b]
\caption{Relative errors per sample for different pairs $(\varepsilon,T)$
using the exponential mollification. Parameter choices $M=\frac{2c}{\bar{\sigma}^{2}}\frac{\hat{x}^{2}}{\varepsilon^{2\kappa}}$
with $\kappa=0.4$
and $\hat{x}=0.5$}
\label{Table3bNL}
\begin{tabular*}{\tablewidth}{@{\extracolsep{\fill}}lcccccccccc@{}}\hline
& \multicolumn{10}{c@{}}{$\bolds{T}$}\\[-4pt]
& \multicolumn{10}{l@{}}{\hrulefill} \\
$\bolds{\varepsilon}$ & $\mathbf{0.5}$ & $\mathbf{1}$ & $\mathbf{1.5}$ & $\mathbf{2.5}$ &
$\mathbf{4}$ & $\mathbf{5}$ & $\mathbf{6}$ & $\mathbf{8}$ & $\mathbf{10}$ & $\mathbf{13}$\\
\hline
$0.14$ & $2.7$ & $2.5$ & $2.0$ & $1.5$ & $1.2$ & $1.0 $ & $1.0$ & $1.0$ &
$1.0$ & $1.1$\\
$0.12$ & $2.6$ & $2.9$ & $2.4$ & $1.8$ & $1.4$ & $1.2 $ & $1.1$ & $1.1$ &
$1.1$ & $1.1$\\
$0.09$ & $2.1$ & $3.7$ & $3.1$ & $2.4$ & $1.9$ & $1.7 $ & $1.6$ & $1.4$ &
$1.3$ & $1.2$\\
$0.08$ & $1.9$ & $4.1$ & $3.5$ & $2.7$ & $2.1$ & $1.8 $ & $1.8$ & $1.5$ &
$1.4$ & $1.3$\\
$0.07$ & $1.9$ & $4.5$ & $3.9$ & $3.0$ & $2.4$ & $2.1 $ & $1.9$ & $1.7$ &
$1.5$ & $1.4$\\
$0.05$ & $1.8$ & $5.5$ & $5.2$ & $4.2$ & $3.2$ & $2.9 $ & $2.7$ & $2.3$ &
$2.0$ & $1.9$\\
$0.04$ & $1.7$ & $6.0$ & $6.4$ & $5.2$ & $4.1$ & $3.7$ & $3.4$ & $2.9$ & $2.7$
& $2.3$\\
$0.03$ & $2.0$ & $5.9$ & $8.2$ & $6.9$ & $5.6$ & $4.9 $ & $4.5$ & $3.9$ &
$3.6$ & $3.1$\\
\hline
\end{tabular*}
\end{table}

Next we investigate how the suggested change of measure performs. To apply the
control, we choose values for the parameters $(\hat{x},M,t^{\ast},\delta)$
according to the discussion in Section~\ref{S:SimulationResultsLinearProblem}. However, for reasons that will become
clearer in the proof of Lemma \ref{L:Region_y_hatx_NonLinear}, we need to
strengthen the condition $M\geq4c/\bar{\sigma}^{2}$ to $M>4c/\bar{\sigma}^{2}%
$, say $M\geq5c/\bar{\sigma}^{2}$. When we link the other parameters to
$\varepsilon$, then $z\rightarrow0$ as $\varepsilon\rightarrow0$. Since $z$ measures the size of the neighborhood
on which the linearization is relevant, we do not explicitly take into account the error from the approximation
around the neighborhood of the rest point of the true dynamics by its
linearization when selecting the parameters for the implementation of the scheme.

Estimated values and corresponding estimated relative errors of the exit
probabilities of interest for different pairs $(\varepsilon,T)$ and different
combinations values for~$z$ are in Tables~\ref{Table2aNL}--\ref{Table5bNL}. Estimated relative errors for $M=\frac{2c}{\bar{\sigma}^{2}%
}\frac{\hat{x}^{2}}{\varepsilon^{2\kappa}}$ with $\kappa=0.4$  and $\hat
{x}=0.4$ are reported in Table~\ref{Table2aNL}, whereas the related relative
errors are reported in Table~\ref{Table2bNL}. In Tables~\ref{Table3bNL} and
\ref{Table4bNL}, we report only estimated relative errors for the same value
of $\kappa$ but for $\hat{x}=0.5$ and $\hat{x}=1$, respectively. The related
probability estimates are almost identical to those of Table~\ref{Table2aNL},
so they are not repeated.

Note that for Table~\ref{Table3bNL}, $t^{\ast}\geq T$ when $T=0.5$ and for
$\varepsilon\leq0.05$. Similarly, for Table~\ref{Table4bNL}, $t^{\ast}\geq T$
when $T=0.5$ and for $T=1$ when $\varepsilon\leq0.09$. For such values, the
quasipotential subsolution is being used everywhere [see
(\ref{Eq:NonLinearChangeOfMeasure})], and the numerical results for these
values agree with those from Table~\ref{Table1bNL}.

In order to illustrate the effect when the linear approximation is used over a
relatively large region, the data in Table~\ref{Table5bNL} are estimated
relative errors when $M$ is considerably smaller than before, and thus $z$ is
considerably larger. In particular, we have taken $\kappa=0.25$ and $\hat
{x}=1$. Comparing Tables~\ref{Table4bNL} and \ref{Table5bNL}, we notice that
if the error from the linearization is not confined to a small enough region,
then the algorithm degrades in $\varepsilon$ though it appears stable in $T$.
This is consistent with the theoretical results, which imply a uniformity in
$T$ but only logarithmic optimality in $\varepsilon$. That said, one would
like to minimize errors associated with linearization as far as possible. As
noted in Section~\ref{S:SimulationResultsLinearProblem}, one should choose
the scaling parameter $\kappa\in(0,1/2)$. However, with the nonlinear problem
minimizing the region over which the approximation is used calls for larger
$\kappa$, and so one want it close to but not equal to $1/2$. For the problems
considered here, $\kappa=0.4$ worked well.

\begin{table}[t]
\caption{Relative errors per sample for different pairs $(\varepsilon,T)$
using the exponential mollification. Parameter choices $M=\frac{2c}%
{\bar{\sigma}^{2}}\frac{\hat{x}^{2}}{\varepsilon^{2\kappa}}$ with $\kappa=0.4$
and $\hat{x}=1$}
\label{Table4bNL}
\begin{tabular*}{\tablewidth}{@{\extracolsep{\fill}}lcccccccccc@{}}
\hline
& \multicolumn{10}{c@{}}{$\bolds{T}$}\\[-4pt]
& \multicolumn{10}{c@{}}{\hrulefill}\\
$\bolds{\varepsilon}$ & $\mathbf{0.5}$ & $\mathbf{1}$ & $\mathbf{1.5}$ & $\mathbf{2.5}$ &
$\mathbf{4}$ & $\mathbf{5}$ & $\mathbf{6}$ & $\mathbf{8}$ & $\mathbf{10}$ & $\mathbf{13}$\\
\hline
$0.14$ & $1.5$ & $2.4$ & $2.6$ & $2.2$ & $1.8$ & $1.6 $ & $1.5$ & $1.3$ &
$1.2$ & $1.0$\\
$0.12$ & $1.5$ & $2.3$ & $3.0$ & $2.7$ & $2.2$ & $2.0 $ & $1.9$ & $1.6$ &
$1.4$ & $1.3$\\
$0.09$ & $1.6$ & $2.2$ & $3.9$ & $3.8$ & $3.2$ & $2.9 $ & $2.7$ & $2.4$ &
$2.1$ & $1.9$\\
$0.08$ & $1.6$ & $2.1$ & $4.1$ & $4.2$ & $3.7$ & $3.3 $ & $3.1$ & $2.7$ &
$2.5$ & $2.1$\\
$0.07$ & $1.8$ & $2.1$ & $4.2$ & $4.7$ & $4.0$ & $3.7 $ & $3.4$ & $3.0$ &
$2.7$ & $2.4$\\
$0.05$ & $1.8$ & $2.1$ & $4.4$ & $6.0$ & $5.3$ & $4.8 $ & $4.5$ & $3.9$ &
$3.6$ & $3.2$\\
$0.04$ & $1.9$ & $2.1$ & $4.8$ & $7.1$ & $6.3$ & $5.8$ & $5.5$ & $4.8$ & $4.3$
& $3.9$\\
$0.03$ & $2.0$ & $2.1$ & $3.5$ & $8.7$ & $8.2$ & $7.5 $ & $6.9$ & $6.2$ &
$5.6$ & $4.9$\\
\hline
\end{tabular*}
\end{table}

\begin{table}[b]
\caption{Relative errors per sample for different pairs $(\varepsilon,T)$
using the exponential mollification. Parameter choices $M=\frac{2c}%
{\bar{\sigma}^{2}}\frac{\hat{x}^{2}}{\varepsilon^{2\kappa}}$ with
$\kappa=0.25$ and $\hat{x}=1$}
\label{Table5bNL}
\begin{tabular*}{\tablewidth}{@{\extracolsep{\fill}}lcccccccccc@{}}
\hline
& \multicolumn{10}{c@{}}{$\bolds{T}$}\\[-4pt]
& \multicolumn{10}{c@{}}{\hrulefill}\\
$\bolds{\varepsilon}$ & $\mathbf{0.5}$ & $\mathbf{1}$ & $\mathbf{1.5}$ & $\mathbf{2.5}$ &
$\mathbf{4}$ & $\mathbf{5}$ & $\mathbf{6}$ & $\mathbf{8}$ & $\mathbf{10}$ & $\mathbf{13}$\\
\hline
$0.14$ & $1.5$ & \phantom{0}$5$  & \phantom{00}$5$   & \phantom{00}$4$   & \phantom{00}$3$   & \phantom{00}$3$   & \phantom{000}$3$    & \phantom{000}$2$    & \phantom{00}$2$   & \phantom{000}$2$    \\
$0.12$ & $1.5$ & \phantom{0}$7$  & \phantom{00}$7$   & \phantom{00}$6$   & \phantom{00}$5$   & \phantom{00}$5$   & \phantom{000}$4$    & \phantom{000}$4$    & \phantom{00}$3$   & \phantom{000}$3$    \\
$0.09$ & $1.6$ & $14$ & \phantom{0}$19$  & \phantom{0}$17$  & \phantom{0}$14$  & \phantom{0}$13$  & \phantom{00}$12$   & \phantom{00}$11$   & \phantom{00}$9$   & \phantom{000}$8$    \\
$0.08$ & $1.6$ & $21$ & \phantom{0}$30$  & \phantom{0}$28$  & \phantom{0}$23$  & \phantom{0}$22$  & \phantom{00}$20$   & \phantom{00}$17$   & \phantom{0}$15$  & \phantom{00}$14$   \\
$0.07$ & $1.8$ & $31$ & \phantom{0}$54$  & \phantom{0}$51$  & \phantom{0}$42$  & \phantom{0}$39$  & \phantom{00}$36$   & \phantom{00}$31$   & \phantom{0}$28$  & \phantom{00}$25$   \\
$0.05$ & $1.8$ & $36$ & $276$ & $278$ & $214$ & $220$ & \phantom{0}$185$  & \phantom{0}$153$  & $148$ & \phantom{0}$128$  \\
$0.04$ & $1.9$ & \phantom{0}$5$  & $568$ & $577$ & $665$ & $616$  & \phantom{0}$507$  & \phantom{0}$400$  & $415$ & \phantom{0}$374$  \\
$0.03$ & $2.0$ & \phantom{0}$3$  & $302$ & \phantom{0}$60$  & $190$ & \phantom{0}$39$  & $1878$ & $1485$ & \phantom{0}$96$  & $1562$ \\
\hline
\end{tabular*}
\end{table}

\subsection{Analysis of the simulation scheme for the nonlinear problem}

\label{SS:ProofNonLinearProblem} In this subsection, we present the
theoretical analysis of the simulation scheme for general one-dimensional
nonlinear dynamics and provide rigorous bounds on performance. As for the
linear case, the analysis is valid for $\varepsilon>0$ without degradation as
$T\rightarrow\infty$. The analysis and the theoretical bound for the
performance of this scheme are completely analogous to the linear problem,
modulo the additional error coming from the linearization of the dynamics in
the neighborhood of the stable equilibrium point.

To distinguish between the linear and the nonlinear problem, we need to
introduce some notation. For a function $W\in\mathcal{C}^{1,2}([0,T]\times
\mathbb{R})$, we define the operator
\begin{eqnarray*}
\bar{\mathcal{G}}^{\varepsilon}[W](t,x) &=& W_{t}(t,x)+\bar{
\mathbb{H}}\bigl(x,DW(t,x)\bigr)+\frac{\varepsilon}{2}
\sigma^{2}(x)D^{2}W(t,x),\\
 \bar{\mathbb{H}}(x,p) &=& b(x)p-
\tfrac{1}{2}\bigl\llvert \sigma(x)p\bigr\rrvert ^{2}.
\end{eqnarray*}
In analogy to (\ref{Eq:Ge_2}), for smooth functions $W,U$, we define
\[
\bar{{\mathcal{G}}}^{\varepsilon}[W,U](t,x)=\bar{\mathcal{G}}^{\varepsilon
}[W](t,x)-
\tfrac{1}{2}\bigl\llvert \sigma(x) \bigl( DW(t,x)-DU(t,x) \bigr) \bigr\rrvert
^{2}.\label{Eq:Ge_2NL}%
\]
Moreover, setting $c=-b^{\prime}(x_{0})>0$ and $\bar{\sigma}=\sigma(x_{0})$,
we recall that
\begin{eqnarray*}
\mathcal{G}^{\varepsilon}[W](t,x) &=& W_{t}(t,x)+\mathbb{H}
\bigl(x,DW(t,x)\bigr)+\frac
{\varepsilon}{2}\bar{\sigma}^{2}D^{2}W(t,x),\\
\mathbb{H}(x,p)&=& -cxp-\tfrac{1}{2}\llvert \bar{\sigma}p\rrvert
^{2}.
\end{eqnarray*}

The operators with bars correspond to the nonlinear problem, whereas the
operators without bars give the corresponding first-order approximations. We
define an operator measuring the error from the approximation by%
%
\begin{equation}\label{Eq:ErrorOfOperator}%
R^{\varepsilon}[W](t,x)=\bar{\mathcal{G}}^{\varepsilon}[W](t,x)-
\mathcal{G}%
^{\varepsilon}[W](t,x).
\end{equation}
Moreover, since $b(x)$ and $\sigma(x)$ are $\mathcal{C}^{1}(\mathbb{R})$, we
can write for any $x\in\mathbb{R}$,
\[
b(x)  =b(x_{0})+b^{\prime}(x_{0})
(x-x_{0})+R_{1}(x)\quad \mbox{and}\quad \sigma(x) =
\sigma(x_{0})+R_{2}(x),
\]
where $R_{1}(x)|x|^{-2}$ and $R_{2}(x)|x|^{-1}$ are locally bounded. By
assumption we have that $b(x_{0})=0$ and $\sigma^{2}(x)>0$.

As with the linear problem, the subsolution used for the analysis is based on
the $\delta$-exponential mollification (\ref{Eq:second_moll}) reduced by the
multiplicative factor $(1-\eta)$. We recall that this differs from the
subsolution used for the design, which has $\eta=0$.

Next we proceed with the mathematical analysis of the scheme. The analysis is
parallel to what was done for the linear problem, modulo adjustments due to
the linearization of the dynamics in the neighborhood of the stable point, and
therefore we mainly focus on the differences. In order to simplify the
notation we assume without loss of generality (as it was done in the linear
problem) that the stable equilibrium is $x_{0}=0$. We write $\hat{x}_{+}%
=-\hat{x}_{-}=\hat{x}$, and for notational convenience assume that
$A_{2}=-A_{1}=A$. As in the linear problem, $z=\hat{x}(c/M\bar{\sigma}%
^{2})^{1/2}/2$ and $H=10z$.

The following lemma bounds the error from the approximation.

\begin{lemma}
\label{L:BoundForApproximationOfOperatorG}
Consider $(t,x)\in [
0,T-t^{\ast}]\times[0,z]$. Then, for $z<\min\{1,A\}$,
\begin{eqnarray*}
&& \bigl\llvert R^{\varepsilon}\bigl[F_{2}^{M}\bigr](t,x)\bigr
\rrvert \\
&&\qquad \leq  2a^{M}%
(t) \bigl(z+\hat{x}e^{-c(t-T)}
\bigr)\sup_{x\in [0,z]}\bigl|b(x)+cx\bigr|
\\
&&\qquad\quad{}+ \bigl[ 2 \bigl( a^{M}(t) \bigl(z+\hat{x}e^{-c(t-T)}\bigr)
\bigr) ^{2}+\varepsilon a^{M}(t) \bigr] \sup
_{x\in [0,z]}\bigl|\sigma^{2}%
(x)-\bar{
\sigma}^{2}\bigr|
\\
&&\qquad \leq   C_{0} \bigl\{ a^{M}(t) \bigl(z+
\hat{x}e^{-c(t-T)}\bigr)z^{2}+ \bigl( a^{M}(t) \bigl(z+
\hat{x}e^{-c(t-T)}\bigr) \bigr) ^{2}z+a^{M}(t)
\varepsilon z \bigr\},
\end{eqnarray*}
where
\[
C_{0}=\sup_{x\in [0,A]} \biggl[ \frac{\llvert  R_{1}(x)\rrvert
}{|x|^{2}}+
\frac{\llvert  R_{2}(x)\rrvert  }{|x|}\bigl\llvert 2\bar{\sigma }+R_{2}(x)\bigr\rrvert
\biggr].
\]
In addition, for a fixed constant $C_{1}<\infty$, we have
\[
\sup_{x\in [0,z]}\bigl|\bar{F}_{1}(x)-F_{1}(x)\bigr|\leq
C_{1}z^{3}.
\]
\end{lemma}

\begin{pf}
The conclusion follows after a straightforward substitution by using that for all $x\in [0,A]$,
$|R_{1}(x)|\leq C|x|^{2}$  and $|R_{2}(x)|\leq C|x|$ for a constant $C$ that depends only on $A$.
\end{pf}

For notational convenience we identify the quantity appearing in the upper
bound for $\llvert  R^{\varepsilon}[F_{2}^{M}](t,x)\rrvert$ as given in
Lemma \ref{L:BoundForApproximationOfOperatorG},%
\begin{eqnarray}
r(\varepsilon,\hat{x},M,t)&=& a^{M}(t) \bigl(z+\hat{x}e^{-c(t-T)}
\bigr)z^{2}
\nonumber
\\[-8pt]
\label{Eq:defofR}
\\[-8pt]
\nonumber
&&{}+ \bigl( a^{M}(t) \bigl(z+\hat{x}e^{-c(t-T)}
\bigr) \bigr) ^{2}z+a^{M}(t)\varepsilon z.
\end{eqnarray}
The following lemma shows that the error term induced by the local
approximation of the dynamics in the neighborhood of the stable equilibrium
point does not degrade as $T$ gets large.

\begin{lemma}
\label{L:IntegralOfCorrectionTerm} We have that
\begin{eqnarray*}
\int_{0}^{T-t^{\ast}}r(\varepsilon,\hat{x},M,t)\,dt &=&
J_{1}\bigl(t^{\ast}%
,T,M\bigr)z^{3}+J_{2}
\bigl(t^{\ast},T,M\bigr)z^{2}\hat{x}\\
&&{}+ J_{3}
\bigl(t^{\ast},T,M\bigr)z\hat{x}%
^{2}+J_{4}
\bigl(t^{\ast},T,M\bigr) \varepsilon z
\end{eqnarray*}
where, letting $K=\frac{2c}{M}+\bar{\sigma}^{2}$,
\begin{eqnarray*}
J_{1}\bigl(t^{\ast},T,M\bigr) & =& \frac{1}{2\bar{\sigma}^{2}} \biggl(
1-\frac{c}%
{\bar{\sigma}^{2}} \biggr) \log\frac{K-\bar{\sigma}^{2}e^{-2cT}}{K-\bar
{\sigma}^{2}e^{-2ct^{\ast}}}\nonumber\\
&&{}+\frac{c}{2\bar{\sigma}^{4}} \biggl[
\frac
{K}{K-\bar{\sigma}^{2}e^{-2ct^{\ast}}}-\frac{K}{K-\bar{\sigma}^{2}e^{-2cT}%
} \biggr]
\nonumber
\\
J_{2}\bigl(t^{\ast},T,M\bigr) & =& \frac{1}{\bar{\sigma}\sqrt{K}} \biggl( 1-
\frac{c}%
{\bar{\sigma}^{2}} \biggr) \biggl[ \log\frac{1+({\bar{\sigma}}/{\sqrt{K}})e^{-ct^{\ast}}}{1-({\bar{\sigma}}/{\sqrt{K}})e^{-ct^{\ast}}}
\\
&&\qquad\qquad\hspace*{35pt}{}-
\log \frac{1+({\bar{\sigma}}/{\sqrt{K}})e^{-cT}}{1-({\bar{\sigma}}/{\sqrt{K}})e^{-cT}}
\biggr]
\nonumber
\\
&&{}+\frac{c}{2\bar{\sigma}^{4}} \biggl[ \frac{2\bar{\sigma}^{2}e^{-ct^{\ast}}%
}{K-\bar{\sigma}^{2}e^{-2ct^{\ast}}}-\frac{2\bar{\sigma}^{2}e^{-cT}}%
{K-\bar{\sigma}^{2}e^{-2cT}} \biggr]
\nonumber
\\
J_{3}\bigl(t^{\ast},T,M\bigr) & =&\frac{c}{2\bar{\sigma}^{4}} \biggl[
\frac{\bar{\sigma
}^{2}}{K-\bar{\sigma}^{2}e^{-2ct^{\ast}}}-\frac{\bar{\sigma}^{2}}%
{K-\bar{\sigma}^{2}e^{-2cT}} \biggr]
\nonumber
\\
J_{4}\bigl(t^{\ast},T,M\bigr) & =&\frac{1}{2\bar{\sigma}^{2}}\log
\frac{K-\bar{\sigma
}^{2}e^{-2cT}}{K-\bar{\sigma}^{2}e^{-2ct^{\ast}}}.
\nonumber
\end{eqnarray*}

In particular, $\lim_{T\rightarrow\infty}\int_{0}^{T-t^{\ast}}r(\varepsilon
,\hat{x},M,t)\,dt<\infty$.
\end{lemma}

The proof of this lemma follows by straightforward integration of
$r(\varepsilon,\hat{x},M,t)$. Moreover, we note that
\[
J_{i}\bigl(t^{\ast},T,M\bigr)=O(1)\qquad \mbox{as } M\rightarrow
\infty\mbox{ for all }i=1,2,3,4
\]
and that the definition of $z=\hat{x}(c/M\bar{\sigma}^{2})^{1/2}/2$ implies
\begin{eqnarray}
&& \int_{0}^{T-t^{\ast}}r(\varepsilon,\hat{x},M,t)\,dt
=O\bigl( z^{3}+z^{2}\hat {x}+z\hat{x}^{2}+\varepsilon
z \bigr),\nonumber\\
\eqntext{\mbox{ uniformly in }T<\infty  \mbox{ as } M\rightarrow\infty.}
\end{eqnarray}

\setcounter{footnote}{3}
\begin{rem}\label{R:NonLinearRemark}
The following three lemmas are analogous to Lemmas \mbox{\ref{L:Region_0_y}--\ref{L:Region_y_hatx}}\footnote{Due to page restrictions of the journal and because of the similarity to the proofs for the linear case, Lemmas \ref{L:Region_0_y_NonLinear},
\ref{L:Region_hatx_A_NonLinear} and \ref{L:Region_y_hatx_NonLinear}
are presented here without proof. However,
complete proofs can be found in the extended version on \arxivurl{arXiv:1303.0450}.} from the linear case. The important difference between
the nonlinear and the linear case is that the statements involve approximation
errors, and the statements hold if one confines the linearized dynamics to a
small neighborhood of the equilibrium point as dictated by the sizes of
$t^{\ast}$ and $M$, or equivalently by $t^{\ast}$ and $z$. Due to the natural
scaling of $M$ in terms of $\varepsilon$ as indicated in Section~\ref{S:SimulationResultsLinearProblem}, as $\varepsilon$ gets smaller,
$z$~will get smaller and be confined to a sufficiently small region that the
statements of the lemmas are valid. However, the lemmas below are stated for
$z$ sufficiently small, without referencing to the natural scaling used in the
simulation algorithm.
\end{rem}

\begin{lemma}
\label{L:Region_0_y_NonLinear}
Assume that $(t,x)\in [0,T-t^{\ast}]\times[0,z]$, $\delta\geq\varepsilon$ and $\eta\leq1/2$. Then, for
sufficiently small $z$, we have up to an exponentially negligible term
\[
\bar{\mathcal{G}^{\varepsilon}}\bigl[U^{\delta,\eta},U^{\delta}
\bigr](t,x)\geq-(1-\eta) C_{0} r(\varepsilon,\hat{x},M,t).
\]
\end{lemma}

\begin{lemma}
\label{L:Region_hatx_A_NonLinear}
Assume that $(t,x)\in [0,T-t^{\ast
}]\times[ H,A]$, and assume $\delta\geq\varepsilon$. Define
\[
\eta_{0}(\varepsilon)  \doteq\sup_{x\in[-A,-H]\cup[H,A]}
\frac
{-\varepsilon\sigma^{2}(x)D (  b(x)\sigma^{-2}(x) )  }{-\varepsilon
\sigma^{2}(x)D (  b(x)\sigma^{-2}(x) )  +b^{2}(x)\sigma^{-2}(x)}.
\]
Let $\varepsilon>0$ be sufficiently small such that $\eta_{0}(\varepsilon
)<1/4$, and consider $\eta\in(\eta_{0}(\varepsilon), 1/4)$. Then, for
sufficiently small $z$ and up to exponentially negligible terms,
\[
\bar{\mathcal{G}}^{\varepsilon}\bigl[U^{\delta,\eta},U^{\delta}
\bigr](t,x)\geq0.
\]
\end{lemma}

\begin{lemma}
\label{L:Region_y_hatx_NonLinear}
Assume that $(t,x)\in [0,T-t^{\ast
}]\times [ z,H]$ and that $M\geq5c/\bar{\sigma}^{2}$. Set $\delta
=2\varepsilon$ and $\sigma_{\ast}^{2}=\sup_{x\in [-A,A]}\sigma^{2}(x)$.
Then, for sufficiently small $z$ we have up to an exponentially negligible
term
\begin{eqnarray*}
&& \bar{\mathcal{G}}^{\varepsilon}\bigl[U^{\delta,\eta},U^{\delta}
\bigr](t,x) \\
&&\qquad \geq  \frac{\sigma_{\ast}^{2}}{2} \biggl[ \frac{1}{\bar{\sigma}^{2}} \biggl(
\frac{c^{2}\eta}{2\bar{\sigma}^{2}} \bigl( z-\hat{x}e^{c(t-T)} \bigr) ^{2}-2
\varepsilon c \biggr) +\Gamma(t,z,H,\hat{x},\varepsilon,\eta,T) \biggr] \wedge0
\\
&&\qquad\quad{}-C_{0}(1-\eta)r(\varepsilon,\hat{x},M,t),
\end{eqnarray*}
where
\begin{eqnarray}
&& \Gamma(t,z,H,\hat{x},\varepsilon,\eta,T)\nonumber\\
 &&\label{Eq:defofGamma}\qquad =  \inf_{x\in [ z,H]}
\biggl[ \frac{c\eta}{2\bar{\sigma}^{2}} \bigl( D\bar{F}_{1}(x)-DF_{1}(x)
\bigr) \bigl( z-\hat{x}e^{c(t-T)} \bigr)
\\
&&\nonumber\hspace*{34pt}\qquad\quad{}+\frac{1}{8}\eta \bigl( D\bar{F}_{1}(x)-DF_{1}%
(x)
\bigr) ^{2}+2\varepsilon \biggl( \frac{c}{\bar{\sigma}^{2}}+D \biggl(
\frac{b(x)}{\sigma^{2}(x)} \biggr) \biggr) \biggr]
\end{eqnarray}
and $\sigma^{2}(x)\geq\sigma_{1}^{2}>0$ for all $x\in\mathbb{R}$.
\end{lemma}

The performance bound is then summarized in the following theorem. The proof
of Theorem \ref{T:MainBoundNonLinear} is the same as the proof of Theorem
\ref{T:MainBoundLinear} for the linear case, so it will not be repeated here.

\begin{theorem}
\label{T:MainBoundNonLinear}
Assume $\delta=2\varepsilon$, $\eta\in(\eta
_{0}(\varepsilon),1/4)$, $z^{2}c\eta\geq8\varepsilon\bar{\sigma}^{2}$ and that
$M\geq5c/\bar{\sigma}^{2}$, where $\eta_{0}(\varepsilon)$ is as in Lemma
\ref{L:Region_hatx_A_NonLinear}. Set $\sigma_{\ast}^{2}=\sup_{x\in
[-A,A]}\sigma^{2}(x)$. Let $\bar{u}$ be the control based on the
function $\bar{U}^{\delta}$ defined via (\ref{Eq:NonLinearChangeOfMeasure}),
that is, $\bar{u}(t,x)=-\sigma(x)D\bar{U}^{\delta}(t,x)$. Then, up to an
exponentially negligible term, for $\varepsilon\in(0,\varepsilon_{0})$ such
that $\eta_{0}(\varepsilon_{0})=1/4$ and for $z$ sufficiently small, we have
\begin{eqnarray*}
&& -\varepsilon\log Q^{\varepsilon}(0,0;\bar{u})\\
&&\qquad \geq  2 \biggl[
I_{1}%
\bigl(\varepsilon,\eta,T,t^{\ast},\hat{x},M
\bigr)-(1-\eta)C_{0}\int_{0}^{T-t^{\ast}%
}r(
\varepsilon,\hat{x},M,t)\,dt \biggr] 1_{ \{  T\geq t^{\ast} \}}
\\
&&\qquad\quad{}+2I_{2}(\varepsilon,T)1_{ \{  T<t^{\ast} \} }.
\end{eqnarray*}
Here
\begin{eqnarray*}
&& I_{1}\bigl(\varepsilon,\eta,T,t^{\ast},\hat{x},M\bigr)\\
&&\qquad= (1-
\eta)\bar{U}^{\delta
}(0,0)
\\
&&\qquad\quad{}+\frac{\sigma_{\ast}^{2}}{2}\int_{J} \biggl[ \frac{1}{\bar{\sigma}^{2}%
}
\biggl( \frac{c^{2}\eta}{2\bar{\sigma}^{2}} \bigl( z-\hat{x}e^{c(s-T)}%
 \bigr)
^{2}-2\varepsilon c \biggr) +\Gamma(s,z,H,\hat{x},\varepsilon,\eta,T)
\biggr] \,ds\\
&&\qquad\quad{}-t^{\ast}c^{\ast}\varepsilon,
\end{eqnarray*}
with $J$ the times in $[0,T-t^{\ast}]$ where the integrand is negative,
$\Gamma(s,z,H,\hat{x},\varepsilon,\eta,T)$ as in (\ref{Eq:defofGamma}),
\[
\bar{U}^{\delta}(0,0)\geq\frac{c}{K-\bar{\sigma}^{2}e^{-2cT}}\hat{x}%
^{2}+
\biggl( 2L-\frac{c}{\bar{\sigma}^{2}}\hat{x}^{2} \biggr) -\delta\log3
\]
and
\[
I_{2}(\varepsilon,T)=2L-c^{\ast}T\varepsilon \quad \mbox{and}\quad
c^{\ast
}=\sup_{x\in [-A,A]}\sigma^{2}(x)\bigl\llvert
D\bigl(b(x)/\sigma^{2}
(x)\bigr)\bigr\rrvert >0.
\]
\end{theorem}

The bound of Theorem \ref{T:MainBoundNonLinear} takes a complicated form, but
as in the linear case, the performance does not degrade as $T\rightarrow
\infty$. This was also reflected by the simulation data in Section~\ref{SS:SimulationNonLinearProblem}. Let us now justify this claim.

Notice that by Lemma \ref{L:IntegralOfCorrectionTerm} the term $\int_{0}^{T-t^{\ast}}r(\varepsilon,\hat{x},M,t)\,dt$ is finite, uniformly in $T$.
Next we need to argue, similar to the linear problem, that when $T-s$ is
sufficiently large and $z$ is sufficiently small, the integrand of the second
term in the definition of $I_{1}(\varepsilon,\eta,T,t^{\ast},\hat{x},M)$ is in
fact positive. Let us denote the integrand of the second term by
\[
B(s,z,H,\hat{x},\varepsilon,\eta,T)+\Gamma(s,z,H,\hat{x},\varepsilon,\eta,T),
\]
where $B(s,z,H,\hat{x},\varepsilon,\eta,T)=\frac{1}{\bar{\sigma}^{2}} (
\frac{c^{2}\eta}{2\bar{\sigma}^{2}} (  z-\hat{x}e^{c(s-T)} )
^{2}-2\varepsilon c )  $. The term $\Gamma(s,z,\break H, \hat{x},\varepsilon
,\eta,T)$ is composed by three terms and we shall argue below they are
dominated (even when they are negative), by the second term in the definition
$B(s,z,H,\hat{x},\varepsilon,\eta,T)$, that is, by $2\varepsilon c/\bar{\sigma
}^{2}$ when $z$ is small enough. This means, as in the case of the linear
problem, that when the integral will be finite uniformly in~$T$. Let us now
support the claim just made. It is easy to see that for $x\in [ z,10z]$,
the first term in the definition of $\Gamma$ can be either positive or
negative, but it is of order $\eta z^{3}$. The second term in the definition
of $\Gamma$ is positive, and for $x\in [ z,10z]$, it is of order $\eta
z^{4}$. Finally, the third term in the definition of $\Gamma$ may be positive
or negative, but in either case, it will be of order $\varepsilon z$ for
$x\in [ z,10z]$. Therefore, for $z$ sufficiently small, $\Gamma$ is
dominated by the second term in the definition $B$, that is, by $2\varepsilon
c/\bar{\sigma}^{2}$. Hence, the argument that was used for the linear problem
in order to show that the integrand of the second term in the definition of
$I_{1}(\varepsilon,\eta,T,t^{\ast},\hat{x},M)$ is in fact positive when $T-s$
is large enough, allows us to reach the same conclusion here as well, given
that $z$ is chosen sufficiently small.

\begin{appendix}
\section*{Appendix}\label{app}

In this appendix we provide proofs of some auxiliary lemmas used in the main
body of the manuscript.

\begin{pf*}{Proof of Lemma \ref{Lem:mollification}}
Without loss of generality, we can
restrict attention to $n=2$. We have
\begin{eqnarray*}
\partial_{t}U^{\delta}(t,x) &=& \rho_{1}(t,x;\delta)
\partial_{t}\tilde{U}%
_{1}(t,x)+
\rho_{2}(t,x;\delta)\partial_{t}\tilde{U}_{2}(t,x),
\\
DU^{\delta}(t,x) &=& \rho_{1}(t,x;\delta)D\tilde{U}_{1}(t,x)+
\rho_{2}%
(t,x;\delta)D\tilde{U}_{2}(t,x)
\end{eqnarray*}
and
\begin{eqnarray*}
D^{2}U^{\delta}(t,x) & =& \frac{1}{\delta}DU^{\delta}(t,x)^{2}-
\rho _{1}(t,x;\delta) \biggl[ \frac{1}{\delta}D\tilde{U}_{1}(t,x)^{2}-D^{2}
\tilde {U}_{1}(t,x) \biggr]
\\
&&{}-\rho_{2}(t,x;\delta) \biggl[ \frac{1}{\delta}D
\tilde{U}_{2}%
(t,x)^{2}-D^{2}
\tilde{U}_{2}(t,x) \biggr].
\end{eqnarray*}
Omitting function arguments for notational convenience, for $\varepsilon
\leq\delta$,
\begin{eqnarray*}
&& \partial_{t}U^{\delta}+ \biggl[ DU^{\delta}b-
\frac{1}{2}\bigl\llvert \sigma DU^{\delta}\bigr\rrvert ^{2}
\biggr] +\frac{\varepsilon}{2}\alpha D^{2}U^{\delta}
\\
&&\qquad =\rho_{1}\,\partial_{t}\tilde{U}_{1}+
\rho_{2}\,\partial_{t}\tilde{U}_{2}%
+
\rho_{1}D\tilde{U}_{1}b+\rho_{2}D
\tilde{U}_{2}b-\frac{1}{2}\bigl\llvert \sigma (
\rho_{1}D\tilde{U}_{1}+\rho_{2}D
\tilde{U}_{2} ) \bigr\rrvert ^{2}
\\
&&\qquad\quad {}+\frac{\varepsilon}{2}\frac{1}{\delta}\bigl\llvert \sigma (
\rho_{1}D\tilde{U}_{1}+\rho_{2}D
\tilde{U}_{2} ) \bigr\rrvert ^{2}%
-
\frac{\varepsilon}{2}\frac{1}{\delta} \bigl[ \rho_{1}\alpha(D
\tilde{U}%
_{1})^{2}+\rho_{2}\alpha(D
\tilde{U}_{2})^{2} \bigr]
\\
&&\qquad\quad{}+\frac{\varepsilon}{2} \bigl[ \rho_{1}\alpha D^{2}
\tilde{U}_{1}%
+\rho_{2}\alpha D^{2}
\tilde{U}_{2} \bigr]
\\
&&\qquad =\rho_{1} \biggl[ \partial_{t}\tilde{U}_{1}+D
\tilde{U}_{1}b-\frac{1}%
{2}\llvert \sigma D\tilde{U}_{1}
\rrvert ^{2}+\frac{\varepsilon}%
{2}\alpha D^{2}
\tilde{U}_{1} \biggr]
\\
&&\qquad\quad {}+\rho_{2} \biggl[ \partial_{t}\tilde{U}_{2}+D
\tilde{U}_{2}b-\frac{1}%
{2}\llvert \sigma D\tilde{U}_{2}
\rrvert ^{2}+\frac{\varepsilon}%
{2}\alpha D^{2}
\tilde{U}_{2} \biggr]
\\
&&\qquad\quad{}+\frac{1}{2} \biggl( 1-\frac{\varepsilon}{\delta} \biggr) \bigl[
\rho_{1}\llvert \sigma D\tilde{U}_{1}\rrvert ^{2}+
\rho_{2}\llvert \sigma D\tilde{U}_{2}\rrvert ^{2}-
\bigl\llvert \sigma ( \rho_{1}%
D\tilde{U}_{1}+
\rho_{2}D\tilde{U}_{2} ) \bigr\rrvert ^{2} \bigr]
\\
&&\qquad \geq\frac{1}{2} \biggl( 1-\frac{\varepsilon}{\delta} \biggr) \bigl[
\rho_{1}\llvert \sigma D\tilde{U}_{1}\rrvert ^{2}+
\rho_{2}\llvert \sigma D\tilde{U}_{2}\rrvert ^{2}-
\llvert \rho_{1}\sigma D\tilde{U}%
_{1}+
\rho_{2}\sigma D\tilde{U}_{2}\rrvert ^{2} \bigr] \\
&& \qquad\quad{}+
\rho_{1}\gamma _{1}+\rho_{2}\gamma_{2}
\\
&&\qquad \geq\rho_{1}\gamma_{1}+\rho_{2}
\gamma_{2},
\end{eqnarray*}
where the last line is due to the convexity of $f(x)=x^{2}$.
\end{pf*}

\setcounter{proposition}{0}
\begin{lemma}
\label{L:GeneralBound}
Let $U(t,x)$ and $W(t,x)$ be two continuously
differentiable functions from $[0,T]\times\mathbb{R}\rightarrow\mathbb{R}$. Assume
that $b$ and $\sigma$ are Lipschitz continuous. Set $\bar{u}(t,x)=-\sigma
(x)DU(t,x)$, $v\in\mathcal{A}$, and let $\hat{X}^{\varepsilon}(s)$ solve
\begin{eqnarray}
d\hat{X}^{\varepsilon}(s) &=& b\bigl(\hat{X}^{\varepsilon}(s)\bigr)\,ds+\sigma\bigl(
\hat {X}^{\varepsilon}(s)\bigr) \bigl[ \sqrt{\varepsilon}\,dB(s)-\bigl[\bar{u}
\bigl(s,\hat {X}^{\varepsilon}(s)\bigr)-v(s)\bigr]\,ds \bigr],\nonumber\\
\eqntext{\hat{X}^{\varepsilon}(0)=y.}
\end{eqnarray}
Then for every $\varepsilon>0$, $v\in\mathcal{A}$ and stopping time $\hat
{\tau}^{\varepsilon}\leq T$, we have, with probability~$1$,
\begin{eqnarray*}
&& \int_{0}^{\hat{\tau}^{\varepsilon}} \biggl[ \frac{1}{2}v(s)^{2}-
\bar {u}\bigl(s,\hat{X}^{\varepsilon}(s)\bigr)^{2} \biggr] \,ds
\\
&&\qquad \geq2W(0,y)-2W\bigl(\hat{\tau}^{\varepsilon},\hat{X}^{\varepsilon}%
\bigl(\hat{\tau}^{\varepsilon}\bigr)\bigr)+2\sqrt{\varepsilon}\int
_{0}^{\hat{\tau
}^{\varepsilon}}DW\bigl(s,\hat{X}^{\varepsilon}(s)\bigr)
\sigma\bigl(\hat{X}^{\varepsilon
}(s)\bigr)\,dB(s)
\\
&&\qquad\quad {}+2\int_{0}^{\hat{\tau}^{\varepsilon}}\mathcal{G}^{\varepsilon
}[W]
\bigl(s,\hat{X}^{\varepsilon}(s)\bigr)\,ds\\
&& \qquad\quad{}-\int_{0}^{\hat{\tau}^{\varepsilon}%
}
\bigl\llvert \sigma\bigl(\hat{X}^{\varepsilon}(s)\bigr) \bigl( DW\bigl(s,
\hat{X}^{\varepsilon
}(s)\bigr)-DU\bigl(s,\hat{X}^{\varepsilon}(s)\bigr) \bigr)
\bigr\rrvert ^{2}\,ds.
\end{eqnarray*}
\end{lemma}

\begin{pf}
We make use of the min/max representation
\[
\mathbb{H}(x,p)=\inf_{v}\sup_{u}
\biggl[ p\bigl(b(x)-\sigma(x)u+\sigma(x)v\bigr)-\frac
{1}{2}u^{2}+
\frac{1}{4}v^{2} \biggr].
\]
Assume we use the control $\bar{u}(t,x)=-\sigma(x)DU(t,x)$ for the design of
the scheme and choose $p=DW(t,x)$. Then
\begin{eqnarray*}
&& \inf_{v} \biggl[ DW(t,x) \bigl(b(x)-\sigma(x)\bar{u}(x)+
\sigma(x)v\bigr)-\frac{1}%
{2}\bar{u}(t,x)^{2}+
\frac{1}{4}v^{2} \biggr]
\\
&&\qquad =DW(t,x)b(x)+\sigma^{2}(x)DU(t,x)-2\sigma^{2}(x)DW(t,x)
\\
&&\qquad\quad {}-\frac{1}{2}\bigl\llvert \sigma(x)DU(t,x)\bigr\rrvert ^{2}+
\bigl\llvert \sigma(x)DW(t,x)\bigr\rrvert ^{2}
\\
&&\qquad =DW(t,x)b(x)-\frac{1}{2}\bigl\llvert \sigma(x)DW(t,x)\bigr\rrvert
^{2}-\frac{1}{2}\bigl\llvert \sigma(x) \bigl( DW(t,x)-DU(t,x)
\bigr) \bigr\rrvert ^{2}
\\
&&\qquad =\mathbb{H}\bigl(x,DW(t,x)\bigr)-\frac{1}{2}\bigl\llvert \sigma(x)
\bigl( DW(t,x)-DU(t,x) \bigr) \bigr\rrvert ^{2}.
\end{eqnarray*}

Applying It\^{o}'s formula to $W(s,\hat{X}^{\varepsilon}(s))$ then gives
\begin{eqnarray*}
 && W\bigl(\hat{\tau}^{\varepsilon},\hat{X}^{\varepsilon}\bigl(\hat{
\tau}^{\varepsilon
}\bigr)\bigr)-W(0,y) \\
&&\qquad= \int_{0}^{\hat{\tau}^{\varepsilon}}
\bigl[ \partial_{s}W\bigl(s,\hat {X}^{\varepsilon}(s)\bigr)
\\
&&\hspace*{21pt}\qquad\quad{}+ DW\bigl(s,\hat{X}^{\varepsilon}(s)\bigr) \bigl[ b\bigl(\hat{X}%
^{\varepsilon}(s)
\bigr)-\sigma\bigl(\hat{X}^{\varepsilon}(s)\bigr)\bar{u}\bigl(s,\hat
{X}^{\varepsilon}(s)\bigr)\\
&&\hspace*{137pt}\qquad\qquad\qquad{}+\sigma\bigl(\hat{X}^{\varepsilon}(s)\bigr)v(s) \bigr]
\bigr] \,ds
\\
&& \qquad\quad{}+\int_{0}^{\hat{\tau}^{\varepsilon}}\frac{\varepsilon}{2}\sigma
^{2}\bigl(\hat{X}^{\varepsilon}(s)\bigr)D^{2}W\bigl(s,
\hat{X}^{\varepsilon}(s)\bigr)\,ds\\
&& \qquad\quad{}+\int_{0}^{\hat{\tau}^{\varepsilon}}
\sqrt{\varepsilon}DW\bigl(s,\hat{X}^{\varepsilon
}(s)\bigr)\sigma\bigl(
\hat{X}^{\varepsilon}(s)\bigr)\,dB(s)
\\
&&\qquad \geq \int_{0}^{\hat{\tau}^{\varepsilon}} \biggl[ \frac{1}{2}
\bar{u}\bigl(s,\hat {X}^{\varepsilon}(s)\bigr)^{2}-\frac{1}{4}v(s)^{2}
\biggr] \,ds+\sqrt{\varepsilon }\int_{0}^{\hat{\tau}^{\varepsilon}}DW
\bigl(\hat{X}^{\varepsilon}(s)\bigr)\sigma\bigl(\hat {X}^{\varepsilon}(s)
\bigr)\,dB(s)
\\
&&\qquad\quad{}+\int_{0}^{\hat{\tau}^{\varepsilon}} \biggl[ \partial_{s}W
\bigl(s,\hat {X}^{\varepsilon}(s)\bigr)+\mathbb{H}\bigl(\hat{X}^{\varepsilon}(s),DW
\bigl(s,\hat {X}^{\varepsilon}(s)\bigr)\bigr)\\
&& \hspace*{80pt}\qquad\qquad{}+\frac{\varepsilon}{2}
\sigma^{2}\bigl(\hat{X}^{\varepsilon
}(s)\bigr)D^{2}W\bigl(s,
\hat{X}^{\varepsilon}(s)\bigr) \biggr] \,ds
\\
&&\qquad\quad{}-\frac{1}{2}\int_{0}^{\hat{\tau}^{\varepsilon}}\bigl\llvert
\sigma \bigl(\hat{X}^{\varepsilon}(s)\bigr) \bigl( DW\bigl(s,\hat{X}^{\varepsilon}(s)
\bigr)-DU\bigl(s,\hat {X}^{\varepsilon}(s)\bigr) \bigr) \bigr\rrvert
^{2}\,ds
\\
&&\qquad = \int_{0}^{\hat{\tau}^{\varepsilon}} \biggl[ \frac{1}{2}\bar{u}
\bigl(s,\hat {X}^{\varepsilon}(s)\bigr)^{2}-\frac{1}{4}v(s)^{2}
\biggr] \,ds\\
&&\qquad\quad{}+\sqrt{\varepsilon }\int_{0}^{\hat{\tau}^{\varepsilon}}DW
\bigl(s,\hat{X}^{\varepsilon}(s)\bigr)\sigma \bigl(\hat{X}^{\varepsilon}(s)
\bigr)\,dB(s)
\\
&&\qquad\quad{} +\int_{0}^{\hat{\tau}^{\varepsilon}}\mathcal{G}^{\varepsilon
}[W]
\bigl(s,\hat{X}^{\varepsilon}(s)\bigr)\,ds\\
&&\qquad\quad{}-\frac{1}{2}\int
_{0}^{\hat{\tau
}^{\varepsilon}}\bigl\llvert \sigma\bigl(
\hat{X}^{\varepsilon}(s)\bigr) \bigl( DW\bigl(s,\hat{X}^{\varepsilon}(s)\bigr)-DU
\bigl(s,\hat{X}^{\varepsilon}(s)\bigr) \bigr) \bigr\rrvert ^{2}\,ds.
\end{eqnarray*}

Rearranging this expression completes the proof of the lemma.
\end{pf}
\end{appendix}

%
%
%
%
%





\printaddresses
\end{document}